\documentclass{article}
\usepackage{geometry}
\usepackage[british]{babel}
\usepackage{ae}
\usepackage{latexsym}
\usepackage{amsfonts}
\usepackage{amssymb}
\usepackage{amsthm}
\usepackage{graphicx}
\usepackage{multirow}
\usepackage{color}
\usepackage{youngtab}
\usepackage{young}
\usepackage{amsmath}
\usepackage{amscd}
\usepackage{float}
\usepackage{hyperref}

\newcounter{itemcounter}
\numberwithin{itemcounter}{subsection}

\newtheorem{thm}[itemcounter]{Theorem}
\newtheorem{lem}[itemcounter]{Lemma}
\newtheorem{defi}[itemcounter]{Definition}

\newtheorem{rem}[itemcounter]{Remark}

\newtheorem*{thm*}{Theorem}
\newtheorem*{ack*}{Acknowledgements}

\title{Brou\'{e}'s perfect isometry conjecture holds for the double covers of the symmetric and alternating groups}
\author{Michael Livesey}
\date{}

\begin{document}

\maketitle

\begin{abstract}
In~\cite{brugra2014} O. Brunat and J. Gramain recently proved that any two blocks of double covers of symmetric
groups are Brou\'{e} perfectly isometric provided they have the same weight and sign. They also proved a corresponding
statement
for double covers of alternating groups and Brou\'{e} perfect isometries between double covers of symmetric and
alternating groups
when the blocks have opposite signs. Using both the results and methods of O. Brunat and J. Gramain in this paper we
prove that when the weight of a block of a double cover of a symmetric or alternating group is less than $p$ then the
block is Brou\'{e} perfectly isometric to its Brauer correspondent. This means that Brou\'{e}'s perfect isometry
conjecture holds for the double covers of the symmetric and alternating groups. We also explicitly construct the
characters of these Brauer correspondents which may be of independent interest to the reader.
\end{abstract}

\section{Introduction}

Brou\'{e}'s abelian defect group conjecture postulates that every block with abelian defect is derived equivalent to
its Brauer correspondent. The conjecture was proved for symmetric groups in a combination of two papers by J. Chuang
and R. Kessar~\cite[Theorem 2]{chukes2002} and J. Chuang and R. Rouquier~\cite[Theorem 7.2]{churou2008}. Let $p$ be a
prime and recall that to each $p-$block of a symmetric group is associated a weight $w$. The defect group is abelian if
and only if $w<p$. The first of these two papers proved that for any weight $w<p$ there exists a block of a symmetric
group of weight $w$ that is Morita equivalent to its Brauer correspondent. The second paper proved that any two blocks
of, possible different, symmetric groups with the same weight are derived equivalent.
\newline
\newline
Presently an analogue
of~\cite[Theorem 2]{chukes2002} does not exist for the double covers of symmetric groups. However,
recently O. Brunat and J. Gramain proved an analogue of~\cite[Theorem 7.2]{churou2008} at the
level of characters for the double covers (see~\cite[Theroem 4.15]{brugra2014}). In other words they proved that any
two blocks of double covers of symmetric groups with the same weight and sign are Brou\'{e} perfectly isometric. They
also proved
a corresponding statement for double covers of alternating groups and Brou\'{e} perfect isometries between double covers
of symmetric and alternating groups when the blocks have opposite signs. The main notion they employed in their proof
was that of an MN-structure. This allows one to express character values of a block of a double cover of a symmetric
group in terms of character values of blocks of smaller double covers (originally a result of M. Cabanes,
see~\cite[Thoerem 20]{cabane1988}). The MN-structures of two blocks of the same
weight and sign commute with an isometry between these two blocks and this then allows one to prove that this isometry
is in fact a Brou\'{e} perfect isometry between the two blocks.
\newline
\newline
In this paper we develop an MN-structure for the
Brauer correspondent of a block of a double cover. This involves very explicitly constructing the characters of such a
group. This is a new result (see Theorem~\ref{thm:maincha}). Using this MN-structure we then go on to prove our main
Theorem.

\begin{thm}\label{thm:main}
Let $p$ be an odd prime, $n$ a positive integer and $B$ a $p-$block of $\tilde{S}_n$ or $\tilde{A}_n$ with abelian defect
group. Then there exists a Brou\'{e} perfect isometry between $B$ and its Brauer correspondent.
\end{thm}

In other words, Brou\'{e}'s perfect isometry conjecture holds for the double covers of the symmetric and alternating
groups.
\newline
\newline
We now give a brief description of each section. $\S$\ref{sec:prelim} contains all the results on MN-structures and
generalised perfect isometries we need from~\cite{brugra2014} as well as a couple of lemmas on Clifford theory.
$\S$\ref{sec:comb} describes all the relevant combinatorics needed before we introduce the symmetric groups and their
double covers in $\S$\ref{sec:double}. Clifford algebras are needed to construct the characters of parabolic subgroups
of the double covers. We introduce Clifford algebras in $\S$\ref{sec:cliff} and parabolic subgroups and their
characters in $\S$\ref{sec:parachars}. $\S$\ref{sec:NS} introduces the group $\tilde{N}_p^t\tilde{S}_t$ and very
explicitly constructs its characters. We then go on to prove some more detailed lemmas about the character values
of $\tilde{N}_p^t\tilde{S}_t$ in $\S$\ref{sec:charsum}. The MN-rules used in the MN-structures of $\tilde{S}_n$
and $\tilde{N}_p^t\tilde{S}_t$ are described in $\S$\ref{sec:MN} before we prove Theorem~\ref{thm:main} in
$\S$\ref{sec:main}.

\section{Preliminaries}\label{sec:prelim}

Before we go on to look at the specific groups that are are the subject of this paper we need some general preliminaries
on representation theory.

\subsection{MN-structures and generalised perfect isometries}

In this section we define the notion of MN-structure as introduced by Brunat and Gramain in~\cite{brugra2014}. (For more
details see~\cite[$\S2$]{brugra2014}.) This will be the main tool used in proving Theorem~\ref{thm:main}. Let $G$ be a
finite group. We have the usual inner product
\begin{align*}
\langle\chi,\psi\rangle_G=\frac{1}{|G|}\sum_g\chi(g)\overline{\psi(g)}
\end{align*}
for $\chi,\psi\in\mathbb{C}\operatorname{Irr}(G)$. Let $C$ be a union of conjugacy classes of $G$. We define the map
\begin{align*}
\operatorname{res}_C:\operatorname{Irr}(G)&\to\operatorname{Irr}(G)\\
\operatorname{res}_C(\chi)(g)&=
\begin{cases}
\chi(g)&\text{if }g\in C,\\
0&\text{otherwise.}
\end{cases}
\end{align*}
Consider the equivalence relation on $\operatorname{Irr}(G)$ generated by $\chi\sim\psi$ if
$\langle\operatorname{res}_C(\chi),\operatorname{res}_C(\psi)\rangle_G\neq0$. We define the $C-$blocks of $G$ to be the
equivalence classes of this relation. If $B$ is a union of $C-$blocks then we set
$\operatorname{Irr}(B)^C=\{\operatorname{res}_C(\chi)|\chi\in\operatorname{Irr}(B)\}$
and if $b$ is a $\mathbb{C}-$basis of $\mathbb{C}\operatorname{Irr}(B)^C$ then we denote by $b^\vee$ the dual basis of
$\mathbb{C}\operatorname{Irr}(B)^C$ with respect to $\langle,\rangle_G$. In other words $b^\vee=\{\Phi_\psi\}_{\psi\in
b}$ where $\langle\psi,\Phi_\theta\rangle_G=\delta_{\psi,\theta}$. If, in addition, $b$ is a $\mathbb{Z}-$basis of
$\mathbb{Z}\operatorname{Irr}(B)^C$ we define the integers $d_{\chi,\psi}$ via the equation
\begin{align*}
\operatorname{res}_C(\chi)=\sum_{\psi\in b}d_{\chi,\psi}\psi\text{ for all }\chi\in\operatorname{Irr}(B).
\end{align*}

\begin{lem}\cite[Corollary 2.3]{brugra2014}\label{lem:vee}
With the above notation
\begin{align*}
\Phi_\psi=\sum_{\psi\in b}d_{\chi,\psi}\chi\text{ for all }\psi\in b.
\end{align*}
Moreover,
\begin{align*}
\mathbb{Z}\operatorname{Irr}(B)\cap\mathbb{Z}\operatorname{Irr}(G)^C=\mathbb{Z}b^\vee.
\end{align*}
\end{lem}

We now define the notion of an MN-structure.

\begin{defi}\cite[Definition 2.5]{brugra2014}\label{def:MN}
Let $G$ be a finite group, $C$ a union of conjugacy classes with $1\in C$ and $B$ a union of $C-$blocks. We say $G$ has
an MN-structure with respect to $C$ and $B$ if the following hold:
\begin{enumerate}
\item There exists a union of conjugacy classes $S$ of $G$ with $1\in S$.
\item There exists a subset $A\subset C\times S$ and a bijection $A\to G$ given by $(x_C,x_S)\mapsto x_Cx_S=x_Sx_C$.
Furthermore, $(x_C,1),(1,x_S)\in A$ for all $x_C\in C$ and $x_S\in S$ and also $({}^gx_C,{}^gx_S)\in A$ for all
$(x_C,x_S)\in A$ and $g\in G$.
\item For all $x_S\in S$ there exists $G_{x_S}\leq C_G(x_S)$ such that $C\cap G_{x_S}=\{x_C\in C|(x_C,x_S)\in A\}$.
\item For all $x_S\in S$ there exists a union of $(G_{x_S}\cap C)-$blocks $B_{x_S}$ and a map
$r^{x_S}:\mathbb{C}\operatorname{Irr}(B)\to\operatorname{Irr}(B_{x_S})$ such that
\begin{align*}
r^{x_S}(\chi)(x_C)=\chi(x_Cx_S)\text{ for all }\chi\in B\text{ and }(x_C,x_S)\in A.
\end{align*}
Also $G_1=G$, $B_1=B$ and $r^1=\operatorname{id}$.
\end{enumerate}
\end{defi}

For the rest of this section we assume $G$ is a finite group with an MN-structure and adopt the notation of
Definition~\ref{def:MN}. For each $x_S\in S$ we define
\begin{align*}
d_{x_S}&:\mathbb{C}\operatorname{Irr}(B)\to\mathbb{C}\operatorname{Irr}(B_{x_S})\\
d_{x_S}&(\chi)=\operatorname{res}_C\circ r^{x_S}(\chi)\text{ for all }\chi\in\operatorname{Irr}(B)
\end{align*}
and extend linearly. We also let $e_{x_S}:\mathbb{C}\operatorname{Irr}(B_{x_S})\to\operatorname{Irr}(B)$ be the
adjoint of $d_{x_S}$. In other words
\begin{align*}
\langle e_{x_S}(\chi),\psi\rangle_G=\langle\chi,d_{x_S}(\psi)\rangle_{G_{x_S}}
\text{ for all }\chi\in\mathbb{C}\operatorname{Irr}(B_{x_S})\text{ and }\psi\in\mathbb{C}\operatorname{Irr}(B).
\end{align*}
Now let $S=\cup_{\lambda\in\Lambda}\lambda$ where each $\lambda$ is a conjugacy class of $G$ and pick a set
$\{s_\lambda\}_{\lambda\in\Lambda}$ of reprsentatives for the conjugacy classes in $\Lambda$. For each
$\lambda\in\Lambda$ we define $G_\lambda:=G_{s_\lambda}$, $B_\lambda:=B_{s_\lambda}$, $r^\lambda:=r^{s_\lambda}$,
$d_\lambda:=d_{s_\lambda}$, $e_\lambda:=e_{s_\lambda}$ and
\begin{align*}
l_\lambda:\mathbb{C}\operatorname{Irr}(G_\lambda)^C&\to\mathbb{C}\operatorname{Irr}(G)\\
l_\lambda(\chi)(x)&=
\begin{cases}
\chi(x_C)&\text{if }x_S^G=\lambda\\
0&\text{otherwise.}
\end{cases}
\end{align*}
We also pick a $\mathbb{Z}-$basis $b_\lambda$ for $\mathbb{Z}\operatorname{Irr}(B_\lambda)^{C\cap G_\lambda}$ and set
$b^\vee_\lambda=\{\Phi_\psi|\psi\in b_\lambda\}$ to be the dual basis of $b_\lambda$. Let $G'$ be a finite group, $C'$
a union of conjugacy classes and $B'$ a union of $C'-$blocks of $G'$. Consider the isomorphism
\begin{align*}
\Theta:\mathbb{C}\operatorname{Irr}(B)\otimes\mathbb{C}\operatorname{Irr}(B')&\to\operatorname{End}(\mathbb{C}
\operatorname{Irr}(B),\mathbb{C}\operatorname{Irr}(B'))\\
\chi\otimes\chi'&\mapsto(\psi\mapsto\langle\psi,\overline{\chi}\rangle\chi').
\end{align*}
We note that if $f\in\operatorname{End}(\mathbb{C}\operatorname{Irr}(B),\mathbb{C}\operatorname{Irr}(B'))$ then
\begin{align*}
\hat{f}:=\Theta^{-1}(f)=\sum_{j=1}^r\overline{e_j^\vee}\otimes f(e_j)
\end{align*}
for some basis $(e_1,\dots,e_r)$ of $\mathbb{C}\operatorname{Irr}(B)$.

\begin{thm}\cite[Theorem 2.9]{brugra2014}\label{thm:mainJ}
Let $G$ and $G'$ be two finite groups. Suppose that
\begin{enumerate}
\item The group $G$ (respectively $G'$) has an MN-structure with respect to $C$ and $B$ (respectively $C'$ and $B'$). We
keep the same notation as above, and the objects relative to $G'$ are denoted with a ``prime''.
\item Assume there are subsets $\Lambda_0\subset\Lambda$ and $\Lambda_0'\subset\Lambda'$ such that:
\begin{enumerate}
\item For every $\lambda\in\Lambda\backslash\Lambda_0$ (respectively $\lambda'\in\Lambda'\backslash\Lambda_0'$), we
have $r^\lambda=r'^{\lambda'}=0$.
\item There is a bijection $\sigma:\Lambda_0\to\Lambda_0'$ with $\sigma(\{1\})=\{1\}$ and for $\lambda\in\Lambda_0$, an
isometry
$I_\lambda:\mathbb{C}\operatorname{Irr}(B_\lambda)\to\mathbb{C}\operatorname{Irr}(B_\lambda')$ such that
$I_\lambda\circ r^\lambda=r'^{\sigma(\lambda)}\circ I_{\{1\}}$.
\end{enumerate}
\item For $\lambda\in\Lambda_0$, we have $I_\lambda(\mathbb{C}b^\vee_\lambda)=\mathbb{C}b'^\vee_{\sigma(\lambda)}$.
\end{enumerate}
We write $J_\lambda=I_\lambda|_{\mathbb{C}b^\vee_\lambda}$ and $J^*_\lambda$ for the adjoint of $J_\lambda$.
Then for all $x\in G,x'\in G'$, we have
\begin{align*}
\hat{I}_{\{1\}}(x,x')=\sum_{\lambda\in\Lambda_0}\sum_{\psi\in
b_\lambda}\overline{e_\lambda(\Phi_\psi)(x)}l_{\sigma(\lambda)}'(J^{*-1}_\lambda(\psi))(x').
\end{align*}
\end{thm}

\begin{rem}\cite[Remark 2.10]{brugra2014}\label{rem:part3}
If for all $\lambda\in\Lambda_0$ (respectively $\Lambda_0'$) $G_\lambda$ (respectively $G_\lambda'$) has an MN-structure
with respect to $C\cap G_\lambda$ and $B_\lambda$ (respectively $C'\cap G_\lambda'$ and $B_\lambda'$) and $G_\lambda$
and $G_\lambda'$ satisfy parts (1) and (2) of Theorem~\ref{thm:mainJ} with respect to $C\cap G_\lambda$, $B_\lambda$
and $C'\cap G_\lambda'$, $B_\lambda'$ then part (3) is automatically satisfied.
\end{rem}

We set $\overline{C}=G\backslash C$ and $\overline{C}'=G'\backslash C'$.

\begin{defi}\cite[Proposition 2.14]{brugra2014}
An isometry $I:\mathbb{C}\operatorname{Irr}(B)\to\mathbb{C}\operatorname{Irr}(B')$ is called a generalized perfect
isometry if $I(\mathbb{Z}\operatorname{Irr}(B))=\mathbb{Z}\operatorname{Irr}(B')$ and if $\hat{I}(x,x')=0$ for all
$(x,x')\in(C\times\overline{C}')\cup(\overline{C}\times C')$.
\end{defi}

\begin{lem}\label{lem:gen}
Suppose all the hypotheses of Theorem~\ref{thm:mainJ} hold then $I_{\{1\}}$ is a generalized perfect isometry.
\end{lem}

Now let $C$ (respectively $C'$) be the set of $p-$regular elements of $G$ (respectively $G'$). Also let
$(K,\mathcal{R},k)$ be a splitting $p-$modular system for all the groups considered in the rest of this paper. In
particular for both $G$ and $G'$.

\begin{defi}
We describe $I:\mathbb{C}\operatorname{Irr}(B)\to\mathbb{C}\operatorname{Irr}(B')$ as a Brou\'{e} perfect isometry if
\begin{enumerate}
\item For every $(x,x')\in G\times G'$, $\hat{I}(x,x')\in|C_G(x)|\mathcal{R}\cap|C_{G'}(x')|\mathcal{R}$.
\item $I$ is a generalized perfect isometry.
\end{enumerate}
\end{defi}

\begin{lem}\label{lem:virprocha}
Let $x\in G$ and suppose $A\leq G_{x_S}$ has $p'-$index with $x_C\in A$ and $\Phi\downarrow_A$ a virtual
projective character of $A$ for all $\Phi\in\mathbb{Z}b^\vee$, where $b$ is a $\mathbb{Z}-$basis of
$\mathbb{Z}\operatorname{Irr}(B_{x_S})^{(C\cap G_{x_S})}$. Then $\hat{I}(x,x')\in|C_G(x)|\mathcal{R}$.
\end{lem}

\begin{proof}
See the proof of~\cite[Theorem 2.19]{brugra2014}.
\end{proof}

\subsection{Projective representations and inertial subgroups}\label{sec:projinr}

\begin{thm}(Clifford correspondence)\label{thm:inrnor}
Let $G$ be a finite group, $N$ a normal subgroup of $G$, $\chi$ a character of $N$ and $I_G(\chi)$ the inertial subgroup
of $\chi$ in $G$. If $\psi$ is an irreducible constituent of
$\chi\uparrow^{I_G(\chi)}$ then $\psi\uparrow^G$ is irreducible. Furthermore, every irreducible character of $G$ is of
this form for a unique $G-$orbit of a characters of $N$.
\end{thm}

\begin{proof}
See~\cite[Theorem 6.11]{isaacs1976}.
\end{proof}

\begin{rem}\label{rem:inrnor}
Note that by the transitivity of induction $I_G(\chi)$ can be replaced with any subgroup containing $I_G(\chi)$ in the above theorem.
\end{rem}

\begin{lem}\label{lem:proj}
Let $G$ be a finite group, $H$ a normal subgroup and $\rho$ a representation of $G$ over $\mathbb{C}$. Suppose
$\rho=\rho_1\otimes\rho_2$, where $\rho_1$ and $\rho_2$ are irreducible
projective representations of $G$ with $\rho_2(h)=1$ for all $h\in H$ and $\rho_1\downarrow_H$ an irreducible
representation of $H$, then $\rho$ is irreducible.
\end{lem}

\begin{proof}
Let $M=M_1\otimes M_2$ be the $\mathbb{C}G-$module afforded by $\rho$ with $\rho_1$ affording $M_1$ and $\rho_2$
affording $M_2$. We can view $M_1$ as an irreducible
$\mathbb{C}H-$submodule of $M$ and as $\mathbb{C}H-$modules $M\cong M_1^{\oplus\dim(\rho_2)}$. By Schur's lemma any
irreducible $\mathbb{C}H-$submodule of $M$ is of the form
$M_1\otimes\mathbb{C}m$ for some non-zero $m\in M_2$. As $\rho_2$ is irreducible then any $\mathbb{C}G-$submodule of $M$
containing $M_1\otimes\mathbb{C}m$ must be the whole of $M$.
Hence $\rho$ is irreducible.
\end{proof}

\section{Combinatorics of partitions}\label{sec:comb}

The combinatorics of partition is very important when describing the characters of both the symmetric groups and their
double covers. In this section we introduce all the combinatorics relevant to us in this paper.
\newline
\newline
For $n\in\mathbb{N}_0$ let $\mathcal{P}_n$ be the set of partitions of $n$. If
$\lambda=(\lambda_1\geq\dots\geq\lambda_t>0)\in\mathcal{P}_n$ we define $l(\lambda):=t$ and
$\sigma(\lambda)=(-1)^{n-t}$. For $\lambda\in\mathcal{P}_n$ and $q$ a positive integer we take from~\cite{olsson1993}
the notion of a $q-$hook and $q-$sign $\delta_q(\lambda)$. If $b$ is a $q-$hook of $\lambda$ such that the partition
obtained
by removing $b$ is $\mu$ then we set $c^\lambda_\mu:=b$ and $L(b)$ to be the leg length of $b$. Let $\lambda_{(q)}$,
$(\lambda)^{(q)}$ and $w_q(\lambda)$ be the $q-$core, $q-$quotient and $q-$weight of $\lambda$ respectively. Finally we
set
$M_q(\lambda)$ to be the set of partitions of $(n-q)$ obtained by removing a $q-$hook from $\lambda$.
\newline
\newline
Let $\mathcal{D}_n$ be the set of partitions of $n$ with distinct parts and
\begin{align*}
\mathcal{D}_n^+:=\{\lambda\in\mathcal{D}_n|\sigma(\lambda)=1\},
\mathcal{D}_n^-:=\{\lambda\in\mathcal{D}_n|\sigma(\lambda)=-1\}.
\end{align*}
In addition we define
\begin{align*}
\mathcal{O}_n:=\{\lambda\in\mathcal{P}_n|\text{all parts of }\lambda\text{ are odd}\}.
\end{align*}
For $\lambda\in\mathcal{D}_n$ and $q$ an odd positive integer we again take from~\cite{olsson1993} the notion of a
$q-$bar and
$\overline{q}-$sign $\delta_{\overline{q}}(\lambda)$. If $b$ is a $q-$bar of $\lambda$ such that the partition obtained
by removing $b$ is $\mu$ then we set $\overline{c}^\lambda_\mu=b$ and $L(b)$ to be the leg length of $b$. Let
$\lambda_{(\overline{q})}$, $\lambda^{(\overline{q})}$ and $w_{\overline{q}}(\lambda)$ be the $q-$bar core, $q-$bar
quotient
and $\overline{q}-$weight of $\lambda$ respectively. Finally we set $M_{\overline{q}}(\lambda)$ to be the set of
partitions of $(n-q)$ obtained by removing a $q-$bar from $\lambda$.

\begin{lem}\label{lem:leglength}$ $
\begin{enumerate}
\item Let $\lambda\in\mathcal{P}_n$ and $q$ a positive integer. Then for all $\mu\in M_q(\lambda)$ we have
$(-1)^{L(c^\lambda_\mu)}=\delta_q(\lambda)\delta_q(\mu)$.
\item Let $\lambda\in\mathcal{D}_n$ and $q$ an odd positive integer. Then for all $\mu\in M_{\overline{q}}(\lambda)$ we
have
$(-1)^{L(\overline{c}^\lambda_\mu)}=\delta_{\overline{q}}(\lambda)\delta_{\overline{q}}(\mu)$.
\end{enumerate}
\end{lem}

\begin{proof}$ $
\begin{enumerate}
\item See~\cite[Lemma 3.12]{olsson1993}.
\item See~\cite[Lemma 4.9]{olsson1993}.
\end{enumerate}
\end{proof}

We also wish to define a $\overline{q}-$sign for $\overline{q}-$quotients. If
$\lambda^{(\overline{q})}=(\lambda_0,\dots,\lambda_{(q-1)/2})$ then we define
\begin{align*}
\delta_{\overline{q}}(\lambda^{(\overline{q})}):=\delta_{\overline{q}}
(\lambda_0)\delta_q(\lambda_1)\dots\delta_q(\lambda_{(q-1)/2}).
\end{align*}

For now we give no further comment on this $\overline{q}-$sign for $\overline{q}-$quotients. Its purpose will become
clear when we explicitly describe the Brou\'{e} perfect isometry of Theorem~\ref{thm:main}.

\section{Double covers of symmetric and alternating groups}\label{sec:double}

In this section we introduce double covers of symmetric and alternating groups and gather together some basic facts
about them.
\newline
\newline
Let $n$ be a positive integer. We define the symmetric group on $n$ letters,
\begin{align*}
S_n:=\langle s_1,\dots,s_{n-1}\mid s_j^2=1,s_js_k=s_ks_j\text{ if }|j-k|>1,(s_js_{j+1})^3=1\text{ for }(0\leq j\leq
n-1)\rangle.
\end{align*}
We now label the irreducible chraracters of $S_n$. For more details see~\cite[Theorem 2.1.11]{jamker1981}.

\begin{thm}
The irreducible characters of $S_n$ are labeled by partitions of $n$.
\end{thm}

For $\lambda\vdash n$ we label by $\operatorname{Sym}_\lambda$ the irreducible character of $S_n$ corresponding to
$\lambda$
through the above theorem. If $\lambda=(\lambda_1,\dots,\lambda_t)$ with $\lambda_j\vdash n_j$ then $\lambda$ labels a
character
of $S_{n_1}\times\dots\times S_{n_t}$ that we label $\operatorname{Sym}_\lambda$.
\newline
\newline
We also define the double covers $S^+_n$ and $S^-_n$ of $S_n$,
\begin{align*}
S^+_n:=&\langle z,t_1,\dots,t_{n-1}\mid z^2=1, t_jz=zt_j, t_j^2=1,\\&t_jt_k=zt_kt_j\text{ if
}|j-k|>1,(t_jt_{j+1})^3=z\text{ for }(0\leq j\leq n-1)\rangle,\\
S^-_n:=&\langle z,t_1,\dots,t_{n-1}\mid z^2=1, t_jz=zt_j, t_j^2=z,\\&t_jt_k=zt_kt_j\text{ if
}|j-k|>1,(t_jt_{j+1})^3=1\text{ for }(0\leq j\leq n-1)\rangle.
\end{align*}
From now on we will use $\tilde{S}_n$ to denote either $S^+_n$ or $S^-_n$. Note we have a homomorphism
\begin{align*}
\theta_n:\tilde{S}_n&\to S_n\\
t_j&\mapsto s_j
\end{align*}
with kernel $\{1,z\}$. We define $\tilde{A}_n$ to be $\theta_n^{-1}(A_n)$.
\newline
\newline
The conjugacy classes of $S_n$ are labelled by partitions of $n$ corresponding to the cycle type of the elements in
that conjugacy class. If $\pi\vdash n$ and $C_\pi$ is the corresponding conjugacy class of $S_n$ then
$\tilde{C}_\pi=\theta_n^{-1}(C_\pi)$ is the union of one or two conjugacy classes of $\tilde{S}_n$. If
$\pi=(1^{\pi_1},2^{\pi_2},\dots)$ and $g\in C_\pi$ then $C_{S_n}(g)\cong\prod_j(C_j\wr S_{\pi_j})$ and so
$|C_{S_n}(g)|=\prod_jj^{\pi_j}\pi_j!$. From the above statement about conjugacy classes we can deduce that for
$g\in\tilde{S}_n$
\begin{align*}
|C_{\tilde{S}_n}(g)|=|C_{S_n}(\theta_n(g))|\text{ or }2|C_{S_n}(\theta_n(g))|.
\end{align*}
We therefore have the following lemma.

\begin{lem}\label{lem:dbcent}
Let $g\in\tilde{S}_n$ be in the conjugacy class $\tilde{C}_\pi$ for some $\pi=(1^{\pi_1},2^{\pi_2},\dots)\vdash n$.
Then
\begin{align*}
|C_{\tilde{S}_n}(g)|=\prod_jj^{\pi_j}\pi_j!\text{ or }2\prod_jj^{\pi_j}\pi_j!.
\end{align*}
\end{lem}

If $g\in S_n$ is in the conjugacy class $C_\pi$ with $\pi\in\mathcal{O}_n$. Then $g$ has odd order say $m$. Now
let $h\in\tilde{S}_n$ with $\theta_n(h)=g$ and so $h^m\in\{1,z\}$. By multiplying by $z$ we can assume $h$ has order
$m$ and we use $o(g)$ to denote such an $h$.
\newline
\newline
For any subgroup $H$ of $\tilde{S}_n$ with $z\in H$ we call an irreducible character $\chi$ of $H$ an irreducible spin
character if $z\notin\ker(\chi)$ and denote the set of such
characters by $\operatorname{IrrSp}(H)$. If in addition $H\not\leq\tilde{A}_n$ then an irreducible spin character
$\chi$ of $H$ is referred to as self-associate if $\epsilon.\chi=\chi$, where $\epsilon$ is the sign character of $H$
with kernel $H\cap\tilde{A}_n$. When the embedding of $H$ in an $\tilde{S}_n$ is clear from the context then
$\epsilon$ will always have this meaning. We denote by $\operatorname{IrrSp}^+(H)$ and $\operatorname{IrrSp}^-(H)$ the
set of self-associate and non-self-associate characters of $H$ respectively.
\newline
\newline
If $\chi\in\operatorname{IrrSp}^+(H)$ then $\chi\downarrow_{H\cap\tilde{A}_n}$ is the sum of two distinct irreducible
spin characters of $H\cap\tilde{A}_n$. We label these two
characters $\overline{\chi}^+,\overline{\chi}^-$. If $\chi\in\operatorname{IrrSp}^-(H)$ then
$\chi\downarrow_{H\cap\tilde{A}_n}=(\epsilon.\chi)\downarrow_{H\cap\tilde{A}_n}$
is an irreducible spin character of $H\cap\tilde{A}_n$. We label this character $\overline{\chi}$.
\newline
\newline
Schur proved in~\cite{schur1911} that there is the following labelling of the irreducible spin characters of
$\tilde{S}_n$
and we adopt Schur's labelling for the rest of this paper.

\begin{thm}
The irreducible spin characters of $\tilde{S}_n$ are labelled in the following way. Each $\lambda\in\mathcal{D}^+$
labels an irreducible self-associate spin character of
$\tilde{S}_n$. We denote such a character by $\xi_\lambda$. Each $\lambda\in\mathcal{D}^-$ labels an associate pair of
irreducible spin characters of $\tilde{S}_n$. We denote such a
pair by $\xi^+_\lambda,\xi^-_\lambda$. Furthermore, the above characters form a complete list of irreducible spin
characters of $\tilde{S}_n$.
\end{thm}

Due to the remarks preceeding the above theorem we can also label the irreducible spin characters of $\tilde{A}_n$ and
we
make the choice of labelling as described in~\cite[$\S4$]{brugra2014}.

\begin{thm}\cite{schur1911}\label{thm:values}
Let $\lambda\in\mathcal{D}_n$, $\pi\in\mathcal{P}_n$ and $g\in\tilde{C}_\pi$.
\begin{enumerate}
\item If $\sigma(\lambda)=1$ then $\xi_\lambda(g)\neq0$ only if $\pi\in\mathcal{O}_n$.
\item If $\sigma(\lambda)=-1$ then $\xi^\pm_\lambda(g)\neq0$ only if $\pi\in\mathcal{O}_n$ or $\pi=\lambda$.
Furthermore if $\pi=\lambda$ then
\begin{align*}
\xi_\lambda^\pm(g)=\pm i^{(n-l(\lambda)+1)/2}\frac{\sqrt{(\lambda_1\lambda_2\dots)}}{2}.
\end{align*}
\end{enumerate}
Now suppose $\sigma(\pi)=1$ and so $g\in\tilde{A}_n$.
\begin{enumerate}
\item[3.] If $\sigma(\lambda)=1$ then $\overline{\xi}^+_\lambda(g)\neq\overline{\xi}^-_\lambda(g)$ only if $\pi=\lambda$
and in this case
\begin{align*}
\overline{\xi}^+_\lambda(g)-\overline{\xi}^-_\lambda(g)=\pm i^{(n-l(\lambda))/2}\sqrt{(\lambda_1\lambda_2\dots)}.
\end{align*}
\end{enumerate}
\end{thm}

We conclude this section by discussing the $p-$blocks of $\tilde{S}_n$ and $\tilde{A}_n$, where $p$ is an odd prime.

\begin{thm}
Let $\lambda,\mu\in\mathcal{D}_n$. If $w_{\overline{p}}(\lambda)=0$ and $\sigma(\lambda)=-1$ then
$\xi^+_\lambda$ and $\xi^-_\lambda$ lie in $p-$blocks of $\tilde{S}_n$ on their own, otherwise $\xi^{(\pm)}_\lambda$ and
$\xi^{(\pm)}_\mu$ lie in the same $p-$block if and only if $\lambda_{(\overline{p})}=\mu_{(\overline{p})}$.
\end{thm}

\begin{proof}
See~\cite[Theorem 1.1]{humphr1986}.
\end{proof}

If $\gamma$ is a $p-$bar core then we denote by $\tilde{S}_{n,\gamma}$ the block of $\tilde{S}_n$ corresponding to
$\gamma$ or
by $\tilde{S}_{n,\gamma}^\pm$ the block of $\tilde{S}_n$ containing $\xi_\gamma^\pm$ if $\gamma\in\mathcal{D}^-_n$.
We label the corresponding block idempotent $e_{n,\gamma}^{(\pm)}$.

\begin{thm}
Let $w\in\mathbb{N}_0$ and $\gamma\vdash(n-pw)$ be a $p-$bar core. Consider the subgroup
\begin{align*}
S_{pw}\leq S_{n-pw}\times S_{pw}\leq S_n.
\end{align*}
Any Sylow $p-$subgroup $P$ of $\theta_n^{-1}(S_{pw})\leq\tilde{S}_n$ is a defect group of
$\tilde{S}_{n,\gamma}^{(\pm)}$.
Furthermore,
\begin{align*}
N_{\tilde{S}_n}(P)=\tilde{S}_{n-pw}N_{\tilde{S}_{pw}}(P)
\end{align*}
and if $w>0$ then the idempotent of the Brauer correspondent of $\tilde{S}_{n,\gamma}$ is
\begin{align*}
\begin{cases}
e_{n-pw,\gamma}&\text{if }\sigma(\gamma)=1,\\
e_{n-pw,\gamma}^++e_{n-pw,\gamma}^-&\text{if }\sigma(\gamma)=-1.
\end{cases}
\end{align*}
\end{thm}

\begin{proof}
See~\cite[Theorems A and B and Corrollary 26]{cabane1988}.
\end{proof}

Using clifford theory we can deduce the corresponding theorems for $\tilde{A}_n$. For proofs see~\cite[Proposition
3.16]{kessar1996}.

\begin{thm}
Let $\lambda,\mu\in\mathcal{D}_n$. If $w_{\overline{p}}(\lambda)=0$ and $\sigma(\lambda)=1$ then
$\overline{\xi}^+_\lambda$
and $\overline{\xi}^-_\lambda$ lie in $p-$blocks of $\tilde{A}_n$ on their own, otherwise
$\overline{\xi}^{(\pm)}_\lambda$
and $\overline{\xi}^{(\pm)}_\mu$ lie in the same $p-$block if and only if
$\lambda_{(\overline{p})}=\mu_{(\overline{p})}$.
\end{thm}

Let $w\in\mathbb{N}_0$ and $\gamma\vdash(n-pw)$ be a $p-$bar core. Then we label the corresponding block(s) of
$\tilde{A}_n$ by $\tilde{A}_{n,\gamma}^{(\pm)}$ and by $\overline{e}^{(\pm)}_\gamma$
the corresponding block idempotent of $\tilde{A}_{n,\gamma}^{(\pm)}$.

\begin{thm}$ $
\begin{enumerate}
\item If $w=0$ and $\sigma(\lambda)=-1$ then $\overline{e}_{n,\gamma}=e^+_{n,\gamma}+e^-_{n,\gamma}$.
\item If $w=0$ and $\sigma(\lambda)=1$ then $e_{n,\gamma}=\overline{e}^+_{n,\gamma}+\overline{e}^-_{n,\gamma}$.
\item If $w>0$ then $e_{n,\gamma}=\overline{e}_{n,\gamma}$.
\end{enumerate}
\end{thm}

\begin{thm}
Any Sylow $p-$subgroup $P$ of $\theta_n^{-1}(S_{pw})\leq\tilde{S}_n$ is a defect group of
$\tilde{A}_{n,\gamma}^{(\pm)}$. Furthermore,
\begin{align*}
N_{\tilde{A}_n}(P)=\tilde{S}_{n-pw}N_{\tilde{S}_{pw}}(P)\cap\tilde{A}_n
\end{align*}
and the idempotent of the Brauer correspondent of $\tilde{A}_{n,\gamma}^{(\pm)}$ is
\begin{align*}
\begin{cases}
\overline{e}_{n-pw,\gamma}&\text{if }\sigma(\gamma)=-1,\\
\overline{e}_{n-pw,\gamma}^++\overline{e}_{n-pw,\gamma}^-&\text{if }\sigma(\gamma)=1.
\end{cases}
\end{align*}
\end{thm}

\section{Clifford algebras}\label{sec:cliff}

For any positive integer $n$ we define the Clifford algebra $\mathcal{C}_n$ to be the $\mathbb{C}-$algebra generated by
$e_1,\dots,e_n$ subject to the relations $e_j^2=1$ and
$e_je_k=-e_ke_j$ if $j\neq k$. In particular $\mathcal{C}_n$ has $\mathbb{C}-$basis $\{e_I\}_I$ where $I$ runs over the
subsets of $[n]:=\{1,\dots,n\}$ and
$e_I:=e_{j_1}.\dots.e_{j_t}$ if $I=\{j_1<\dots<j_t\}$ and $e_\varnothing:=1$. We also want to define the special Clifford
algebra $\mathcal{C}_n^+$. We define this to be the subalgebra
of $\mathcal{C}_n$ with $\mathbb{C}-$basis $\{e_I\}_I$ where $I$ runs over the subsets of $[n]$ of even size.

\begin{lem}\label{lem:cliffrep}
$ $
\begin{enumerate}
\item If $n=2k$ is even then, as $\mathbb{C}-$algebras, $\mathcal{C}_n\cong M_{2^k}(\mathbb{C})$, the algebra of
$2^k\times 2^k$ matrices over $\mathbb{C}$. The character of this
representation is given by $\chi_n^{\mathcal{C}}:\sum_Ic_Ie_I\mapsto 2^kc_\varnothing$.
\item If $n=2k+1$ is odd then, as $\mathbb{C}-$algebras, $\mathcal{C}_n\cong M_{2^k}(\mathbb{C})\oplus
M_{2^k}(\mathbb{C})$. Futhermore, this isomorphism induces two irreducible
characters of $\mathcal{C}_n$ given by $\chi_n^{\mathcal{C}\pm}:\sum_Ic_Ie_I\mapsto 2^kc_\varnothing\pm(2i)^kc_{[n]}$.
\item If $n=2k$ is even then as $\mathbb{C}-$algebras, $\mathcal{C}_n^+\cong M_{2^k}(\mathbb{C})\oplus
M_{2^k}(\mathbb{C})$. Futhermore, this isomorphism induces two irreducible characters of $\mathcal{C}_n$ given by
$\chi_n^{\mathcal{C}\pm}:\sum_Ic_Ie_I\mapsto 2^{k-1}c_\varnothing\pm i(2i)^{k-1}c_{[n]}$.
\item If $n=2k+1$ is odd then, as $\mathbb{C}-$algebras, $\mathcal{C}_n^+\cong M_{2^k}(\mathbb{C})$ and the character
associated to this isomorphism is given by $\chi_n^{\mathcal{C}}:\sum_Ic_Ie_I\mapsto 2^kc_\varnothing$.
\end{enumerate}
\end{lem}

\begin{proof}
$ $
\begin{enumerate}
\item[1,2.] See~\cite[$\S3$]{stembr1989}.
\item[3,4.] We get the desired results through the isomorphism
\begin{align*}
\mathcal{C}_{n-1}&\to\mathcal{C}_n^+\\
e_j&\mapsto ie_je_n.
\end{align*}
\end{enumerate}
\end{proof}

\begin{lem}\label{lem:embed}$ $
\begin{enumerate}
\item One can realise $S^+_n$ and $S^-_n$ as subgroups of $\mathcal{C}_n$ via
\begin{align*}
\phi_n^+:S^+_n&\to\mathcal{C}_n& \text{ and }&&\phi_n^-:S^-_n&\to\mathcal{C}_n\\
t_j&\mapsto\frac{1}{\sqrt{2}}(e_j+e_{j+1})&\text{  }&&t_j&\mapsto\frac{i}{\sqrt{2}}(e_j+e_{j+1}).
\end{align*}
\item For $\phi_n=\phi_n^\pm$ we have
\begin{align*}
\phi_n(t_j)e_j\phi_n(t_j^{-1})=e_{j+1}
\end{align*}
\end{enumerate}
\end{lem}

\begin{proof}$ $
\begin{enumerate}
\item See~\cite[$\S3$]{stembr1989}.
\item We treat only the case of $\phi_n^+$, as the case of $\phi_n^-$ is similar. Note that $t_j^{-1}=t_j$ and so
\begin{align*}
\phi_n(t_j)e_j\phi_n(t_j^{-1})=\frac{1}{\sqrt{2}}(e_j+e_{j+1})e_j\frac{1}{\sqrt{2}}(e_j+e_{j+1})=
\frac{1}{2}(e_j+e_{j+1}+e_{j+1}-e_j)=e_{j+1}.
\end{align*}
\end{enumerate}
\end{proof}

\section{Characters of parabolic subgroups}\label{sec:parachars}

Let $H\leq\tilde{S}_n$ with $z\in H$ and $H\not\leq\tilde{A}_n$. If $V$ is the representation space of a self-associate
irreducible spin representation $\rho$ of $H$ then there exists
some $S:V\to V$ such that $S\rho(h)S^{-1}=\epsilon(h)\rho(h)$ for all $h\in H$. Note that by Schur's lemma
$S^2=\lambda\operatorname{Id}_V$ for some $\lambda\in\mathbb{C}$. We can
always scale $S$ so that $\lambda=1$. In this case we call $S$ an associator and $S$ is uniquely determined up to
multiplication by $\pm1$.
\newline
\newline
Let $t\in\mathbb{N}$, $n_1,\dots,n_t\in\mathbb{N}$ and $n:=\sum_jn_j$. Suppose we have a family of groups $H_j$ with
$H_j\leq\tilde{S}_{n_j}$ and $H_j\not\leq\tilde{A}_{n_j}$ for
$(1\leq j\leq t)$. Now we explicitly construct the irreducible spin characters of
\begin{align*}
H_1\dots H_t\leq\tilde{S}_{n_1}\dots\tilde{S}_{n_t}\leq\tilde{S}_n.
\end{align*}
For this construction we
follow~\cite[$\S4$]{stembr1989}, where it is assumed each $H_j=\tilde{S}_{n_j}$. However, the proofs run through with
no extra complications in this more general setup.
\newline
\newline
Let $\chi_j\in\operatorname{IrrSp}(H_j)$ and assume without loss of generality that $\chi_1,\dots,\chi_r$ are all
self-associate and $\chi_{r+1},\dots,\chi_t$ are all
non-self-associate. Let $V_j$ be the representation space of the character $\chi_j$ and $S_j$ an associator map of $V_j$
for $(1\leq j\leq r)$ and $\operatorname{Id}_{V_j}$
for $(r+1\leq j\leq t)$. Finally let $V$ be an irreducible representation space of $\mathcal{C}_{t-r}$. If $x\in H_j$
then we define the action of $x$ on
$V\otimes V_1\otimes\dots\otimes V_t$ as follows:
\begin{align}\label{des:parachars}
x.(v\otimes v_1\otimes\dots\otimes v_t)=
\begin{cases}
e_j.v\otimes A_1.v_1\otimes\dots\otimes A_t.v_t&\text{if }j>r\text{ and }x\in H_j\backslash\tilde{A}_{n_j},\\
v\otimes A_1.v_1\otimes\dots\otimes A_t.v_t&\text{otherwise},
\end{cases}
\end{align}
where
\begin{align*}
(A_1,\dots,A_t)=
\begin{cases}
(1,\dots,1,x,1,\dots,1)&\text{if }x\in H_j\cap\tilde{A}_{n_j},\\
(S_1,\dots,S_{j-1},x,1,\dots,1)&\text{if }x\in H_j\backslash\tilde{A}_{n_j}.
\end{cases}
\end{align*}
This defines a representation of $H_1\dots H_t$ and it is non-self-associate if and only if $(t-r)$ is odd. Furthermore,
one can obtain its associate by replacing $V$ with the other irreducible representation of $\mathcal{C}_{t-r}$ or by
replacing $\chi_j$ with $\epsilon.\chi_j$ for some $j>r$.
\newline
\newline
Now suppose that for each $j$ with $(1\leq j\leq t)$ we have sets $\Lambda_j^+$ and $\Lambda_j^-$ with
\begin{align*}
\operatorname{IrrSp}^+(H_j)=\{\chi_\lambda|\lambda\in\Lambda_j^+\},
\operatorname{IrrSp}^-(H_j)=\{\chi^\pm_\lambda|\lambda\in\Lambda_j^-\}.
\end{align*}
Now let $\lambda_j\in\Lambda_j:=\Lambda_j^+\cup\Lambda_j^-$ for $(1\leq j\leq t)$. Then we denote by
$\chi_{(\lambda_1,\dots,\lambda_t)}$ the corresponding character of $H_1\dots H_t$
if $\lambda_j\in\Lambda_j^-$ for an even number of $j$'s or by $\chi^\pm_{(\lambda_1,\dots,\lambda_t)}$ if
$\lambda_j\in\Lambda_j^-$ for an odd number of $j$'s.

\begin{thm}
A complete set of irreducible spin characters of $H_1\dots H_t$ are given by
\begin{align*}
\{\chi_{(\lambda_1,\dots,\lambda_t)}|\lambda_j\in\Lambda_j^-\text{ for an even number of }j\text{'s}\}\cup
\{\chi^\pm_{(\lambda_1,\dots,\lambda_t)}|\lambda_j\in\Lambda_j^-\text{ for an odd number of }j\text{'s}\}.
\end{align*}
\end{thm}

\begin{proof}
See~\cite[4.3]{stembr1989}.
\end{proof}

We can explicitly write down the irreducible spin characters when $t=2$.

\begin{lem}\label{lem:t=2}
Let $s_j\in H_j$ for $j=1,2$.
\begin{enumerate}
\item If $\lambda_1\in\Lambda^+_1,\lambda_2\in\Lambda^+_2$ then
\begin{align*}
\chi_{(\lambda_1,\lambda_2)}(s_1s_2)=\chi_{\lambda_1}(s_1)\chi_{\lambda_2}(s_2).
\end{align*}
\item If $\lambda_1\in\Lambda^+_1,\lambda_2\in\Lambda^-_2$ then
\begin{align*}
\chi_{(\lambda_1,\lambda_2)}^\pm(s_1s_2)=
\begin{cases}
\overline{\chi}^\pm_{\lambda_1}(s_1)\chi^\pm_{\lambda_2}(s_2)+\overline{\chi}^\mp_{\lambda_1}(s_1)\chi^\mp_{\lambda_2}
(s_2)&\text{if }s_1\in\tilde{A}_{n_1},\\
0&\text{otherwise.}
\end{cases}
\end{align*}
\item If If $\lambda_1\in\Lambda^-_1,\lambda_2\in\Lambda^-_2$ then
\begin{align*}
\chi_{(\lambda_1,\lambda_2)}(s_1s_2)=
\begin{cases}
\chi^+_{\lambda_1}(s_1)\overline{\chi}_{\lambda_2}(s_2)+\chi^-_{\lambda_1}(s_1)\overline{\chi}_{\lambda_2}(s_2)&\text{if
}s_2\in\tilde{A}_{n_2},\\
0&\text{otherwise.}
\end{cases}
\end{align*}
\end{enumerate}
\end{lem}

Note that in (2) we are also setting up a labelling of characters.

\begin{proof}
See~\cite[4.2]{stembr1989}.
\end{proof}

We would also like to describe the irreducible spin characters of $H_1.H_2\cap\tilde{A}_n$.

\begin{lem}\label{lem:t=2A}
Let $s_j\in H_j$ for $j=1,2$ with $s_1s_2\in H_1.H_2\cap\tilde{A}_n$.
\begin{enumerate}
\item If $\lambda_1\in\Lambda^+_1,\lambda_2\in\Lambda^+_2$ then
\begin{align*}
\overline{\chi}_{(\lambda_1,\lambda_2)}^\pm(s_1s_2)=
\begin{cases}
\overline{\chi}^+_{\lambda_1}(s_1)\overline{\chi}^\pm_{\lambda_2}(s_2)+\overline{\chi}^-_{\lambda_1}(s_1)\overline{\chi}
^\mp_{\lambda_2}(s_2)&\text{if }
s_1\in\tilde{A}_{n_1}\text{ and }s_2\in\tilde{A}_{n_2},\\
0&\text{otherwise.}
\end{cases}
\end{align*}
\item If $\lambda_1\in\Lambda^+,\lambda_2\in\Lambda^-$ then
\begin{align*}
\overline{\chi}_{(\lambda_1,\lambda_2)}(s_1s_2)=
\begin{cases}
\chi_{\lambda_1}(s_1)\chi^\pm_{\lambda_2}(s_2)&\text{if }s_1\in\tilde{A}_{n_1}\text{ and }s_2\in\tilde{A}_{n_2},\\
0&\text{otherwise.}
\end{cases}
\end{align*}
\item If $\lambda_1\in\Lambda^-,\lambda_2\in\Lambda^-$ then
\begin{align*}
\overline{\chi}_{(\lambda_1,\lambda_2)}^\pm(s_1s_2)=
\begin{cases}
\chi^+_{\lambda_1}(s_1)\chi^+_{\lambda_2}(s_2)&\text{if }s_1\in\tilde{A}_{n_1}\text{ and }s_2\in\tilde{A}_{n_2},\\
\pm i\chi^+_{\lambda_1}(s_1)\chi^+_{\lambda_2}(s_2)&\text{otherwise.}
\end{cases}
\end{align*}
\end{enumerate}
\end{lem}

Again note that in 1 and 3 we are setting up a labelling of characters.

\begin{proof}
We can do this by directly decomposing the vector space in (\ref{des:parachars}).

\begin{enumerate}
\item Let $V_j\cong V_j^+\oplus V_j^-$ be a decomposition of $V_j$ as a $\mathbb{C}(H_j\cap\tilde{A}_{n_j})-$module
for $j=1,2$. Then
\begin{align*}
V_1\otimes V_2\cong(V_1^+\otimes V_2^+\oplus V_1^-\otimes V_2^-)\oplus(V_1^+\otimes V_2^-\oplus V_1^-\otimes V_2^+)
\end{align*}
is a decomposition of $V_1\otimes V_2$ as a $\mathbb{C}(H_1.H_2\cap\tilde{A}_n)-$module. The result follows.
\item If $\lambda_1\in\Lambda^+,\lambda_2\in\Lambda^-$ then
$\chi_{(\lambda_1,\lambda_2)}^\pm\downarrow_{H_1.H_2\cap\tilde{A}_n}$ is irreducible.
\item If $\lambda_1\in\Lambda^-,\lambda_2\in\Lambda^-$ then let $V\cong V^+\oplus V^-$ be a decomposition of $V$ in
(\ref{des:parachars}) as a
$\mathcal{C}_2^+-$module. Then
\begin{align*}
V\otimes V_1\otimes V_2\cong(V^+\otimes V_1\otimes V_2)\oplus(V^+\otimes V_1\otimes V_2)
\end{align*}
is a decomposition of $V\otimes V_1\otimes V_2$ as a $\mathbb{C}(H_1.H_2\cap\tilde{A}_n)-$module. The result follows.
\end{enumerate}
\end{proof}

\section{\texorpdfstring{$\tilde{N}_p^t\tilde{S}_t$}{TEXT} and its characters}\label{sec:NS}

We now introduce the group $\tilde{N}_p^t\tilde{S}_t$. The study of this group, its characters and its MN-structure will
form the bulk of the rest of the paper. We will eventually use this group to construct intermediate Brou\'{e} perfect
isometries in order to prove Theorem~\ref{thm:main}.

\subsection{Introduction to \texorpdfstring{$\tilde{N}_p^t\tilde{S}_t$}{TEXT}}\label{subsec:NS}

Let $p$ be an odd prime and consider the subgroup $N_p:=N_{S_p}(C_p)\leq S_p$ where $C_p$ is generated by a $p-$cycle.
Note that $N_p\cong C_p\rtimes C_{p-1}$ where $C_{p-1}$ acts as the full automorphism group of $C_p$. We denote by
$\tilde{N}_p$ the subgroup $\theta_p^{-1}(N_p)$ of $\tilde{S}_p$. Note that $\tilde{N}_p\cong
C_p\rtimes\theta_p^{-1}(C_{p-1})$.
\newline
\newline
Consider the subgroup $N_p^t\leq S_p^t\leq S_{pt}$ for some positive integer $t$. We denote by $\tilde{N}_p^t$ the
subgroup
$\theta_{pt}^{-1}(N_p^t)\leq\theta_{pt}^{-1}(S_p^t)\leq\tilde{S}_{pt}$ and by $\tilde{N}_p^t\tilde{S}_t$ the subgroup
$\theta_{pt}^{-1}(N_p^t\rtimes S_t)$ of $\theta_{pt}^{-1}(S_p^t\rtimes S_t)\leq\tilde{S}_{pt}$, where $S_t$ acts by
permuting the factors of $N_p^t$. We will constantly view $\tilde{N}_p^t\tilde{S}_t$ as a subgroup of $\tilde{S}_{pt}$.
By~\cite[Lemma 3.5]{micols1990} we have that $\theta_{pt}^{-1}(S_t)\cong S^\pm_t$. In either
case we will denote $\theta_{pt}^{-1}(S_t)$ by $\tilde{S}_t$. For any pair of subgroups $A\leq\tilde{N}_p^t$,
$H\leq\tilde{S}_t$ we denote by $AH$ the subgroup of $\tilde{N}_p^t\tilde{S}_t$ generated by $A$ and $H$.
\newline
\newline
Let $(1\leq j<l\leq t)$. We denote by $[j\to l]$ the isomorphism between the $j^{\operatorname{th}}$ and
$l^{\operatorname{th}}$ factor of $\tilde{N}_p^t$ given by conjugating by some $s\in\tilde{S}_t$, where
$\theta_t(s)=(j,j+1,\dots,l-1,l)$. Note that by considering the case
$l=j+1$ and induction, one has
\begin{align*}
[l\to m]\circ[j\to l]=[j\to m]\text{ for }(1\leq j<l<m\leq t).
\end{align*}
Now for some positive integer $q\leq t$ consider
the subgroup of $\tilde{N}_p^t\tilde{S}_t$ that fixes $\{1,\dots,pq\}$ where $\tilde{N}_p^t\tilde{S}_t$ is viewed as a
subgroup of $\tilde{S}_{pt}$. We denote this subgroup by $\tilde{N}_p^{t-q}\tilde{S}_{t-q}[q]$ and set
\begin{align*}
\tilde{N}_p^{t-q}[q]&:=\tilde{N}_p^{t-q}\tilde{S}_{t-q}[q]\cap\tilde{N}_p^t,\\
\tilde{S}_{t-q}[q]&:=\tilde{N}_p^{t-q}\tilde{S}_{t-q}[q]\cap\tilde{S}_t.
\end{align*}

\begin{lem}\label{lem:shift}
There exists an isomorphism $[q]$ between $\tilde{N}_p^{t-q}\tilde{S}_{t-q}$ and $\tilde{N}_p^{t-q}\tilde{S}_{t-q}[q]$
given by
\begin{align*}
[q]:\tilde{N}_p^{t-q}\tilde{S}_{t-q}&\to\tilde{N}_p^{t-q}\tilde{S}_{t-q}[q]\\
x&\mapsto[j\to j+q](x)&&\text{ for }x\text{ in the }j^{\operatorname{th}}\text{ factor of }\tilde{N}_p^{t-q}\text{ for
}(1\leq j\leq t-q)\\
t_j&\mapsto t_{j+q}&&\text{ for }(1\leq j\leq t-q).
\end{align*}
\end{lem}

\begin{proof}
It is easy to see that $[q]$ induces an isomorphism between $\tilde{N}_p^{t-q}$ and $\tilde{N}_p^{t-q}[q]$ and also
between $\tilde{S}_{t-q}$ and $\tilde{S}_{t-q}[q]$. Let $(1\leq j,l\leq t)$ and $x$ be in the $j^{\operatorname{th}}$
factor of $\tilde{N}_p^{t-q}$. Then if $l\neq j,j+1$,
\begin{align*}
[q](t_lxt_l^{-1})=[q](z^{(1-\epsilon(x))/2}x)=s(z^{(1-\epsilon(x))/2}x)s^{-1},
\end{align*}
where $\theta_t(s)=(j,j+1,\dots,j+q-1,j+q)$. On the other hand
\begin{align*}
[q](t_l)[q](x)[q](t_l^{-1})=t_{l+q}sxs^{-1}t_{l+q}^{-1}=z^{(1-\epsilon(x))/2}sxs^{-1}.
\end{align*}
Now if $l=j$ then
\begin{align*}
[q](t_lxt_l^{-1})=s't_lxt_ls'^{-1},
\end{align*}
where $\theta_t(s')=(j+1,\dots,j+q-1,j+q+1)$. Therefore $\theta_t(s't_l)=(j,j+1,\dots,j+q,j+q+1)$
\begin{align*}
[q](t_l)[q](x)[q](t_l^{-1})=t_{l+q}sxs^{-1}t_{l+q}^{-1}=(t_{l+q}s)x(t_{l+q}s)^{-1}
\end{align*}
and $\theta_t(t_{l+q}s)=(j,j+1,\dots,j+q,j+q+1)$.
\end{proof}

Let $y_0$ be a generator of $C_p$ and $y_1$ a generator of $C_{p-1}$ then
\begin{align*}
y_0,y_1,y_2:=y_1^2,\dots,y_{p-1}:=y_1^{p-1}
\end{align*}
is a complete list of representatives of conjugacy classes of $N_p$. We write
\begin{align*}
N_p\wr S_t=((N_p)_1\times\dots\times(N_p)_t)\rtimes S_t.
\end{align*}
Let $n_1,\dots,n_l\leq t$ be distinct positive integers and $x_j\in(N_p)_{n_j}$ for $(1\leq j\leq l)$. We denote by
\begin{align*}
x:=((x_1,\dots,x_l);(n_1,\dots,n_l))
\end{align*}
the element $(x_1,\dots,x_l).(n_1,\dots,n_l)$ of $N_p\wr S_t$, where
$(x_1,\dots,x_l)$ is understood to be in $(N_p)_{n_1}\times\dots\times(N_p)_{n_l}$ and $(n_1,\dots,n_l)$ is the usual
notation for a cycle in $S_t$. We describe $x$ as a cycle of length $l$ in $N_p\wr S_t$ and associate to it the element
$f(x):=(x_1\cdot\ldots\cdot x_l)\in N_p$. Now let $g=\prod x\in N_p\wr S_t$ be a product of disjoint cycles of
$N_p\wr S_t$, where disjoint means the corresponding elements of $S_t$ are disjoint cycles. We associate to $g$ the
multipartition
$\pi=(\pi_0,\dots,\pi_{p-1})$ where $\pi_j$ has as its parts the lengths of the cycles $x$ with $f(x)$ conjugate to
$y_j$ in $N_p$.
We say $g$ is of type $\pi$. By~\cite[4.2.8]{jamker1981} the type of an element in $N_p\wr S_t$ completely determines
its conjugacy
class.
\newline
\newline
Now let $x_j\in\tilde{N}_p$ for $(1\leq j\leq l)$ and $\tau\in\tilde{S}_t$ with $\theta_t(\tau)=(n_1,\dots,n_l)$. We
denote by
\begin{align*}
x:=((x_1,\dots,x_l);\tau)\in\tilde{N}_p^t\tilde{S}_t
\end{align*}
the element $(x_1\dots x_l).\tau$ of $\tilde{N}_p^t\tilde{S}_t$
where it is
understood that $x_j\in\theta^{-1}_{pt}((N_p)_{n_j})\cong\tilde{N}_p$ for $(1\leq j\leq l)$. We describe
$g=\prod x\in\tilde{N}_p^t\tilde{S}_t$ as a product of disjoint cycles in $\tilde{N}_p^t\tilde{S}_t$ if
$\theta_{pt}(g)=\prod\theta_{pt}(x)$ is a product of disjoint cycles in $N_p\wr S_t$. We also say $g$ is of type $\pi$
if
$\theta_{pt}(g)$ is of type $\pi$. Elements of a given type form either one or two conjugacy classes of
$\tilde{N}_p^t\tilde{S}_t$.
\newline
\newline
If $x\in N_p\wr S_t$ then $x$ is conjugate in $N_p\wr S_t$ to an element with disjoint cycle decomposition
\begin{align*}
\prod_j((x_j,1,\dots,1);\tau_j).
\end{align*}
Then the order of $((x_j,1,\dots,1);\tau_j)$ is $|\tau_j|\operatorname{ord}(x_j)$ and so
\begin{align*}
\operatorname{ord}(x)=\operatorname{ord}(\prod_j((x_j,1,\dots,1);\tau_j))=\operatorname{lcm}_j(|\tau_j|\operatorname{ord
}(x_j)).
\end{align*}
Now for each $\tau_j=(n_{j,1},\dots,n_{j,{l_j}})$ set
\begin{align*}
C_j:=\{(g,\dots,g)|g\in C_{N_p}(x_j)\}\leq N_p^{l_j}.
\end{align*}
Then
\begin{align*}
C_{N_p\wr S_t}(x)\cap N_p^t\cong\prod_jC_j.
\end{align*}
Now let $x\in\tilde{N}_p^t\tilde{S}_t$. It is clear that
\begin{align*}
\operatorname{ord}(x)=\operatorname{ord}(\theta_{pt}(x))\text{ or }2\operatorname{ord}(\theta_{pt}(x)).
\end{align*}
Also, as in $\S$\ref{sec:double}, we have
\begin{align*}
|C_{\tilde{N}_p^t\tilde{S}_t}(x)|=|C_{N_p\wr S_t}(\theta_{pt}(x))|\text{ or }2|C_{N_p\wr S_t}(\theta_{pt}(x))|.
\end{align*}
Bringing all the above together we have the following lemma: 

\begin{lem}\label{lem:ordcent}
Let $x\in\tilde{N}_p^t\tilde{S}_t$ have type $\pi$ with $t<p$. Then
\begin{enumerate}
\item $x$ is $p-$regular if and only if $\pi_0=\varnothing$,
\item $|C_{\tilde{N}_p^t\tilde{S}_t}(x)|_p=p^{l(\pi_0)+l(\pi_{p-1})}$.
\end{enumerate}
\end{lem}

\begin{proof}$ $
\begin{enumerate}
\item This follows from the above comments.
\item This follows from the above comments, that $C_p$ is self-centralising and that $p\nmid[N_p\wr S_t:N_p^t]$.
\end{enumerate}
\end{proof}

\subsection{Characters of \texorpdfstring{$\tilde{N}_p^t\tilde{S}_t$}{TEXT}}

There are $p$ irreducible spin characters of $\tilde{N}_p$. The $(p-1)$ linear spin characters are given by inflating
the
$(p-1)$ linear spin characters of $\theta_p^{-1}(C_{p-1})$ and the remaining character is given by inducing any faithful
linear character of $\theta_p^{-1}(C_p)\cong C_p\times\langle z\rangle$. All the linear characters are
non-self-associate and the
non-linear character is self associate. We label the $(p-1)/2$ associate pairs of linear characters $\zeta_j^+$ and
$\zeta_j^-$
for $(1\leq i\leq(p-1)/2)$ and the self-associate character $\zeta_0$.
\newline
\newline
Denote by $\chi_{\boldsymbol{t}}^{(\pm)}$ the character(s) of $\tilde{N}_p^t$ labelled by
\begin{align*}
(\zeta_0,\dots,\zeta_0,\dots,\zeta_{(p-1)/2},\dots,\zeta_{(p-1)/2}),
\end{align*}
where $\boldsymbol{t}=(t_0,\dots,t_{(p-1)/2})$ with each $\zeta_j$ appearing $t_j$ times.
We want to describe the inertial subgroup of $\chi_{\boldsymbol{t}}^{(\pm)}$.

\begin{lem}
\begin{align*}
I_{\tilde{N}_p^t\tilde{S}_t}(\chi_{\boldsymbol{t}}^{(\pm)})/\tilde{N}_p^t\cong
\begin{cases}
S_{t_0}\times S_{t_1}\times\dots\times S_{t_{(p-1)/2}}&\text{if }t-t_0 \text{ is even,}\\
A_{t_0}\times S_{t_1}\times\dots\times S_{t_{(p-1)/2}}&\text{if }t-t_0 \text{ is odd.}
\end{cases}
\end{align*}
\end{lem}

\begin{proof}
See~\cite[Proposition 3.13]{micols1990}.
\end{proof}

Adopting the notation of Lemma~\ref{lem:shift} we set
\begin{align*}
\tilde{N}_p^t\tilde{S}_{\boldsymbol{t}}=\tilde{N}_p^t
(\tilde{S}_{t_0}[t'_0]\tilde{S}_{t_1}[t'_1]\dots\tilde{S}_{t_{(p-3)/2}}[t_{(p-3)/2}']\tilde{S}_{t_{(p-1)/2}}[t_{(p-1)/2}
']),
\end{align*}
where $t_l'=\sum_{j=0}^{l-1}$ for all $(0\leq l\leq(p-1)/2)$.
\newline
\newline
By rearranging factors we can see that every irreducible spin character of $\tilde{N}_p^t$ is conjugate to
$\chi_{\boldsymbol{t}}^{(\pm)}$
for a unique $\boldsymbol{t}$. However, when $(t-t_0)$ is odd we need to determine if $\chi_{\boldsymbol{t}}^+$ is
conjugate to
$\chi_{\boldsymbol{t}}^-$. In this case we have
\begin{align*}
[\tilde{N}_p^t\tilde{S}_{\boldsymbol{t}}:I_{\tilde{N}_p^t\tilde{S}_t}(\chi_{\boldsymbol{t}}^{\pm})]=
\begin{cases}
1&\text{if }t_0\leq1,\\
2&\text{if }t_0>1.
\end{cases}
\end{align*}
Therefore we have an element that swaps $\chi_{\boldsymbol{t}}^+$ and $\chi_{\boldsymbol{t}}^-$ if and only if $t_0>1$
and hence we have the following lemma.

\begin{lem}\label{lem:list}
A complete list of representatives of $\tilde{N}_p^t\tilde{S}_t-$conjugacy classes of irreducible spin characters of
$\tilde{N}_p^t$ are given by
\begin{align*}
\{\chi_{\boldsymbol{t}}|\boldsymbol{t},t-t_0\text{ even}\}\cup\{\chi_{\boldsymbol{t}}^+|\boldsymbol{t},t-t_0\text{ odd,
}t_0>1\}
\cup\{\chi_{\boldsymbol{t}}^\pm|\boldsymbol{t},t-t_0\text{ odd, }t_0\leq1\}.
\end{align*}
\end{lem}

With Theorem~\ref{thm:inrnor} and Remark~\ref{rem:inrnor} in mind we want to describe the irreducible constituents of
$\chi_{\boldsymbol{t}}^{(\pm)}\uparrow^{\tilde{N}_p^t\tilde{S}_{\boldsymbol{t}}}$. To do this let's first assume
$t_0=0$.

\begin{lem}\label{lem:exten>0}
If $t_0=0$ then $\chi_{\boldsymbol{t}}^{(\pm)}$ extends to a character of $\tilde{N}_p^t\tilde{S}_{\boldsymbol{t}}$.
Moreover, this character
is self-associate if and only if $t$ is even.
\end{lem}

\begin{proof}
Using the description in $\S$\ref{sec:parachars} we can view the representation space of $\chi_{\boldsymbol{t}}^{(\pm)}$
as an irreducible
$\mathcal{C}_t-$module. We can now define an action of $\tilde{S}_{\boldsymbol{t}}$ via the homomorphism
\begin{align*}
\tilde{S}_{\boldsymbol{t}}\hookrightarrow\tilde{S}_t\to\mathcal{C}_t,
\end{align*}
where the second map is given by $\phi_t$ in Lemma~\ref{lem:embed}. Note that by part (2) of this lemma and the
identification of factors in
$\S$\ref{subsec:NS} this does indeed define a representation of $\tilde{N}_p^t\tilde{S}_{\boldsymbol{t}}$. The second
part is immediate
from~\ref{lem:cliffrep}. For more details see~\cite[Lemma 4.3]{micols1990}.
\end{proof}

We label the character(s) in the above lemma $\chi^{(\pm)}_{\overline{\boldsymbol{t}}}$. Continuing with our assumption
that $t_0=0$ we now
are in a position to describe all the irreducible constituents of
$\chi_{\boldsymbol{t}}^{(\pm)}\uparrow^{\tilde{N}_p^t\tilde{S}_{\boldsymbol{t}}}$ in this case.

\begin{lem}\label{lem:j>0}
If $t_0=0$ then the irreducible constituents of
$\chi_{\boldsymbol{t}}^{(\pm)}\uparrow^{\tilde{N}_p^t\tilde{S}_{\boldsymbol{t}}}$ are
precisely the characters of the form $\chi^{(\pm)}_{\overline{\boldsymbol{t}}}\otimes\chi$ where $\chi$ is an
irreducible character of
$S_{t_1}\times\dots\times S_{t_{(p-1)/2}}$ inflated to $\tilde{N}_p^t\tilde{S}_{\boldsymbol{t}}$. Moreover,
$\chi^{(\pm)}_{\overline{\boldsymbol{t}}}\otimes\chi$ is self-associate if and only if $t$ is even and if $t$ is odd
then
$\chi^{\pm}_{\overline{\boldsymbol{t}}}\otimes\chi$ are associates.
\end{lem}

\begin{proof}
Adopting the notation from Lemma~\ref{lem:proj} we set
$G:=\tilde{N}_p^t\tilde{S}_{\boldsymbol{t}}$, $H:=\tilde{N}_p^t$ and $M_1,M_2$ to be the representation spaces
corresponding to
$\chi^{(\pm)}_{\overline{\boldsymbol{t}}}$ and $\chi$ respectively. Lemma~\ref{lem:proj} now says that the required
character is
irreducible. Next suppose $M=M_1\otimes M_2$ and $M'=M_1\otimes M_2'$ are two such representations. If $M$ and $M'$ are
isomorphic as
$\mathbb{C}G-$modules then they must be isomorphic as $\mathbb{C}H-$modules and hence by Schur's lemma any isomorphism
is of the form
$1\otimes\psi$ for some invertible linear map $\psi:M_2\to M_2'$. Hence, distinct $\chi$ give rise to distinct
characters
$\chi^{(\pm)}_{\overline{\boldsymbol{t}}}\otimes\chi$. We now show there are no more characters of $G$ appearing. Let
$e$ be the idempotent/sum of idempotents
of $\mathbb{C}H$ corresponding to the character(s) $\chi^{(\pm)}_{\overline{\boldsymbol{t}}}$. Then
\begin{align*}
&\dim(\mathbb{C}Ge)\geq\sum_{\chi\in\operatorname{Irr}(S_{t_1}\times\dots\times
S_{t_{(p-1)/2}})}\dim(\chi^{(\pm)}_{\overline{\boldsymbol{t}}}\otimes\chi)^2\\
=&\dim(\chi^{(\pm)}_{\overline{\boldsymbol{t}}})^2\sum_{\chi\in\operatorname{Irr}(S_{t_1}\times\dots\times
S_{t_{(p-1)/2}})}\dim(\chi)^2=\dim(\mathbb{C}He)t_1!\dots t_{(p-1)/2}!=\dim(\mathbb{C}Ge).
\end{align*}
Therefore we have equality throughout and all irreducible constituents of
$\chi_{\boldsymbol{t}}^{(\pm)}\uparrow^{\tilde{N}_p^t\tilde{S}_{\boldsymbol{t}}}$ are of the desired form.
\newline
\newline
For the second part we note that if $t$ is even then by Lemma~\ref{lem:exten>0}
$\chi_{\overline{\boldsymbol{t}}}^{(\pm)}(x)=0$ and hence
$(\chi^{(\pm)}_{\overline{\boldsymbol{t}}}\otimes\chi)(x)=0$ for any
$x\in\tilde{N}_p^t\tilde{S}_t\backslash\tilde{A}_{pt}$. If $t$ is odd
then there exists $x\in\tilde{N}_p^t\backslash\tilde{A}_{pt}$ such that
$\chi_{\overline{\boldsymbol{t}}}^{(\pm)}(x)\neq0$ and hence
$(\chi^{(\pm)}_{\overline{\boldsymbol{t}}}\otimes\chi)(x)\neq0$. It is clear that
$\chi^{(\pm)}_{\overline{\boldsymbol{t}}}\otimes\chi$ are
associates.
\end{proof}

If $\lambda=(\varnothing,\lambda_1,\dots,\lambda_{(p-1)/2})$ with $\lambda_j\vdash t_j$ then $\lambda$ labels a
character of
$S_{t_1}\times\dots\times S_{t_{(p-1)/2}}$ and we denote by $\chi_\lambda$ if $t$ is even or by $\chi_\lambda^\pm$ if
$t$ is odd the
corresponding character of $\tilde{N}_p^t\tilde{S}_{\boldsymbol{t}}$. As described in $\S$\ref{sec:double} we also have
the character(s)
$\overline{\chi}_\lambda^{(\pm)}$ of $\tilde{N}_p^t\tilde{S}_{\boldsymbol{t}}\cap\tilde{A}_{pw}$. We do not worry about
the choices made
for the labelling of $\chi_\lambda^\pm$ or $\overline{\chi}_\lambda^\pm$ only that it is fixed from now on.
\newline
\newline
We now turn our attention to the case $t_1=\dots=t_{(p-1)/2}=0$. In this case, unless $t\leq 3$,
$\chi_{\boldsymbol{t}}$ does not extend to a character of $\tilde{N}_p^t\tilde{S}_t$. However, it does always extend to
a character of $\tilde{N}_p^t\rtimes S_t$, where the action of $S_t$ on $\tilde{N}_p^t$ is given, through $\theta_{t}$,
by the
action of $\tilde{S}_t$ in $\tilde{N}_p^t\tilde{S}_t$. Let $\rho:\operatorname{GL}(U)\to\operatorname{GL}(U)$ be the
representation
corresponding to $\zeta_0$ and $S:U\to U$ an associator for $\rho$. Let $U^+$ (respectively $U^-$) be the $1-$eigenspace
(respectively
$(-1)-$eigenspace) of $S$ and let $u_1$ and $u_2$ be eigenvectors for $S$. We define the map
\begin{align*}
T:U\otimes U&\to U\otimes U\\
u_1\otimes u_2&\mapsto\eta(u_2\otimes u_1),\\
\text{where }\eta&:=
\begin{cases}
-1 &\text{ if }u_1,u_2\in U^-,\\
1 &\text{ otherwise,}
\end{cases}
\end{align*}
and extend linearly.

\begin{lem}\label{lem:T}$ $
\begin{enumerate}
\item $T$ commutes with $S\otimes S$.
\item $((S\otimes T)\circ(T\otimes S))^3=1$.
\item If $x\in\tilde{N}_p$ then $T\circ(\rho_1(x)\otimes1)\circ
T^{-1}=S^{(1-\epsilon(x))/2}\otimes\rho_1(x)$.
\end{enumerate}
\end{lem}

\begin{proof}$ $
\begin{enumerate}
\item Let $u_1,u_2\in U$ be eigenvalues of $S$. Then,
\begin{align*}
(T\circ(S\otimes S))(u_1\otimes u_2)=\eta(u_2\otimes u_1)=((S\otimes S)\circ T)(u_1\otimes u_2),\\
\text{where }\eta:=
\begin{cases}
1 &\text{ if }u_1,u_2\in U^+,\\
-1 &\text{ otherwise,}
\end{cases}
\end{align*}
and hence $T$ and $S\otimes S$ commute.
\item Let $u_1,u_2,u_3\in U$ be eigenvalues of $S$. Then,
\begin{align*}
&((S\otimes T)\circ(T\otimes S))(u_1\otimes u_2\otimes u_3)=\eta(u_2\otimes u_3\otimes u_1),\\
\text{where }\eta:=&
\begin{cases}
-1 &\text{ if }u_1\in U^+\text{ and }u_2,u_3\text{ do not have the same eigenvalue,}\\
1 &\text{ otherwise,}
\end{cases}
\end{align*}
and hence $((S\otimes T)\circ(T\otimes S))^3=1$.
\item Let $u_1,u_2\in U$ be eigenvalues of $S$. Then,
\begin{align*}
(T\circ(\rho_1(x)\otimes1))(u_1\otimes u_2)&=\eta u_2\otimes u_1=((S^{(1-\epsilon(x))/2}\otimes\rho_1(x))\circ
T)(u_1\otimes u_2),\\
\text{where }\eta&:=
\begin{cases}
-1 &\text{ if }u_1\in U^+\text{ and }u_2\in U^-,\\
1 &\text{ otherwise,}
\end{cases}
\end{align*}
and hence $T\circ(\rho_1(x)\otimes1)\circ T^{-1}=S^{(1-\epsilon(x))/2}\otimes\rho_1(x)$.
\end{enumerate}
\end{proof}

We decompose $\zeta_0$ as a character of $\tilde{N}_p\cap\tilde{A}_p$ as
$\zeta_0=\overline{\zeta}^+_0+\overline{\zeta}^-_0$ where $\overline{\zeta}^+_0$ (respectively $\overline{\zeta}^-_0$)
corresponds to the subspace $U^+$ (respectively $U^-$).
\newline
\newline
Now let $U_1,\dots,U_t$ be $t$ isomorphic copies of $U$ and let $S_j$ be the linear map on $U_j$ corresponding to $S$
through the isomorphism with $U$. Now for $(1\leq j\leq t-1)$ define the linear map
\begin{align*}
T_j:U_1\otimes\dots\otimes U_t&\to U_1\otimes\dots\otimes U_t\\
u_1\otimes\dots\otimes u_t&\mapsto S_1(u_1)\otimes\dots\otimes S_{j-1}(u_{j-1})\otimes T\otimes
S_{j+2}(u_{j+2})\otimes\dots\otimes S_t(u_t).
\end{align*}
We now define a representation of $\tilde{N}_p^t\rtimes S_t$. First identify $U_1\otimes\dots\otimes U_t$ with the
representation space of $\chi_{\boldsymbol{t}}$ via the description in $\S$\ref{sec:parachars}. 

\begin{lem}\label{lem:exten0}
The following action defines a representation $\rho$ of $\tilde{N}_p^t\rtimes S_t$ on $U_1\otimes\dots\otimes U_t$
that extends $\chi_{\boldsymbol{t}}$,
\begin{align*}
\rho(s_j)(u):=T_j(u)\text{ for all }u\in U_1\otimes\dots\otimes U_t\text{ and }(1\leq j\leq t-1).
\end{align*}
\end{lem}

\begin{proof}
One needs to check that $\rho$ defines a representation of $S_t$ and that
$\rho(s_jxs_j^{-1})=\rho(s_j)\rho(x)\rho(s_j^{-1})$ for all $(1\leq j\leq t-1)$ and $x\in\tilde{N}_p^t$.
\newline
\newline
First we note that $T_j^2=1$ and, as $T$ commutes with $S\otimes S$ by Lemma~\ref{lem:T} part (1), we also have that
$T_jT_l=T_lT_j$ for $|j-l|>1$. Lemma~\ref{lem:T} part (2) says that $((S\otimes T)\circ(T\otimes S))^3=1$ and hence that
$(T_jT_{j+1})^3=1$ and so $\rho$ does indeed define a representation of $S_t$.
\newline
\newline
Now let $x\in\tilde{N}_p^t$ such that $x$ lies in the $l^{\operatorname{th}}$ factor of $\tilde{N}_p^t$. If $l\neq
j,j+1$
then
\begin{align*}
T_j\rho(x)T_j^{-1}=\epsilon(x)\rho(x)=\rho(s_jxs_j^{-1}).
\end{align*}
If $l=j$ then Lemma~\ref{lem:T} part (3) tells us that
$T_j\rho(x)T_j^{-1}=\rho(s_jxs_j^{-1})$ and similarly for $l=j+1$.
\end{proof}

We denote by $\operatorname{Exten}^+_t$ the character of the $\rho$ in the above lemma.
We now wish to explicitly describe the character values of $\operatorname{Exten}_t^+$. We denote by
$\frac{1}{2}(\tilde{N}_p^t\rtimes S_t)$ the preimage of $(N_p\wr S_t)\cap A_{pt}$ under the natural map
$\tilde{N}_p^t\rtimes S_t\twoheadrightarrow N_p\wr S_t$.
\newline
\newline
We mirror the notation from $\S$\ref{subsec:NS} for labelling the elements of $\tilde{N}_p^t\rtimes S_t$.

\begin{lem}\label{lem:semidirect}
Let $0=n_0<\dots<n_m=t$ and $x_l\in\tilde{N}_p$ for $(1\leq l\leq m)$. Consider
\begin{align*}
x=\prod_l((x_l,1,\dots,1);(n_{l-1}+1,\dots,n_l))\in\tilde{N}_p^t\rtimes S_t.
\end{align*}
If for any $l$ we have $x_l\notin\tilde{A}_p$ then $\operatorname{Exten}^+_t(x)=0$, otherwise
\begin{align*}
\operatorname{Exten}^+_t(x)=
\begin{cases}
\prod\limits_{l,(n_l-n_{l-1})\text{odd}}\zeta_0(x_l).\prod\limits_{l,(n_l-n_{l-1})\text{even}}(\overline{\zeta}
^+_0(x_l)-\overline{\zeta}^-_0(x_l))&\text{if }x\in\frac{1}{2}(\tilde{N}_p^t\rtimes S_t),\\
\prod\limits_l(\overline{\zeta}^+_0(x_l)-\overline{\zeta}^-_0(x_l))&\text{if
}x\notin\frac{1}{2}(\tilde{N}_p^t\rtimes S_t).
\end{cases}
\end{align*}
\end{lem}

\begin{proof} Let $u^+_1,\dots u^+_{(p-1)/2}$ be a basis of $U^+$ and $u^-_1,\dots u^-_{(p-1)/2}$ a basis of $U^-$.
Consider the action of $x$ on $U^{\otimes t}$ with respect to the basis
\begin{align*}
\{u_{j_1}^{\epsilon_1}\otimes\dots\otimes u_{j_t}^{\epsilon_t}\}_{j_l\in[\frac{p-1}{2}],\epsilon_l\in\{\pm1\}}.
\end{align*}
We can view the action of $x$ as a Kronecker product of its actions on the vector spaces $V_l$ with bases
\begin{align*}
\{u_{j_{n_{l-1}+1}}^{\epsilon_{n_{l-1}+1}}\otimes\dots\otimes u_{j_{n_l}}^{\epsilon_{n_l}}\}
\end{align*}
for varying $l$. Note that the action of $x$ on $V_l$ is not necessarily the same as the action of
$((x_l,1,\dots,1);(n_{l-1}+1,\dots,n_l))$ on $V_l$. $((x_r,1,\dots,1);(n_{r-1}+1,\dots,n_r))$ may, for some $r\neq l$,
act via
$S^{\otimes(n_l-n_{l-1})}$ on $V_l$. If this is true for an odd number of $r\neq l$ then ``an extra $S$'' acts
on $V_l$. We now write down how $x$ acts on $V_l$.
\begin{align*}
x.&(u_{j_{n_{l-1}+1}}^{\epsilon_{n_{l-1}+1}}\otimes\dots\otimes u_{j_{n_l}}^{\epsilon_{n_l}})=
\eta(x_l.u_{j_{n_{l-1}+2}}^{\epsilon_{n_{l-1}+2}}\otimes u_{j_{n_{l-1}+3}}^{\epsilon_{n_{l-1}+3}}\otimes\dots\otimes
u_{j_{n_l}}^{\epsilon_{n_l}}\otimes u_{j_{n_{l-1}+1}}^{\epsilon_{n_{l-1}+1}}),\\
&\text{where}\\
\eta=&
\begin{cases}
\big{(}\prod_{r=n_{l-1}+2}^{n_l}\epsilon_r\big{)}&\text{if }\epsilon_{n_{l-1}+1}=1\text{ and no extra }S,\\
\big{(}\prod_{r=n_{l-1}+2}^{n_l}\epsilon_r\big{)}\big{(}\prod_{r=n_{l-1}+1}^{n_l}\epsilon_r\big{)}&\text{if
}\epsilon_{n_{l-1}+1}=1\text{ and extra }S,\\
\big{(}\prod_{r=n_{l-1}+2}^{n_l}\epsilon_r\big{)}&\text{if }\epsilon_{n_{l-1}+1}=-1\text{ and no extra }S,\\
\big{(}\prod_{r=n_{l-1}+2}^{n_l}\epsilon_r\big{)}\big{(}\prod_{r=n_{l-1}+1}^{n_l}\epsilon_r\big{)}&\text{if
}\epsilon_{n_{l-1}+1}=-1\text{ and extra }S.
\end{cases}
\end{align*}
Now one can see that the only basis vectors $v$ of $V_l$ that could possibly appear with non-zero coefficient in $x.v$
are those of the form $u\otimes\dots\otimes u$. Since elements of $\tilde{N}_p\backslash(\tilde{N}_p\cap\tilde{A}_p)$
swap
$U^+$ and $U^-$ we have that $\operatorname{Exten}_t^+(x)=0$ if $x_l\notin\tilde{A}_p$ for some $l$. So let's assume
$x_l\in\tilde{A}_p$ for all $l$. Now for $j_{n_{l-1}+1}=\dots=j_{n_l}$ and $\epsilon_{n_{l-1}+1}=\dots=\epsilon_{n_l}$
we have
\begin{align*}
\eta=&
\begin{cases}
1&\text{if }\epsilon_{n_{l-1}+1}=1\text{ and no extra }S,\\
1&\text{if }\epsilon_{n_{l-1}+1}=1\text{ and extra }S,\\
(-1)^{n_l-n_{l-1}-1}&\text{if }\epsilon_{n_{l-1}+1}=-1\text{ and no extra }S,\\
-1&\text{if }\epsilon_{n_{l-1}+1}=-1\text{ and extra }S.
\end{cases}
\end{align*}
Then $u\otimes\dots\otimes u$ appears in $x.(u\otimes\dots\otimes u)$ with coefficient $\eta$ (as above) multiplied by
the coefficient of $u$ in $x_lu$. Therefore, the trace of the action of $x$ on
$V_l$ is
\begin{align*}
\begin{cases}
\zeta_0(x_l)&\text{if }n_l-n_{l-1}\text{ is odd and no extra }S,\\
\overline{\zeta}^+_0(x_l)-\overline{\zeta}^-_0(x_l)&\text{otherwise.}
\end{cases}
\end{align*}
The lemma now follows.
\end{proof}

We now wish to show that if $t>1$ then $\operatorname{Exten}^+_t\downarrow_{\frac{1}{2}(\tilde{N}_p^t\rtimes S_t)}$ is
irreducible. This is equivalent to there existing some
$x\in(\tilde{N}_p^t\rtimes S_t)\backslash\frac{1}{2}(\tilde{N}_p^t\rtimes S_t)$ with $\operatorname{Exten}^+_t(x)\neq0$.
Due to Lemma~\ref{lem:semidirect} it is enough to show that there exists some $x\in\tilde{N}_p\cap\tilde{A}_p$ such that
$\overline{\zeta}^+_0(x)\neq\overline{\zeta}^-_0(x)$. This is of course the case. Will denote by
$\operatorname{Exten}^-_t$ the associate of $\operatorname{Exten}^+_t$ with respect to the subgroup
$\frac{1}{2}(\tilde{N}_p^t\rtimes S_t)\leq\tilde{N}_p^t\rtimes S_t$.
\newline
\newline
As a $\mathbb{C}(\tilde{N}_p^t\cap\tilde{A}_{pt})-$module $U^{\otimes t}$ decomposes as
\begin{align*}
\bigg{(}\bigoplus_{\prod_j\epsilon_j=1}\bigotimes_jU_j^{\epsilon_j}\bigg{)}\oplus\bigg{(}\bigoplus_{\prod_j\epsilon_j=-1
}\bigotimes_jU_j^{\epsilon_j}\bigg{)}.
\end{align*}
Note that this decomposition respects the action of $A_t$ (or even $S_t$) and so extends to an irreducible decomposition
of $U^{\otimes t}$ as a $\mathbb{C}((\tilde{N}_p\cap\tilde{A}_p)\rtimes A_t)-$module. We label the two characters of
these modules $\overline{\operatorname{Exten}}_t^+$ and $\overline{\operatorname{Exten}}_t^-$ respectively. We now
describe these two characters.

\begin{lem}\label{lem:aa}
Let $0=n_0<\dots<n_m=t$ and $x_l\in\tilde{N}_p$ for $(1\leq l\leq m)$. Consider
\begin{align*}
x=\prod_l((x_l,1,\dots,1);(n_{l-1}+1,\dots,n_l))\in(\tilde{N}_p^t\cap\tilde{A}_{pt})\rtimes A_t.
\end{align*}
If for any $l$ we have $x_l\notin\tilde{A}_p$ then $\overline{\operatorname{Exten}}_t^\pm(x)=0$, otherwise
\begin{align*}
\overline{\operatorname{Exten}}_t^+(x)=&\sum_{\prod\limits_{l,(n_l-n_{l-1})\text{odd}}\epsilon_l=1}\bigg{(}\prod_{l,
(n_l-n_ { l-1})\text{odd}}\overline{ \zeta}^{\epsilon_l}_0(x_l)\bigg{)}.
\prod_{l,(n_l-n_{l-1})\text{even}}(\overline{\zeta}^+_0(x_l)-\overline{\zeta}^-_0(x_l))\\
&\text{and}\\
\overline{\operatorname{Exten}}_t^-(x)=&\sum_{\prod\limits_{l,(n_l-n_{l-1})\text{odd}}\epsilon_l=-1}\bigg{(}\prod_{l,
(n_l-n_{l-1})\text{odd}}\overline{ \zeta }^{ \epsilon_l}_0(x_l)\bigg{)}.
\prod_{{l,(n_l-n_{l-1})\text{even}}}(\overline{\zeta}^+_0(x_l)-\overline{\zeta}^-_0(x_l)).\\
\end{align*}
\end{lem}

\begin{proof}
We proceed as in the proof of Lemma~\ref{lem:semidirect}. For $\epsilon\in\{\pm1\}$ we define $V_l^\epsilon$ to be the
subspace of $V_l$ that is the linear span of
\begin{align*}
\{u_{j_{n_{l-1}+1}}^{\epsilon_{n_{l-1}+1}}\otimes\dots\otimes
u_{j_{n_l}}^{\epsilon_{n_l}}\}_{\prod_r\epsilon_r=\epsilon}.
\end{align*}
Clearly we have $V_l=V_l^+\oplus V_l^-$ and also
\begin{align*}
\bigoplus_{\prod_j\epsilon_j=1}\bigotimes_jU^{\epsilon_j}=\bigoplus_{\prod_l\epsilon_l=1}\bigotimes_lV_l^{\epsilon_l}
\text{ and }
\bigoplus_{\prod_j\epsilon_j=-1}\bigotimes_jU^{\epsilon_j}=\bigoplus_{\prod_l\epsilon_l=-1}\bigotimes_lV_l^{\epsilon_l
}.
\end{align*}
Once again we can conclude that $\overline{\operatorname{Exten}}_t^+(x)=\overline{\operatorname{Exten}}_t^-(x)=0$ if
for some $l$ we have that $x_l\notin\tilde{A}_p$ so let's assume $x_l\in\tilde{A}_p$ for all $l$. Then the trace of
the action of $x$ on $V_l^\epsilon$ is
\begin{align*}
\begin{cases}
\overline{\zeta}^\epsilon_0(x_l)&\text{if }n_l-n_{l-1}\text{ is odd},\\
\overline{\zeta}^+_0(x_l)-\overline{\zeta}^-_0(x_l)&\text{if }n_l-n_{l-1}\text{ is even and }\epsilon=1,\\
0&\text{if }n_l-n_{l-1}\text{ is even and }\epsilon=-1,
\end{cases}
\end{align*}
and hence the result.
\end{proof}

Now set $M_1:=U^{\otimes t}$ and $M_2$ the representation space of the character $\xi^{(\pm)}_\lambda$ of $\tilde{S}_t$,
where $\lambda\vdash t$. We can now define a representation of $\tilde{N}_p^t\tilde{S}_t$ on $M_1\otimes M_2$. If
$m_1\in M_1$, $m_2\in M_2$, $x\in\tilde{N}_p^t$ and $s\in\tilde{S}_t$ then we define the action as follows:
\begin{align*}
x:m_1\otimes m_2&\mapsto x.m_1\otimes m_2\\
s:m_1\otimes m_2&\mapsto\theta_t(s).m_1\otimes s.m_2.
\end{align*}
We denote by $\chi_\lambda^{(\pm)}$ the character of this representation. If $x\in\tilde{N}_p^t$ and $s\in\tilde{S}_t$
then $\chi_\lambda^{(\pm)}(xs)=\operatorname{Exten}_t^+(x\theta_t(s))\xi^{(\pm)}_\lambda(s)$.

\begin{lem}\label{lem:j=0}
The irreducible constituents of $\chi_{\boldsymbol{t}}\uparrow^{\tilde{N}_p^t\tilde{S}_t}$ are precisely the characters
\begin{align*}
\{\chi_\lambda|\lambda\vdash t,\sigma(\lambda)=1\}\cup\{\chi_\lambda^\pm|\lambda\vdash t,
\sigma(\lambda)=-1\}.
\end{align*}
Furthermore, if $\sigma(\lambda)=1$ then $\chi_\lambda$ is self-associate and if $\sigma(\lambda)=-1$ then
$\chi_\lambda^\pm$ are associates.
\end{lem}

\begin{proof}
The first part is identical to the proof of Lemma~\ref{lem:j>0}. For the second part suppose
$x\in\tilde{N}_p^t$ and $s\in\tilde{S}_t$. If $\sigma(\lambda)=1$ then as $\xi_\lambda$ is self-associate then
$\chi_\lambda(xs)=0$ if $s\in\tilde{S}_t\backslash\tilde{A}_t$. Also if $x\in\tilde{N}_p^t\backslash\tilde{A}_{pt}$ then
by Lemma~\ref{lem:semidirect} $\chi_\lambda(xs)=0$.
If $\sigma(\lambda)=-1$ then as $\xi^\pm_\lambda$ is non-self-associate then $\xi_\lambda^\pm(g)\neq0$ for some
$g\in\tilde{S}_t\backslash\tilde{A}_t$. Now if in Lemma~\ref{lem:semidirect} we construct $x$ using $\theta_t(g)$ and
all
$x_l$'s such that $\overline{\zeta_0}^+(x_l)\neq\overline{\zeta_0}^-(x_l)$ then $\chi_\lambda^\pm(x)\neq0$ and
$x\in\tilde{N}_p^t\tilde{S}_t\backslash\tilde{A}_{pt}$. We now clearly have that $\chi_\lambda^\pm$ are associates.
\end{proof}

Note that if in the definition of $\chi_\lambda^{(\pm)}$ we replace $\operatorname{Exten}_t^+$ by
$\operatorname{Exten}_t^-$ or we replace $\xi_\lambda^{(\pm)}$ by $\xi_\lambda^{(\mp)}$ then we replace
$\chi_\lambda^{(\pm)}$ by $\chi_\lambda^{(\mp)}$.
\newline
\newline
Let $t>1$ and $\sigma(\lambda)=1$. By directly decomposing the representation space of $\chi_\lambda$ we see that we
have a
character $\overline{\overline{\chi}}_\lambda^{++}$ of $(\tilde{N}_p^t\cap\tilde{A}_{pt})\tilde{A}_t$ given by
\begin{align*}
\overline{\overline{\chi}}_\lambda^{++}(xs)=\overline{\operatorname{Exten}}_t^+(x\theta_t(s))\overline{\xi}^+_\lambda(s)
\text{ for all }x\in\tilde{N}_p^t\cap\tilde{A}_{pt}\text{ and }s\in\tilde{A}_t.
\end{align*}
By comparing dimensions we see that
$\overline{\overline{\chi}}_\lambda^{++}\uparrow^{\tilde{N}_p^t\tilde{S}_t}=\chi_\lambda$
and hence $\overline{\overline{\chi}}_\lambda^{++}$ is irreducible. Frobenius reciprocity now gives that
$\chi_\lambda\downarrow_{(\tilde{N}_p^t\cap\tilde{A}_{pt})\tilde{A}_t}$ is the sum of four non-isomorphic characters. We
label
these four characters
\begin{align*}
\overline{\overline{\chi}}_\lambda^{++}(xs)&=\overline{\operatorname{Exten}}_t^+(x\theta_t(s))\overline{\xi}
^+_\lambda(s)\\
\overline{\overline{\chi}}_\lambda^{+-}(xs)&=\overline{\operatorname{Exten}}_t^+(x\theta_t(s))\overline{\xi}
^-_\lambda(s)\\
\overline{\overline{\chi}}_\lambda^{-+}(xs)&=\overline{\operatorname{Exten}}_t^-(x\theta_t(s))\overline{\xi}
^+_\lambda(s)\\
\overline{\overline{\chi}}_\lambda^{--}(xs)&=\overline{\operatorname{Exten}}_t^-(x\theta_t(s))\overline{\xi}
^-_\lambda(s),
\end{align*}
where $x\in\tilde{N}_p^t\cap\tilde{A}_{pt}$ and $s\in\tilde{A}_t$. Note that by the comments preceding~\ref{lem:aa} we
have
that $\overline{\operatorname{Exten}}_t^+$ is fixed by conujugation by $S_t$. Therefore conjugating by
$s\in\tilde{S}_t\backslash\tilde{A}_t$ takes $\overline{\overline{\chi}}_\lambda^{++}$ to
$\overline{\overline{\chi}}_\lambda^{+-}$. Also conjugating by $x\in\tilde{N}_p^t\backslash\tilde{A}_{pt}$ cannot fix
$\overline{\overline{\chi}}_\lambda^{++}$ and so must take it to $\overline{\overline{\chi}}_\lambda^{-+}$. We now label
the
following irreducible characters of $\tilde{N}_p^t\tilde{S}_t\cap\tilde{A}_{pt}$
\begin{align*}
\overline{\chi}_\lambda^+&:=\overline{\overline{\chi}}_\lambda^{++}\uparrow^{(\tilde{N}_p^t\tilde{S}_t)\cap\tilde{A}_{pt
}}=
\overline{\overline{\chi}}_\lambda^{--}\uparrow^{(\tilde{N}_p^t\tilde{S}_t)\cap\tilde{A}_{pt}}\\
\overline{\chi}_\lambda^-&:=\overline{\overline{\chi}}_\lambda^{+-}\uparrow^{(\tilde{N}_p^t\tilde{S}_t)\cap\tilde{A}_{pt
}}=
\overline{\overline{\chi}}_\lambda^{-+}\uparrow^{(\tilde{N}_p^t\tilde{S}_t)\cap\tilde{A}_{pt}}.
\end{align*}
Now suppose $\sigma(\lambda)=-1$. Again we have characters $\overline{\overline{\chi}}_\lambda^\pm$ of
$(\tilde{N}_p^t\cap\tilde{A}_{pt})\tilde{A}_t$ given by
\begin{align*}
\overline{\overline{\chi}}_\lambda^\pm(xs)=\overline{\operatorname{Exten}}_t^\pm(x\theta_t(s))\overline{\xi}_\lambda(s)
\text{ for all }x\in\tilde{N}_p^t\cap\tilde{A}_{pt}\text{ and }s\in\tilde{A}_t.
\end{align*}
By comparing dimensions we see that
$\overline{\overline{\chi}}_\lambda^\pm\uparrow^{\tilde{N}_p^t\tilde{S}_t\cap\tilde{A}_{pt}}=\overline{\chi}_\lambda$
and hence $\overline{\overline{\chi}}_\lambda^+$ is irreducible.
\newline
\newline
We now drop all the assumptions on $\boldsymbol{t}=(t_0,\dots,t_{(p-1)/2})$. Using Lemma~\ref{lem:t=2} we will
describe all the irreducible
constituents of $\chi_{\boldsymbol{t}}^{(\pm)}\uparrow^{\tilde{N}_p^t\tilde{S}_{\boldsymbol{t}}}$.
\newline
\newline
Let $\lambda=(\lambda_0,\dots,\lambda_{(p-1)/2})$ be a $(p+1)/2-$tuple of partitions with $\lambda_0$ strict and
$|\lambda_j|=t_j$ for $(0\leq j\leq(p-1)/2)$. We set $\boldsymbol{t}(\lambda):=(t_0,\dots,t_{(p-1)/2})$ and define
$\sigma(\lambda):=\sigma(\lambda_0)(-1)^{t-t_0}$ and
\begin{align*}
\Delta_t^+:=\{\lambda\in\Delta_t|\sigma(\lambda)=1\},\Delta_t^-:=\{\lambda\in\Delta_t|\sigma(\lambda)=-1\}.
\end{align*}
Now using Lemma~\ref{lem:j>0} we can construct the character(s)
$\chi_{(\varnothing,\lambda_1,\dots,\lambda_{(p-1)/2})}^{(\pm)}$ of
$\tilde{N}_p^{t-t_0}\tilde{S}_{(0,t_1,\dots,t_{(p-1)/2})}$. As described in Lemma~\ref{lem:shift} we can and do identify
this subgroup with $\tilde{N}_p^{t-t_0}\tilde{S}_{(0,t_1,\dots,t_{(p-1)/2})}[t_0]$. Using Lemma~\ref{lem:j=0} we
also have the character(s) $\chi_{\lambda_0}^{(\pm)}$ of $\tilde{N}_p^{t_0}\tilde{S}_{t_0}$. With the labelling
convention of
Lemma~\ref{lem:t=2} we can construct the character(s) $\chi_{\lambda}^{(\pm)}$ of
$\tilde{N}_p^t\tilde{S}_{\boldsymbol{t}}$.
Now due to Theorem~\ref{thm:inrnor} and Remark~\ref{rem:inrnor} we have that
$\chi_{\lambda}^{(\pm)}\uparrow^{\tilde{N}_p^t\tilde{S}_t}$ is irreducible and we denote this character by
$\chi^{\lambda(\pm)}$.

\begin{thm}\label{thm:maincha}
A complete list of irreducible spin characters of $\tilde{N}_p^t\tilde{S}_t$ is given by
\begin{align*}
\{\chi^\lambda|\lambda\in\Delta_t^+\}\cup\{\chi^{\lambda\pm}|\lambda\in\Delta_t^-\}.
\end{align*}
\end{thm}

\begin{proof}
If $\boldsymbol{t}=(t_0,\dots,t_{(p-1)/2})$ then due to Theorem~\ref{thm:inrnor}, Remark~\ref{rem:inrnor} and
Lemma~\ref{lem:list} all we have to show is that
\begin{align*}
\{\chi_\lambda|\lambda\in\Delta_t^+,\boldsymbol{t}(\lambda)=\boldsymbol{t}\}\cup
\{\chi_\lambda^\pm|\lambda\in\Delta_t^-,\boldsymbol{t}(\lambda)=\boldsymbol{t}\}
\end{align*}
is a complete list of irreducible constituents of
$\chi_{\boldsymbol{t}}^{(+)}\uparrow^{\tilde{N}_p^t\tilde{S}_{\boldsymbol{t}}}$
if $t_0>1$ or of $\chi_{\boldsymbol{t}}^{(\pm)}\uparrow^{\tilde{N}_p^t\tilde{S}_{\boldsymbol{t}}}$ if $t_0\leq1$. Let's
set
$G:=\tilde{N}_p^t\tilde{S}_{\boldsymbol{t}}$ and $H:=\tilde{N}_p^t$. First suppose $t_0>1$ and $t-t_0$ is even and let
$e$ be the
idempotent of $\mathbb{C}H$ corresponding to the character $\chi_{\boldsymbol{t}}$. Then by
Lemmas~\ref{lem:t=2},~\ref{lem:j>0} and~\ref{lem:j=0}
\begin{align*}
&\dim(\mathbb{C}Ge)\geq\sum_{\lambda\in\Delta_t^+,\boldsymbol{t}(\lambda)=\boldsymbol{t}}\dim(\chi_\lambda)^2+
\sum_{\lambda\in\Delta_t^-,\boldsymbol{t}(\lambda)=\boldsymbol{t}}(\dim(\chi_\lambda^+)^2+\dim(\chi_\lambda^-)^2)\\
=&\sum_{\lambda_0\in\mathcal{D}_t^+,\chi\in\operatorname{Irr}(S_{t_1}\times\dots\times S_{t_{(p-1)/2}})}
(\dim(\xi_\lambda)\dim(\chi)\dim(\chi_{\boldsymbol{t}}))^2+\\
&\sum_{\lambda_0\in\mathcal{D}_t^-,\chi\in\operatorname{Irr}(S_{t_1}\times\dots\times S_{t_{(p-1)/2}})}
\big{(}(\dim(\xi_\lambda^+)\dim(\chi)\dim(\chi_{\boldsymbol{t}}))^2+
(\dim(\xi_\lambda^-)\dim(\chi)\dim(\chi_{\boldsymbol{t}}))^2\big{)}\\
=&t_0!t_1!\dots t_{(p-1)/2}!\dim(\chi_{\boldsymbol{t}})^2=t_0!t_1!\dots
t_{(p-1)/2}!\dim(\mathbb{C}He)=\dim(\mathbb{C}Ge).
\end{align*}
Therefore we have equality throughout and all irreducible constituents of
$\chi_{\boldsymbol{t}}\uparrow^{\tilde{N}_p^t\tilde{S}_{\boldsymbol{t}}}$ are of the desired form. Now let $t_0>1$,
$t-t_0$ be odd and $e_+$ and $e_-$ the idempotents of $\mathbb{C}H$ corresponding to the characters
$\chi_{\boldsymbol{t}}^+$ and
$\chi_{\boldsymbol{t}}^-$ respectively. Then similarly to above
\begin{align*}
&\dim(\mathbb{C}G(e_++e_-))\geq\sum_{\lambda\in\Delta_t^+,\boldsymbol{t}(\lambda)=\boldsymbol{t}}\dim(\chi_\lambda)^2+
\sum_{\lambda\in\Delta_t^-,\boldsymbol{t}(\lambda)=\boldsymbol{t}}(\dim(\chi_\lambda^+)^2+\dim(\chi_\lambda^-)^2)\\
=&\sum_{\lambda_0\in\mathcal{D}_t^-,\chi\in\operatorname{Irr}(S_{t_1}\times\dots\times S_{t_{(p-1)/2}})}
(2\dim(\xi_\lambda^+)\dim(\chi)\dim(\chi_{\boldsymbol{t}}^+))^2+\\
&\sum_{\lambda_0\in\mathcal{D}_t^+,\chi\in\operatorname{Irr}(S_{t_1}\times\dots\times S_{t_{(p-1)/2}})}
\big{(}(\dim(\xi_\lambda)\dim(\chi)\dim(\chi_{\boldsymbol{t}}^+))^2+
(\dim(\xi_\lambda)\dim(\chi)\dim(\chi_{\boldsymbol{t}}^-))^2\big{)}\\
=&2t_0!t_1!\dots t_{(p-1)/2}!\dim(\chi_{\boldsymbol{t}}^+)^2=2t_0!t_1!\dots t_{(p-1)/2}!\dim(\mathbb{C}He_+)
=\dim(\mathbb{C}G(e_++e_-)).
\end{align*}
Therefore we have equality throughout and all irreducible constituents of
$\chi_{\boldsymbol{t}}^+\uparrow^{\tilde{N}_p^t\tilde{S}_{\boldsymbol{t}}}$ are of the desired form. Similar
calculations prove
the result for $t_0\leq1$.
\end{proof}

\section{Results on character values}\label{sec:charsum}

In this section we prove many results about the character values of $\tilde{N}_p^t\tilde{S}_t$.

\begin{lem}\label{lem:sqrt}
Let $a\in N_p\cap\tilde{A}_p$ be an element of order $p$. Then
\begin{align*}
\overline{\zeta}^+_0(a)=\frac{-1\pm i^\frac{p-1}{2}\sqrt{p}}{2}.
\end{align*}
\end{lem}

\begin{proof}
Let $\psi$ be a faithful linear spin character of $\theta^{-1}(C_p)$ with
$\overline{\zeta}^+_0=\psi\uparrow^{\tilde{N}_p\cap\tilde{A}_p}$.
Then $\psi(a)=\omega_p$ for some primitive $p^{\operatorname{th}}$ root of unity $\omega_p$. We set
\begin{align*}
c:=\overline{\zeta}^+_0(a)=\sum^{(p-1)/2}_{j=1}\omega_p^{b^{2j}},
\end{align*}
where $b$ is a generator of the multiplicative group of $\mathbb{F}^\times_p$. We define the field automorphism
\begin{align*}
\tau:\mathbb{Q}(w_p)&\to\mathbb{Q}(w_p)\\
w_p&\mapsto w_p^b.
\end{align*}
Then
\begin{align*}
c+\tau(c)=-1\text{ and }(c-\tau(c))^2=\bigg{(}\frac{-1}{p}\bigg{)}p=(-1)^{\frac{p-1}{2}}p,
\end{align*}
as $(c-\tau(c))$ is a quadratic Gauss sum. This gives
\begin{align*}
c=\frac{-1\pm i^\frac{p-1}{2}\sqrt{p}}{2}.
\end{align*}
\end{proof}

\begin{lem}\label{lem:p'}
Let $a\in\tilde{N}_p$ have $p'-$order. Then $C_{\tilde{N}_p}(a)\nsubseteq\tilde{A}_p$. In particular
$\overline{\zeta}^+_0(a)=\overline{\zeta}^-_0(a)$.
\end{lem}

\begin{proof}
Suppose $b\in\tilde{N}_p$ with $\theta_p(b)=y_1$. All $p'-$elements of $N_p$ are conjugate to a power of $y_1$ and hence
all $p'-$elements of $\tilde{N}_p$ are conjugate to a power of $b$ or $z$ times a power of $b$. The first part is now
clear.
For the second part let $g\in C_{\tilde{N}_p}(a)\backslash\tilde{A}_p$. Then
$\overline{\zeta}^-_0(a)=\overline{\zeta}^+_0(gag^{-1})=\overline{\zeta}^+_0(a)$.
\end{proof}

\begin{lem}\label{lem:qcycle}
Let $q$ be an odd positive integer and $\tau=o((1,\dots,q))\in\tilde{S}_q$. Then
\begin{align*}
\chi^{\mathcal{C}\pm}_q(\phi_q(\tau))=(-1)^{(q^2-1)/8}.
\end{align*}
\end{lem}

\begin{proof}
By Lemma~\ref{lem:embed}
\begin{align*}
\phi_q(\tau)=\pm\frac{1}{2^{\frac{q-1}{2}}}(e_1+e_2)\dots(e_{q-1}+e_q).
\end{align*}
It is not clear which sign is correct but since $(1,\dots,q)^2$ is conjugate to $(1,\dots,q)$ in $S_q$, $\tau^2$ is
conjugate
to $\tau$ in $\tilde{S}_q$ and so we calculate
\begin{align*}
\chi^{\mathcal{C}\pm}_q\bigg{(}\bigg{(}\frac{1}{2^{\frac{q-1}{2}}}(e_1+e_2)\dots(e_{q-1}+e_q)\bigg{)}^2\bigg{)}.
\end{align*}
We now proceed by induction on $q$. The result is clearly true for $q=1$. By
the presentation of $S_n^\pm$ in $\S$\ref{sec:double} and by Lemma~\ref{lem:embed} for $q=3$ we have
\begin{align*}
\chi^{\mathcal{C}\pm}_3(\phi_3(\tau))=\chi^{\mathcal{C}\pm}_3(-\frac{1}{2}(e_1+e_2)(e_2+e_3))=
\frac{1}{2}.-2=-1=(-1)^{\frac{3^2-1}{8}}.
\end{align*}
Now we note that since $\phi_q(\tau)\in\mathcal{C}^+_q$,
\begin{align*}
\chi^{\mathcal{C}\pm}_q(\phi_q(\tau))=
2^{\frac{q-1}{2}}\chi_\varnothing\bigg{(}\bigg{(}\frac{1}{2^{\frac{q-1}{2}}}(e_1+e_2)\dots(e_{q-1}+e_q)\bigg{)}^2\bigg{)
}=
\frac{1}{2^{\frac{q-1}{2}}}\chi_\varnothing(((e_1+e_2)\dots(e_{q-1}+e_q))^2),
\end{align*}
where $\chi_\varnothing(x)$ is equal to the coefficient of $e_\varnothing$ in $x$. Now
\begin{align*}
&\chi^\varnothing([(e_1+e_2)\dots(e_{q-1}+e_q)][(e_1+e_2)\dots(e_{q-1}+e_q)])\\
=&\chi^\varnothing((e_1+e_2)(e_2+e_3)(e_1+e_2)[(e_3+e_4)\dots(e_{q-1}+e_q)][(e_2+e_3)\dots(e_{q-1}+e_q)])\\
=&\chi^\varnothing((2e_1-2e_3)[(e_3+e_4)\dots(e_{q-1}+e_q)][(e_2+e_3)\dots(e_{q-1}+e_q)])\\
=&2\chi^\varnothing((-e_3)[(e_3+e_4)\dots(e_{q-1}+e_q)][e_3(e_3+e_4)\dots(e_{q-1}+e_q)])\\
=&2\chi^\varnothing((-e_3)(e_3+e_4)(-e_3)[(e_4+e_5)\dots(e_{q-1}+e_q)][(e_3+e_4)\dots(e_{q-1}+e_q)])\\
=&2\chi^\varnothing((e_3-e_4)[(e_4+e_5)\dots(e_{q-1}+e_q)][(e_3+e_4)\dots(e_{q-1}+e_q)])\\
=&2\chi^\varnothing((e_3-e_4)(e_4+e_5)(e_3+e_4)[(e_5+e_6)\dots(e_{q-1}+e_q)][(e_4+e_5)\dots(e_{q-1}+e_q)])\\
=&2\chi^\varnothing((-2e_4-2e_3e_4e_5)[(e_5+e_6)\dots(e_{q-1}+e_q)][(e_4+e_5)\dots(e_{q-1}+e_q)])\\
=&-4\chi^\varnothing(e_4[(e_5+e_6)\dots(e_{q-1}+e_q)]e_4[(e_5+e_6)\dots(e_{q-1}+e_q)])\\
=&-4\chi^\varnothing([(e_5+e_6)\dots(e_{q-1}+e_q)][(e_5+e_6)\dots(e_{q-1}+e_q)])\\
\end{align*}
and the claim follows by induction.
\end{proof}

\begin{lem}\label{lem:oneeven}
Let $x\in\tilde{N}_p^t\tilde{S}_t$ be of type $(\pi_0,\dots,\pi_{(p-1)/2})$ with $\pi_j$ having at least one
even part for some even $j$. Suppose further that $x\in\tilde{N}_p^t\tilde{S}_t\cap\tilde{A}_{pt}$ or $\pi_j$ does not
have distinct parts. Then $x$ is conjugate to $zx$ in $\tilde{N}_p^t\tilde{S}_t$.
\end{lem}

\begin{proof}
If $x\in\tilde{A}_{pt}$ then we can assume without loss of generality that $x=\prod_{j=1}^lx_j$ is a disjoint cycle
decomposition of $x$ with $x_1=((a,1,\dots,1);\tau)$ for some positive integer $m$ and
$a\in N_p\cap A_p$. In particular $x_1\notin\tilde{A}_{pt}$ and so $\prod_{j=2}^lx_j\notin\tilde{A}_{pt}$.
Therefore $x_1x_2=zx_2x_1$ and hence $x$ is conjugate to $zx$.
\newline
\newline
If $\pi_j$ does not have distinct parts then we may assume by the first paragraph that $x\notin\tilde{A}_{pt}$ and that
$x=\prod_{j=1}^lx_j$ is a disjoint cycle decomposition of $x$ where $x_1=((a,1,\dots,1);\tau_1)$ and
$x_2=((a,1,\dots,1);\tau_2)$, where $a\in\tilde{N}_p\cap\tilde{A}_p$, $\theta_t(\tau_1)=(1,\dots,m)$ and
$\theta_t(\tau_2)=(m+1,\dots,2m)$ for some positive integer $m$. In particular $\prod_{j=3}^lx_j\notin\tilde{A}_{pt}$.
Now let
$u\in\theta_t^{-1}((1,m+1))\dots(m,2m))$. Then $ux_1x_2u^{-1}=x_2x_1$. If $m$ is even then $x_2x_1=zx_1x_2$ and
$u(\prod_{j=3}^lx_j)u^{-1}=\prod_{j=3}^lx_j$. If $m$ is odd then $x_2x_1=x_1x_2$ and
$u(\prod_{j=3}^lx_j)u^{-1}=z\prod_{j=3}^lx_j$. Either way $uxu^{-1}=zx$.
\end{proof}

Now for the rest of the section let $\lambda\in\Delta_t$ and set $\boldsymbol{t}(\lambda)=:(t_0,\dots,t_{(p-1)/2})$. Now
let
\begin{align*}
x_0\in\tilde{N}_p^{t_0},&&s_0\in\tilde{S}_{t_0},&&x'\in\tilde{N}_p^{t-t_0}[t_0],&&s'\in\tilde{S}_{(0,t_1,\dots,t_{
(p-1)/2})}[t_0]
\end{align*}
and set $x:=x_0s_0x's'\in\tilde{N}_p^t\tilde{S}_{\boldsymbol{t}(\lambda)}$. We will identify $\tilde{N}_p^{t-t_0}[t_0]$
with
$\tilde{N}_p^{t-t_0}$ and $s'\in\tilde{S}_{(0,t_1,\dots,t_{(p-1)/2})}[t_0]$ with
$s'\in\tilde{S}_{(0,t_1,\dots,t_{(p-1)/2})}$
via Lemma~\ref{lem:shift}.

\begin{lem}\label{lem:>0=}
If $\lambda\in\Delta_t^-$, $t_0>0$ and $x$ is $p-$regular then $\chi_\lambda^+(x)=\chi_\lambda^-(x)$.
\end{lem}

\begin{proof}
First note that by Lemma~\ref{lem:ordcent} if $x$ is of type $\pi$ then $\pi_0=\varnothing$. If $\sigma(\lambda_0)=-1$
and
$t-t_0$ is even then by Lemma~\ref{lem:t=2} $\chi_\lambda^+(x)\neq\chi_\lambda^-(x)$ only if
$x_0s_0\notin\tilde{A}_{pt_0}$ and
$x's'\in\tilde{A}_{p(t-t_0)}$.
By Lemmas~\ref{lem:semidirect} and~\ref{lem:p'} we note that $\chi_{\lambda_0}^\pm(x_0)\neq0$ only if
$x_0\in\tilde{A}_{pt_0}$ and $s_0$ has only cycles of odd length but this is a contradiction as
$x_0s_0\notin\tilde{A}_{pt_0}$.
\newline
\newline
Next suppose $\sigma(\lambda)=1$ and $t-t_0$
is odd. Similarly to above we can assume $x_0\in\tilde{A}_{pt}$, $s_0\in\tilde{A}_{pt}$ and
$x's'\notin\tilde{A}_{p(t-t_0)}$. By Lemmas~\ref{lem:aa}
and~\ref{lem:p'} we deduce that
$\overline{\overline{\chi}}_{\lambda_0}^{++}(x_0s_0)=\overline{\overline{\chi}}_{\lambda_0}^{-+}(x_0s_0)$
and $\overline{\overline{\chi}}_{\lambda_0}^{+-}(x_0s_0)=\overline{\overline{\chi}}_{\lambda_0}^{--}(x_0s_0)$ and
therefore
$\overline{\chi}_{\lambda_0}^+(x_0s_0)=\overline{\chi}_{\lambda_0}^-(x_0s_0)$ and so
$\chi_\lambda^+(x)=\chi_\lambda^-(x)$.
\end{proof}

\begin{lem}\label{lem:lambda0=}
Suppose $t_0=0$ and $x$ has disjoint cycle decomposition $x=\prod_{j=1}^lx_j$, with
$f(\theta_{pt}(x_1))\in N_p\cap A_p$.
\begin{enumerate}
\item If $\lambda\in\Delta_t^-$ then $\chi_\lambda^+(x)=\chi_\lambda^-(x)$.
\item If $\lambda\in\Delta_t^+$ and $x\in\tilde{A}_{pt}$ then
$\overline{\chi}_\lambda^+(x)=\overline{\chi}_\lambda^-(x)$.
\end{enumerate}
\end{lem}

\begin{proof}
Let $\phi_{\boldsymbol{t}(\lambda)}:\tilde{N}_p^t\tilde{S}_{\boldsymbol{t}(\lambda)}\to\mathcal{C}_t$ be the
homomorphism in Lemma~\ref{lem:exten>0}. Then the coefficient of $e_{[t]}$ in
$\phi_{\boldsymbol{t}(\lambda)}(x)$ is zero. The result now follows from~\ref{lem:cliffrep}.
\end{proof}

\begin{lem}\label{lem:sqrtR}
Let $\lambda\in\Delta_t^-$ and $x\notin\tilde{A}_{pw}$ with $x's'$ having disjoint cycle decomposition
$x's'=\prod_{j=1}^lx_j$. Then $\chi_\lambda^\pm(x)=0$ unless $x_0s_0$ has type
$(\lambda_0,\varnothing,\dots,\varnothing)$ and $f(\theta_{pt}(x_j))\in N_p\backslash A_p$ for $(1\leq j\leq l)$ in
which case $\chi_\lambda^\pm(x)\in\sqrt{p^{l(\lambda_0)}}\mathcal{R}$.
\end{lem}

\begin{proof}
We first note that if $f(\theta_{pt}(x_j))\in N_p\cap A_p$ for some $j$ then by Lemmas~\ref{lem:t=2}
and~\ref{lem:lambda0=} we have that $\chi_\lambda^+(x)=\chi_\lambda^-(x)$ and hence as
$x\notin\tilde{A}_{pw}$, $\chi_\lambda^+(x)=\chi_\lambda^-(x)=0$.
\newline
\newline
Now suppose $\sigma(\lambda_0)=1$ and $t-t_0$ is odd. By Lemmas~\ref{lem:t=2} and~\ref{lem:aa} we
can assume that
$x_0\in\tilde{N}_p^{t_0}\cap\tilde{A}_{pt_0}$, $s_0\in\tilde{A}_{t_0}$ and $x's'\notin\tilde{A}_{p(t-t_0)}$. Now
\begin{align*}
\chi_\lambda^\pm(x)=\pm(\overline{\operatorname{Exten}}^+_{t_0}(x_0\theta_t(s_0))-\overline{\operatorname{Exten}}
^-_{t_0}(x_0\theta_t(s_0)))
(\overline{\xi}_{\lambda_0}^+(s_0)-\overline{\xi}_{\lambda_0}^-(s_0))\chi_{\lambda\backslash\{0\}}^+(x's').
\end{align*}
Therefore by Theorem~\ref{thm:values} $\chi_\lambda^\pm(x)\neq0$ only if $s_0$ is of type $\lambda_0$.
Also by Lemma~\ref{lem:aa}
\begin{align*}
\overline{\operatorname{Exten}}^+_{t_0}(x_0\theta_t(s_0))-\overline{\operatorname{Exten}}^-_{t_0}
(x_0\theta_t(s_0))\neq0
\end{align*}
only if $x_0s_0$ is of type $(\lambda_0,\varnothing,\dots,\varnothing)$ and in this case
by Lemma~\ref{lem:sqrt} we have that
$\overline{\operatorname{Exten}}^+_{t_0}(x_0\theta_t(s_0))-\overline{\operatorname{Exten}}^-_{t_0}
(x_0\theta_t(s_0))\in\sqrt{p^{l(\lambda_0)}}\mathcal{R}$.
\newline
\newline
Next suppose $\sigma(\lambda_0)=-1$ and $t-t_0$ is even. By Lemma~\ref{lem:t=2} we can assume
$x_0s_0\notin\tilde{A}_{pt_0}$ and $x's'\in\tilde{A}_{p(t-t_0)}$. Now
\begin{align*}
\chi_\lambda^\pm(x)=\pm\operatorname{Exten}^+_{t_0}(x_0\theta_t(s_0))\xi_\lambda(s_0)
(\overline{\chi}_{\lambda\backslash\{0\}}^+(x's')-\overline{\chi}_{\lambda\backslash\{0\}}^-(x's')).
\end{align*}
So by Lemma~\ref{lem:semidirect} $\chi_\lambda^\pm(x)\neq0$ only if $x_0\in\tilde{A}_{pt_0}$. In this case
$s_0\notin\tilde{A}_{t_0}$ and hence by Theorem~\ref{thm:values} we can assume $s_0$ is of
type $\lambda_0$. Therefore as in the previous paragraph $x_0s_0$ is of type
$(\lambda_0,\varnothing,\dots,\varnothing)$ and $\chi_\lambda^\pm(x)\in\sqrt{p^{l(\lambda_0)}}\mathcal{R}$.
\end{proof}

\begin{lem}\label{lem:sqrtAR}
Let $\lambda\in\Delta_t^+$ and $x\in\tilde{A}_{pw}$ with $x's'$ having disjoint cycle decomposition
$x's'=\prod_{j=1}^lx_j$. Then $\overline{\chi}_\lambda^+(x)=\overline{\chi}_\lambda^-(x)$ unless $x_0s_0$ has type
$(\lambda_0,\varnothing,\dots,\varnothing)$ and $f(\theta_{pt}(x_j))\in N_p\backslash A_p$ for $(1\leq j\leq l)$ in
which case $\overline{\chi}_\lambda^+(x)-\overline{\chi}_\lambda^-(x)\in\sqrt{p^{l(\lambda_0)}}\mathcal{R}$.
\end{lem}

\begin{proof}
We first note that if $f(\theta_{pt}(x_j))\in N_p\cap A_p$ for some $j$ then by Lemmas~\ref{lem:t=2}
and~\ref{lem:lambda0=} we have that $\overline{\chi}_\lambda^+(x)=\overline{\chi}_\lambda^-(x)$.
\newline
\newline
Now suppose $\sigma(\lambda_0)=1$ and $t-t_0$ is even. By Lemmas~\ref{lem:t=2A} and~\ref{lem:aa} we can assume
$x_0\in\tilde{A}_{pt_0}$ and $s_0\in\tilde{A}_{t_0}$ and so $x's'\in\tilde{A}_{p(t-t_0)}$. Now
\begin{align*}
\overline{\chi}_\lambda^+(x)-\overline{\chi}_\lambda^-(x)=
(\overline{\chi}_{\lambda_0}^+(x_0s_0)-\overline{\chi}_{\lambda_0}^-(x_0s_0))
(\overline{\chi}_{\lambda\backslash\{0\}}^+(x's')-\overline{\chi}_{\lambda\backslash\{0\}}^-(x's'))
\end{align*}
and
\begin{align*}
(\overline{\chi}_{\lambda_0}^+(x_0s_0)-\overline{\chi}_{\lambda_0}^-(x_0s_0))=
(\overline{\operatorname{Exten}}^+_{t_0}(x_0\theta_t(s_0))-
\overline{\operatorname{Exten}}^-_{t_0}(x_0\theta_t(s_0)))
(\overline{\xi}^+(s_0)-\overline{\xi}^-(s_0)).
\end{align*}
Therefore by Theorem~\ref{thm:values} $\overline{\chi}_\lambda^+(x)-\overline{\chi}_\lambda^-(x)\neq0$ only if
$s_0$ is of type $\lambda_0$ and the result then follows by Lemma~\ref{lem:aa}.
\newline
\newline
Now suppose $\sigma(\lambda_0)=-1$ and $t-t_0$ is odd. By Lemmas~\ref{lem:t=2A} and~\ref{lem:semidirect} we can
assume that $x_0\tilde{A}_{pt_0}$ and $s_0\notin\tilde{A}_{t_0}$ and so $x's'\notin\tilde{A}_{p(t-t_0)}$. Then
\begin{align*}
\overline{\chi}_\lambda^+(x)-\overline{\chi}_\lambda^-(x)=
2i\chi_{\lambda_0}^+(x_0s_0)\chi_{\lambda\backslash\{0\}}^+(x's')
\end{align*}
Therefore by Theorem~\ref{thm:values} $\overline{\chi}_\lambda^+(x)-\overline{\chi}_\lambda^-(x)\neq0$ only if
$s_0$ is of type $\lambda_0$ and the result then follows by Lemma~\ref{lem:semidirect}.
\end{proof}

\begin{lem}\label{lem:lambda0p}
Let $\lambda\in\Delta_t^+$ with $t_0>0$ and let
$x\in\tilde{N}_p^t\tilde{S}_{\boldsymbol{t}(\lambda)}\cap\tilde{A}_{pt}$ be $p-$regular. Then
$\overline{\chi}_\lambda^+(x)=\overline{\chi}_\lambda^-(x)$.
\end{lem}

\begin{proof}
Let $x=x_0x'$ with $x_0\in\tilde{N}_p^{t_0}\tilde{S}_{t_0}$ and
$x'\in\tilde{N}_p^{t-t_0}(\tilde{S}_{t_1}\dots\tilde{S}_{t_{(p-1)/2}})$. Now let $x_0=\prod_{j=1}^lx_j$ be a
disjoint cycle decomposition. Then by Lemma~\ref{lem:ordcent} $f(\theta_{pt}(x_j))$ is $p-$regular for all $j$.
\newline
\newline
First suppose $\sigma(\lambda_0)=1$ and $t-t_0$ is even. Then by Lemma~\ref{lem:t=2A} we may assume
$x_0\in\tilde{A}_{pt_0}$ and $x'\in\tilde{A}_{p(t-t_0)}$. By Lemmas~\ref{lem:aa} and~\ref{lem:p'}
$\overline{\chi}_{\lambda_0}^+(x_0)=\overline{\chi}_{\lambda_0}^-(x_0)$ and therefore that
$\overline{\chi}_\lambda^+(x)=\overline{\chi}_\lambda^-(x)$.
\newline
\newline
Next suppose $\sigma(\lambda_0)=-1$ and $t-t_0$ is odd. By Lemma~\ref{lem:t=2A} the result is clear unless
$x_0\notin\tilde{A}_{pt_0}$ and $x'\notin\tilde{A}_{p(t-t_0)}$. However, again by Lemma~\ref{lem:aa} in this
case $\chi_{\lambda_0}^+(x_0)=0$ and hence $\overline{\chi}_\lambda^+(x)=\overline{\chi}_\lambda^-(x)=0$.
\end{proof}

\section{Murnaghan-Nakayama rules}\label{sec:MN}

By Murnaghan-Nakayama rule we mean the maps $r^{x_S}$ in part (4) of Definition~\ref{def:MN}. For the original
Murnaghan-Nakayama rule for the symmetric group (see~\cite[Theorem 2.4.7]{jamker1981}). In this section we present the
corresponding rule for the double covers of the symmetric and alternating groups as proved by M. Cabanes and O. Brunat
and
J. Gramain respectively. We then go on to prove that $\tilde{N}_p^t\tilde{S}_t$ also possesses a Murnaghan-Nakayama
rule.
\newline
\newline

\subsection{Murnaghan-Nakayama rule for \texorpdfstring{$\tilde{S}_n$}{TEXT} and \texorpdfstring{$\tilde{A}_n$}{TEXT}}

\begin{thm}\label{thm:MNSn}\cite[Thoerem 20]{cabane1988}
Let $n$ and $q$ be positive integers with $q$ odd and $q\leq n$. We write $\tau=o((1,\dots,q))\in\tilde{S}_n$. Let
$x\in\tilde{S}_{n-q}[q]$ which we identify with $\tilde{S}_{n-q}$ via Lemma~\ref{lem:shift}. Then
\begin{align*}
\xi^{(\pm)}_\lambda(\tau x)=\sum_{\mu\in
M_{\overline{q}}(\lambda),\sigma(\mu)=1}a(\xi^{(\pm)}_\lambda,\xi_\mu)\xi_\mu(x)+\sum_{\mu\in
M_{\overline{q}}(\lambda),\sigma(\mu)=-1}
a(\xi^{(\pm)}_\lambda,\xi^+_\mu)\xi^+_\mu(x)+a(\xi^{(\pm)}_\lambda,\xi^-_\mu)\xi^-_\mu(x),
\end{align*}
where
\begin{align*}
a(\xi^{(\varepsilon)}_\lambda,\xi^{(\eta)}_\mu):=
\begin{cases}
(-1)^\frac{q^2-1}{8}\alpha^\lambda_\mu&\text{ if }\sigma(\mu)=1,\\
\frac{1}{2}(-1)^\frac{q^2-1}{8}\alpha^\lambda_\mu&\text{ if }\sigma(\lambda)=1,\sigma(\mu)=-1,\\
\frac{1}{2}(-1)^\frac{q^2-1}{8}(\alpha^\lambda_\mu+\varepsilon\eta i^\frac{q-1}{2}\sqrt{q})&\text{ if
}\sigma(\lambda)=\sigma(\mu)=-1,
\end{cases}
\end{align*}
and
\begin{align*}
\alpha^\lambda_\mu=(-1)^{L(b)}2^{m(b)},
\end{align*}
where $L(b)$ is the leg length of the $q-$bar $b$ removed from $\lambda$ to get $\mu$, and
\begin{align*}
m(b):=
\begin{cases}
1 &\text{ if } \sigma(\lambda)=1,\sigma(\mu)=-1,\\
0 &\text{ otherwise.}
\end{cases}
\end{align*}
\end{thm}

There is also a natural labelling of the characters such that the following theorem for $\tilde{A}_n$ holds. We also
adopt
this labelling for the remainder of this paper.

\begin{thm}\label{thm:MNAn}\cite[Theorem 4.10]{brugra2014}
Let $\tau$ and $x$ be as above with $x\in\tilde{A}_{n-q}$. Then
\begin{align*}
\overline{\xi}^{(\pm)}_\lambda(\tau x)=\sum_{\mu\in M_{\overline{q}}(\lambda),\sigma(\mu)=-1}
a(\overline{\xi}^{(\pm)}_\lambda,\overline{\xi}_\mu)\overline{\xi}_\mu(x)+\sum_{\mu\in
M_{\overline{q}}(\lambda),\sigma(\mu)=1}
a(\overline{\xi}^{(\pm)}_\lambda,\overline{\xi}^+_\mu)\overline{\xi}^+_\mu(x)+a(\overline{\xi}^{(\pm)}_\lambda,\overline
{\xi}^-_\mu)\overline{\xi}^-_\mu(x).
\end{align*}
where
\begin{align*}
a(\overline{\xi}^{(\varepsilon)}_\lambda,\overline{\xi}^{(\eta)}_\mu):=
\begin{cases}
(-1)^\frac{q^2-1}{8}\alpha^\lambda_\mu&\text{ if }\sigma(\lambda)=-1,\\
\frac{1}{2}(-1)^\frac{q^2-1}{8}\alpha^\lambda_\mu&\text{ if }\sigma(\lambda)=1,\sigma(\mu)=-1,\\
\frac{1}{2}(-1)^\frac{q^2-1}{8}(\alpha^\lambda_\mu+\varepsilon\eta i^\frac{q-1}{2}\sqrt{q})&\text{ if
}\sigma(\lambda)=\sigma(\mu)=1,
\end{cases}
\end{align*}
and the $\alpha^\lambda_\mu$ are as in Theorem~\ref{thm:MNSn}.
\end{thm}

\subsection{Murnaghan-Nakayama rule for \texorpdfstring{$\tilde{N}_p^t\tilde{S}_t$}{TEXT}}

Let $t$ be a positive integer and $\lambda\in\Delta_t$ with $\boldsymbol{t}(\lambda)=(t_0,\dots,t_{(p-1)/2})$. We set
\begin{align*}
\lambda\backslash\{j\}=(\lambda_0,\dots,\lambda_{j-1},\varnothing,\lambda_{j+1},\dots,\lambda_{(p-1)/2}).
\end{align*}
If $\mu$ is any partition then we set
\begin{align*}
(\mu,\lambda\backslash\{j\})=(\lambda_0,\dots,\lambda_{j-1},\mu,\lambda_{j+1},\dots,\lambda_{(p-1)/2}).
\end{align*}

\begin{thm}\label{thm:MNj=0}
Let $q$ be a positive integer with $q$ odd and $q\leq t$ and set $\tau=o((1,\dots,q))$ and
$g:=(\boldsymbol{a};\tau)\in\tilde{N}_p^t\tilde{S}_t$, where
$\boldsymbol{a}=(a,1,\dots,1)$ and
\begin{align*}
\overline{\zeta}^+_0(a)=\frac{-1+i^\frac{p-1}{2}\sqrt{p}}{2}.
\end{align*}
Let $x\in\tilde{N}_p^{t-q}\tilde{S}_{t-q}[q]$ which we identify with $\tilde{N}_p^{t-q}\tilde{S}_{t-q}$ as in
Lemma~\ref{lem:shift}. Now suppose $gx\in\tilde{N}_p^t\tilde{S}_{\boldsymbol{t}(\lambda)}$.
\begin{enumerate}
\item If $\sigma(\lambda_0)=1$ and $t-t_0$ is even then
\begin{align*}
\chi_\lambda(gx)=-\bigg{(}&\sum_{\begin{smallmatrix}
\mu\in M_{\overline{q}}(\lambda_0)\\
\sigma(\mu)=1
\end{smallmatrix}}
a(\xi_{\lambda_0},\xi_\mu)\chi_{(\mu,\lambda\backslash\{0\})}(x)+\\
&\sum_{\begin{smallmatrix}
\mu\in M_{\overline{q}}(\lambda_0)\\
\sigma(\mu)=-1
\end{smallmatrix}}
a(\xi_{\lambda_0},\xi^\pm_\mu)(\chi_{(\mu,\lambda\backslash\{0\})}^+(x)+\chi_{(\mu,\lambda\backslash\{0\})}
^-(x))\bigg{)}.
\end{align*}
\item If $\sigma(\lambda_0)=-1$ and $t-t_0$ is even then
\begin{align*}
&\chi_\lambda^+(gx)
=-\sum_{\begin{smallmatrix}
\mu\in M_{\overline{q}}(\lambda_0)\\
\sigma(\mu)=1
\end{smallmatrix}}
a(\xi^\pm_{\lambda_0},\xi_\mu)\chi_{(\mu,\lambda\backslash\{0\})}(x)+\\
\sum_{\begin{smallmatrix}
\mu\in M_{\overline{q}}(\lambda_0)\\
\sigma(\mu)=-1
\end{smallmatrix}}
\bigg{[}&\bigg{(}a(\xi^+_{\lambda_0},\xi^+_\mu)\bigg{(}\frac{-1+i^\frac{p-1}{2}\sqrt{p}}{2}\bigg{)}+
a(\xi^+_{\lambda_0},\xi^-_\mu)\bigg{(}\frac{-1-i^\frac{p-1}{2}\sqrt{p}}{2}\bigg{)}\bigg{)}\chi_{(\mu,\lambda\backslash\{
0\})}^+(x)+\\
&\bigg{(}a(\xi^+_{\lambda_0},\xi^+_\mu)\bigg{(}\frac{-1-i^\frac{p-1}{2}\sqrt{p}}{2}\bigg{)}+
a(\xi^+_{\lambda_0},\xi^-_\mu)\bigg{(}\frac{-1+i^\frac{p-1}{2}\sqrt{p}}{2}\bigg{)}\bigg{)}\chi_{(\mu,\lambda\backslash\{
0\})}^-(x)\bigg{]}.
\end{align*}
\item If $\sigma(\lambda_0)=1$ and $t-t_0$ is odd then
\begin{align*}
&\chi_\lambda^+(gx)=
-\sum_{\begin{smallmatrix}
\mu\in M_{\overline{q}}(\lambda_0)\\
\sigma(\mu)=-1
\end{smallmatrix}}
a(\overline{\xi}^\pm_{\lambda_0},\overline{\xi}_\mu)\chi_{(\mu,\lambda\backslash\{0\})}(x)+\\
\sum_{\begin{smallmatrix}
\mu\in M_{\overline{q}}(\lambda_0)\\
\sigma(\mu)=1
\end{smallmatrix}}
\bigg{[}&\bigg{(}a(\overline{\xi}^+_{\lambda_0},\overline{\xi}^+_\mu)\bigg{(}\frac{-1+i^\frac{p-1}{2}\sqrt{p}}{2}\bigg{)
}+a(\overline{\xi}^+_{\lambda_0},\overline{\xi}^-_\mu)
\bigg{(}\frac{-1-i^\frac{p-1}{2}\sqrt{p}}{2}\bigg{)}\bigg{)}\chi_{(\mu,\lambda\backslash\{0\})}^+(x)+\\
&\bigg{(}a(\overline{\xi}^+_{\lambda_0},\overline{\xi}^+_\mu)\bigg{(}\frac{-1-i^\frac{p-1}{2}\sqrt{p}}{2}\bigg{)}
+a(\overline{\xi}^+_{\lambda_0},\overline{\xi}^-_\mu)
\bigg{(}\frac{-1+i^\frac{p-1}{2}\sqrt{p}}{2}\bigg{)}\bigg{)}\chi_{(\mu,\lambda\backslash\{0\})}^-(x)\bigg{]}.
\end{align*}
\item If $\sigma(\lambda_0)=-1$ and $t-t_0$ is odd then
\begin{align*}
\chi_\lambda(gx)
=-\bigg{(}&\sum_{\begin{smallmatrix}
\mu\in M_{\overline{q}}(\lambda_0)\\
\sigma(\mu)=1
\end{smallmatrix}}
a(\xi^\pm_{\lambda_0},\xi_\mu)(\chi_{(\mu,\lambda\backslash\{0\})}^+(x)+\chi_{(\mu,\lambda\backslash\{0\})}^-(x))+\\
&\sum_{\begin{smallmatrix}
\mu\in M_{\overline{q}}(\lambda_0)\\
\sigma(\mu)=-1
\end{smallmatrix}}
(a(\xi^+_{\lambda_0},\xi^+_\mu)+a(\xi^+_{\lambda_0},\xi^-_\mu))\chi_{(\mu,\lambda\backslash\{0\})}(x)\bigg{)}.
\end{align*}
\end{enumerate}
\end{thm}

\begin{proof}
Let
\begin{align*}
x_0\in\tilde{N}_p^{t_0-q}[q],&&s_0\in\tilde{S}_{t_0-q}[q],&&x'\in\tilde{N}_p^{t-t_0}[t_0],&&s'\in\tilde{S}_{(0,t_1,\dots
,t_{(p-1)/2})}[t_0],
\end{align*}
with $x=x_0s_0x's'$. As always we identify all the above subgroups with their non-shifted counterpart via
Lemma~\ref{lem:shift}.

\begin{enumerate}
\item By Lemma~\ref{lem:t=2}
\begin{align*}
\chi_\lambda(gx)=
\begin{cases}
\chi_{\lambda_0}(gx_0s_0)\overline{\chi}_{\lambda\backslash\{0\}}^+(x's')+
\chi_{\lambda_0}(gx_0s_0)\overline{\chi}_{\lambda\backslash\{0\}}^-(x's')&\text{ if }x's'\in\tilde{A}_{p(t-t_0)},\\
0 &\text{ otherwise.}
\end{cases}
\end{align*}
Now by Lemma~\ref{lem:semidirect} and Theorem~\ref{thm:MNSn}
\begin{align*}
&\chi_{\lambda_0}(gx_0s_0)=
\xi_{\lambda_0}(\tau s_0)\operatorname{Exten}^+_{t_0}(\boldsymbol{a}\theta_{t_0}(\tau)x_0\theta_{t_0}(s_0))\\
=&\bigg{(}\sum_{\begin{smallmatrix}
\mu\in M_{\overline{q}}(\lambda_0)\\
\sigma(\mu)=1
\end{smallmatrix}}
a(\xi_{\lambda_0},\xi_\mu)\xi_\mu(s_0)+
\sum_{\begin{smallmatrix}
\mu\in M_{\overline{q}}(\lambda_0)\\
\sigma(\mu)=-1
\end{smallmatrix}}
a(\xi_{\lambda_0},\xi^\pm_\mu)(\xi^+_\mu(s_0)+\xi^-_\mu(s_0))\bigg{)}\\
&\bigg{(}\frac{-1+i^\frac{p-1}{2}\sqrt{p}}{2}+\epsilon(s_0)\frac{-1-i^\frac{p-1}{2}\sqrt{p}}{2}\bigg{)}
\operatorname{Exten}^+_{t_0-q}(x_0\theta_{t_0-q}(s_0))\\
=&-\bigg{(}\sum_{\begin{smallmatrix}
\mu\in M_{\overline{q}}(\lambda_0)\\
\sigma(\mu)=1
\end{smallmatrix}}
a(\xi_{\lambda_0},\xi_\mu)\chi_\mu(x_0s_0)+
\sum_{\begin{smallmatrix}
\mu\in M_{\overline{q}}(\lambda_0)\\
\sigma(\mu)=-1
\end{smallmatrix}}a(\xi_{\lambda_0},\xi^\pm_\mu)(\chi^+_\mu(x_0s_0)+\chi^-_\mu(x_0s_0))\bigg{)}.
\end{align*}
Therefore, by Lemma~\ref{lem:t=2}, we have that
\begin{align*}
\chi_\lambda(gx)=-\bigg{(}&\sum_{\begin{smallmatrix}
\mu\in M_{\overline{q}}(\lambda_0)\\
\sigma(\mu)=1
\end{smallmatrix}}
a(\xi_{\lambda_0},\xi_\mu)\chi_{(\mu,\lambda\backslash\{0\})}(x)+\\
&\sum_{\begin{smallmatrix}
\mu\in M_{\overline{q}}(\lambda_0)\\
\sigma(\mu)=-1
\end{smallmatrix}}
a(\xi_{\lambda_0},\xi^\pm_\mu)(\chi_{(\mu,\lambda\backslash\{0\})}^+(x)+\chi_{(\mu,\lambda\backslash\{0\})}
^-(x))\bigg{)}.
\end{align*}
\item By Lemma~\ref{lem:t=2}
\begin{align*}
\chi^+_\lambda(gx)=
\begin{cases}
\chi_{\lambda_0}^+(g x_0s_0)\overline{\chi}_{\lambda\backslash\{0\}}^+(x's')+
\chi_{\lambda_0}^-(g x_0s_0)\overline{\chi}_{\lambda\backslash\{0\}}^-(x's')&\text{ if
}x's'\in\tilde{A}_{p(t-t_0)},\\
0 &\text{ otherwise.}
\end{cases}
\end{align*}
Now by Lemma~\ref{lem:semidirect} and Theorem~\ref{thm:MNSn}
\begin{align*}
&\chi_{\lambda_0}^+(g
x_0s_0)=\xi^+_{\lambda_0}(\tau s_0)\operatorname{Exten}^+_{t_0}(\boldsymbol{a}\theta_{t_0}(\tau)x_0\theta_{t_0}(s_0))\\
=&\bigg{(}\sum_{\begin{smallmatrix}
\mu\in M_{\overline{q}}(\lambda_0)\\
\sigma(\mu)=1
\end{smallmatrix}}
a(\xi^\pm_{\lambda_0},\xi_\mu)\xi_\mu(s_0)+
\sum_{\begin{smallmatrix}
\mu\in M_{\overline{q}}(\lambda_0)\\
\sigma(\mu)=-1
\end{smallmatrix}}
(a(\xi^+_{\lambda_0},\xi^+_\mu)\xi^+_\mu(s_0)+a(\xi^+_{\lambda_0},\xi^-_\mu)\xi^-_\mu(s_0))\bigg{)}\\
&\bigg{(}\frac{-1+i^\frac{p-1}{2}\sqrt{p}}{2}+\epsilon(s_0)\frac{-1-i^\frac{p-1}{2}\sqrt{p}}{2}\bigg{)}
\operatorname{Exten}^+_{t_0-q}(x_0\theta_{t_0-q}(s_0))\\
=&-\sum_{\begin{smallmatrix}
\mu\in M_{\overline{q}}(\lambda_0)\\
\sigma(\mu)=1
\end{smallmatrix}}
a(\xi^\pm_{\lambda_0},\xi_\mu)\chi_\mu(x_0s_0)+\\
\sum_{\begin{smallmatrix}
\mu\in M_{\overline{q}}(\lambda_0)\\
\sigma(\mu)=-1
\end{smallmatrix}}
\bigg{[}&\bigg{(}a(\xi^+_{\lambda_0},\xi^+_\mu)\bigg{(}\frac{-1+i^\frac{p-1}{2}\sqrt{p}}{2}\bigg{)}+
a(\xi^+_{\lambda_0},\xi^-_\mu)\bigg{(}\frac{-1-i^\frac{p-1}{2}\sqrt{p}}{2}\bigg{)}\bigg{)}\chi_\mu^+(x_0s_0)+\\
&\bigg{(}a(\xi^+_{\lambda_0},\xi^+_\mu)\bigg{(}\frac{-1-i^\frac{p-1}{2}\sqrt{p}}{2}\bigg{)}+
a(\xi^+_{\lambda_0},\xi^-_\mu)\bigg{(}\frac{-1+i^\frac{p-1}{2}\sqrt{p}}{2}\bigg{)}\bigg{)}\chi_\mu^-(x_0s_0)\bigg{]}.
\end{align*}
Therefore if $x's'\in\tilde{A}_{p(t-t_0)}$ we have that
\begin{align*}
\chi^+_\lambda(gx)=
-\sum_{\begin{smallmatrix}
\mu\in M_{\overline{q}}(\lambda_0)\\
\sigma(\mu)=1
\end{smallmatrix}}
&a(\xi^\pm_{\lambda_0},\xi_\mu)\chi_\mu(x_0s_0)(\overline{\chi}_{\lambda\backslash\{0\}}^+(x's')+\overline{\chi}
_{\lambda\backslash\{0\}}^-(x's'))+\\
\sum_{\begin{smallmatrix}
\mu\in M_{\overline{q}}(\lambda_0)\\
\sigma(\mu)=-1
\end{smallmatrix}}
&\bigg{[}\bigg{(}a(\xi^+_{\lambda_0},\xi^+_\mu)\bigg{(}\frac{-1+i^\frac{p-1}{2}\sqrt{p}}{2}\bigg{)}+
a(\xi^+_{\lambda_0},\xi^-_\mu)\bigg{(}\frac{-1-i^\frac{p-1}{2}\sqrt{p}}{2}\bigg{)}\bigg{)}\\
&(\chi_\mu^+(x_0s_0)\overline{\chi}_{\lambda\backslash\{0\}}^+(x's')+\chi_\mu^-(x_0s_0)\overline{\chi}_{
\lambda\backslash\{0\}}^-(x's'))+\\
&\bigg{(}a(\xi^+_{\lambda_0},\xi^+_\mu)\bigg{(}\frac{-1-i^\frac{p-1}{2}\sqrt{p}}{2}\bigg{)}+
a(\xi^+_{\lambda_0},\xi^-_\mu)\bigg{(}\frac{-1+i^\frac{p-1}{2}\sqrt{p}}{2}\bigg{)}\bigg{)}\\
&(\chi_\mu^+(x_0s_0)\overline{\chi}_{\lambda\backslash\{0\}}^-(x's')+\chi_\mu^-(x_0s_0)\overline{\chi}_{
\lambda\backslash\{0\}}^+(x's'))\bigg{]},
\end{align*}
and so by Lemma~\ref{lem:t=2}
\begin{align*}
\chi^+_\lambda(gx)=-\sum_{\begin{smallmatrix}
\mu\in M_{\overline{q}}(\lambda_0)\\
\sigma(\mu)=1
\end{smallmatrix}}
&a(\xi^\pm_{\lambda_0},\xi_\mu)\chi_{(\mu,\lambda\backslash\{0\})}(x)+\\
\sum_{\begin{smallmatrix}
\mu\in M_{\overline{q}}(\lambda_0)\\
\sigma(\mu)=-1
\end{smallmatrix}}
\bigg{[}&\bigg{(}a(\xi^+_{\lambda_0},\xi^+_\mu)\bigg{(}\frac{-1+i^\frac{p-1}{2}\sqrt{p}}{2}\bigg{)}+
a(\xi^+_{\lambda_0},\xi^-_\mu)\bigg{(}\frac{-1-i^\frac{p-1}{2}\sqrt{p}}{2}\bigg{)}\bigg{)}\chi_{(\mu,\lambda\backslash\{
0\})}^+(x)+\\
&\bigg{(}a(\xi^+_{\lambda_0},\xi^+_\mu)\bigg{(}\frac{-1-i^\frac{p-1}{2}\sqrt{p}}{2}\bigg{)}+
a(\xi^+_{\lambda_0},\xi^-_\mu)\bigg{(}\frac{-1+i^\frac{p-1}{2}\sqrt{p}}{2}\bigg{)}\bigg{)}\chi_{(\mu,\lambda\backslash\{
0\})}^-(x)\bigg{]}.
\end{align*}
\item By Lemma~\ref{lem:t=2}
\begin{align*}
\chi_\lambda^+(gx)=
\begin{cases}
\overline{\chi}_{\lambda_0}^+(gx_0s_0)\chi_{\lambda\backslash\{0\}}^+(x's')+
\overline{\chi}_{\lambda_0}^-(gx_0s_0)\chi_{\lambda\backslash\{0\}}^-(x's')&\text{ if }x_0s_0\in\tilde{A}_{p(t_0-q)},\\
0 &\text{ otherwise,}
\end{cases}
\end{align*}
and if $x_0s_0\in\tilde{A}_{p(t_0-q)}$ then by Lemma~\ref{lem:t=2A}
\begin{align*}
\overline{\chi}_{\lambda_0}^+(gx_0s_0)=
\begin{cases}
\overline{\xi}^+_{\lambda_0}(\tau
s_0)\overline{\operatorname{Exten}}^+_{t_0}(\boldsymbol{a}\theta_{t_0}(\tau)x_0\theta_{t_0}(s_0))+
&\text{ if }x_0\in\tilde{A}_{p(t_0-q)}\text{ and }s_0\in\tilde{A}_{t_0-q},\\
\overline{\xi}^-_{\lambda_0}(\tau
s_0)\overline{\operatorname{Exten}}^-_{t_0}(\boldsymbol{a}\theta_{t_0}(\tau)x_0\theta_{t_0}(s_0))&\\
0 &\text{ otherwise.}
\end{cases}
\end{align*}
Then by Lemma~\ref{lem:aa} and Theorem~\ref{thm:MNAn} if $x_0\in\tilde{A}_{p(t_0-q)}$ and $s_0\in\tilde{A}_{t_0-q}$
\begin{align*}
&\overline{\xi}^+_{\lambda_0}(\tau s_0)\overline{\operatorname{Exten}}^+_{t_0}(\boldsymbol{a}\theta_{t_0}(\tau
)x_0\theta_{t_0}(s_0))\\
=&\bigg{(}\sum_{\begin{smallmatrix}
\mu\in M_{\overline{q}}(\lambda_0)\\
\sigma(\mu)=1
\end{smallmatrix}}
a(\overline{\xi}^+_{\lambda_0},\overline{\xi}^+_\mu)\overline{\xi}^+_\mu(s_0)+a(\overline{\xi}^+_{\lambda_0},
\overline{\xi}^-_\mu)\overline{\xi}^-_\mu(s_0)+
\sum_{\begin{smallmatrix}
\mu\in M_{\overline{q}}(\lambda_0)\\
\sigma(\mu)=-1
\end{smallmatrix}}
a(\overline{\xi}^\pm_{\lambda_0},\overline{\xi}_\mu)\overline{\xi}_\mu(s_0)\bigg{)}\\
&\bigg{[}\bigg{(}\frac{-1+i^\frac{p-1}{2}\sqrt{p}}{2}\bigg{)}\overline{\operatorname{Exten}}^+_{t_0-q}(x_0\theta_{t_0-q}
(s_0))+
\bigg{(}\frac{-1-i^\frac{p-1}{2}\sqrt{p}}{2}\bigg{)}\overline{\operatorname{Exten}}^-_{t_0-q}(x_0\theta_{t_0-q}
(s_0))\bigg{]}\\
=&\bigg{(}\sum_{\begin{smallmatrix}
\mu\in M_{\overline{q}}(\lambda_0)\\
\sigma(\mu)=1
\end{smallmatrix}}
a(\overline{\xi}^+_{\lambda_0},\overline{\xi}^+_\mu)\overline{\overline{\chi}}_\mu^{++}(x_0s_0)+
a(\overline{\xi}^+_{\lambda_0},\overline{\xi}^-_\mu)\overline{\overline{\chi}}_\mu^{+-}(x_0s_0)+\\
&\sum_{\begin{smallmatrix}
\mu\in M_{\overline{q}}(\lambda_0)\\
\sigma(\mu)=-1
\end{smallmatrix}}
a(\overline{\xi}^\pm_{\lambda_0},\overline{\xi}_\mu)\overline{\overline{\chi}}_\mu^+(x_0s_0)\bigg{)}\bigg{(}\frac{
-1+i^\frac{p-1}{2}\sqrt{p}}{2}\bigg{)}+\\
&\bigg{(}\sum_{\begin{smallmatrix}
\mu\in M_{\overline{q}}(\lambda_0)\\
\sigma(\mu)=1
\end{smallmatrix}}
a(\overline{\xi}^+_{\lambda_0},\overline{\xi}^+_\mu)\overline{\overline{\chi}}_\mu^{-+}(x_0s_0)+
a(\overline{\xi}^+_{\lambda_0},\overline{\xi}^-_\mu)\overline{\overline{\chi}}_\mu^{--}(x_0s_0)+\\
&\sum_{\begin{smallmatrix}
\mu\in M_{\overline{q}}(\lambda_0)\\
\sigma(\mu)=-1
\end{smallmatrix}}
a(\overline{\xi}^\pm_{\lambda_0},\overline{\xi}_\mu)\overline{\overline{\chi}}_\mu^-(x_0s_0)\bigg{)}\bigg{(}\frac{
-1-i^\frac{p-1}{2}\sqrt{p}}{2}\bigg{)}.
\end{align*}
Therefore,
\begin{align*}
\chi_{\lambda}^+(gx)
=&\sum_{\begin{smallmatrix}
\eta_1,\eta_2,\eta_3,\eta_4,\eta_5\in\{\pm1\}\\
\eta_1\eta_2\eta_3\eta_4\eta_5=1
\end{smallmatrix}}\bigg{[}\bigg{(}\sum_{\begin{smallmatrix}
\mu\in M_{\overline{q}}(\lambda_0)\\
\sigma(\mu)=1
\end{smallmatrix}}
a(\overline{\xi}^+_{\lambda_0},\overline{\xi}^{\eta_1}_\mu)\overline{\overline{\chi}}_\mu^{\eta_2\eta_3}(x_0s_0)
\chi_{\lambda\backslash\{0\}}^{\eta_4}(x's')+\\
&\frac{1}{2}\sum_{\begin{smallmatrix}
\mu\in M_{\overline{q}}(\lambda_0)\\
\sigma(\mu)=-1
\end{smallmatrix}}
a(\overline{\xi}^\pm_{\lambda_0},\overline{\xi}_\mu)\overline{\overline{\chi}}_\mu^{\eta_2}(x_0s_0)
\chi^{\eta_4}_{\lambda\backslash\{0\}}(x's')\bigg{)}\bigg{(}\frac{-1+\eta_5i^\frac{p-1}{2}\sqrt{p}}{2}
\bigg{)}\bigg{]}
\end{align*}
if $x_0\in\tilde{A}_{p(t_0-q)}$ and $s_0\in\tilde{A}_{t_0-q}$ and $0$ otherwise. Therefore by Lemma~\ref{lem:t=2}
\begin{align*}
\chi_{\lambda}^+(gx)=
-\sum_{\begin{smallmatrix}
\mu\in M_{\overline{q}}(\lambda_0)\\
\sigma(\mu)=-1
\end{smallmatrix}}
&a(\overline{\xi}^\pm_{\lambda_0},\overline{\xi}_\mu)\chi_{(\mu,\lambda\backslash\{0\})}(x)+\\
\sum_{\begin{smallmatrix}
\mu\in M_{\overline{q}}(\lambda_0)\\
\sigma(\mu)=1
\end{smallmatrix}}
&\bigg{[}\bigg{(}\bigg{(}\frac{-1+i^\frac{p-1}{2}\sqrt{p}}{2}\bigg{)}a(\overline{\xi}^+_{\lambda_0},\overline{\xi}
^+_\mu)+\bigg{(}\frac{-1-i^\frac{p-1}{2}\sqrt{p}}{2}\bigg{)}
a(\overline{\xi}^+_{\lambda_0},\overline{\xi}^-_\mu)\bigg{)}\chi_{(\mu,\lambda\backslash\{0\})}^+(x)+\\
&\bigg{(}\bigg{(}\frac{-1-i^\frac{p-1}{2}\sqrt{p}}{2}\bigg{)}a(\overline{\xi}^+_{\lambda_0},\overline{\xi}^+_\mu)+\bigg
{(}\frac{-1+i^\frac{p-1}{2}\sqrt{p}}{2}\bigg{)}
a(\overline{\xi}^+_{\lambda_0},\overline{\xi}^-_\mu)\bigg{)}\chi_{(\mu,\lambda\backslash\{0\})}^+(x)\bigg{]}.
\end{align*}
\item By Lemma~\ref{lem:t=2}
\begin{align*}
\chi_\lambda(gx)=
\begin{cases}
\chi_{\lambda_0}^+(g x_0s_0)\overline{\chi}_{\lambda\backslash\{0\}}(x's')+
\chi_{\lambda_0}^-(g x_0s_0)\overline{\chi}_{\lambda\backslash\{0\}}(x's')&\text{ if
}x's'\in\tilde{A}_{t-t_0},\\
0 &\text{ otherwise.}
\end{cases}
\end{align*}
Now by Lemma~\ref{lem:semidirect} and Theorem~\ref{thm:MNSn}
\begin{align*}
&\chi_{\lambda_0}^+(g
x_0s_0)=\xi^+_{\lambda_0}(\tau s_0)\operatorname{Exten}^+_{t_0}(\boldsymbol{a}\theta_{t_0}(\tau )x_0\theta_{t_0}(s_0))\\
=&\bigg{(}\sum_{\begin{smallmatrix}
\mu\in M_{\overline{q}}(\lambda_0)\\
\sigma(\mu)=1
\end{smallmatrix}}
a(\xi^\pm_{\lambda_0},\xi_\mu)\xi_\mu(s_0)+
\sum_{\begin{smallmatrix}
\mu\in M_{\overline{q}}(\lambda_0)\\
\sigma(\mu)=-1
\end{smallmatrix}}
(a(\xi^+_{\lambda_0},\xi^+_\mu)\xi^+_\mu(s_0)+a(\xi^+_{\lambda_0},\xi^-_\mu)\xi^-_\mu(s_0))\bigg{)}\\
&\bigg{(}\frac{-1+i^\frac{p-1}{2}\sqrt{p}}{2}+\epsilon(s_0)\frac{-1-i^\frac{p-1}{2}\sqrt{p}}{2}\bigg{)}
\operatorname{Exten}^+_{t_0-q}(x_0\theta_{t_0-q}(s_0))\\
=&-\sum_{\begin{smallmatrix}
\mu\in M_{\overline{q}}(\lambda_0)\\
\sigma(\mu)=1
\end{smallmatrix}}
a(\xi^\pm_{\lambda_0},\xi_\mu)\chi_\mu(x_0s_0)+\\
&\sum_{\begin{smallmatrix}
\mu\in M_{\overline{q}}(\lambda_0)\\
\sigma(\mu)=-1
\end{smallmatrix}}
\bigg{[}\bigg{(}a(\xi^+_{\lambda_0},\xi^+_\mu)\bigg{(}\frac{-1+i^\frac{p-1}{2}\sqrt{p}}{2}\bigg{)}+
a(\xi^+_{\lambda_0},\xi^-_\mu)\bigg{(}\frac{-1-i^\frac{p-1}{2}\sqrt{p}}{2}\bigg{)}\bigg{)}\chi_\mu^+(x_0s_0)+\\
&\bigg{(}a(\xi^+_{\lambda_0},\xi^+_\mu)\bigg{(}\frac{-1-i^\frac{p-1}{2}\sqrt{p}}{2}\bigg{)}+
a(\xi^+_{\lambda_0},\xi^-_\mu)\bigg{(}\frac{-1+i^\frac{p-1}{2}\sqrt{p}}{2}\bigg{)}\bigg{)}\chi_\mu^-(x_0s_0))\bigg{]}.
\end{align*}
Therefore if $x's'\in\tilde{A}_{t-t_0}$
\begin{align*}
&\chi_\lambda(gx)=
-2\sum_{\begin{smallmatrix}
\mu\in M_{\overline{q}}(\lambda_0)\\
\sigma(\mu)=1
\end{smallmatrix}}
a(\xi^\pm_{\lambda_0},\xi_\mu)\chi_\mu(x_0s_0)\overline{\chi}_{\lambda\backslash\{0\}}(x's')+\\
&\sum_{\begin{smallmatrix}
\mu\in M_{\overline{q}}(\lambda_0)\\
\sigma(\mu)=-1
\end{smallmatrix}}
\bigg{[}\bigg{(}a(\xi^+_{\lambda_0},\xi^+_\mu)\bigg{(}\frac{-1+i^\frac{p-1}{2}\sqrt{p}}{2}\bigg{)}+
a(\xi^+_{\lambda_0},\xi^-_\mu)\bigg{(}\frac{-1-i^\frac{p-1}{2}\sqrt{p}}{2}\bigg{)}\bigg{)}\\
&(\chi_\mu^+(x_0s_0)+\chi_\mu^-(x_0s_0))+\\
&\bigg{(}a(\xi^+_{\lambda_0},\xi^+_\mu)\bigg{(}\frac{-1-i^\frac{p-1}{2}\sqrt{p}}{2}\bigg{)}+
a(\xi^+_{\lambda_0},\xi^-_\mu)\bigg{(}\frac{-1+i^\frac{p-1}{2}\sqrt{p}}{2}\bigg{)}\bigg{)}\\
&(\chi_\mu^+(x_0s_0)+\chi_\mu^-(x_0s_0))\bigg{]}\overline{\chi}_{\lambda\backslash\{0\}}(x's')
\end{align*}
and so by Lemma~\ref{lem:t=2} we have that
\begin{align*}
\chi_\lambda(gx)=-&\sum_{\begin{smallmatrix}
\mu\in M_{\overline{q}}(\lambda_0)\\
\sigma(\mu)=1
\end{smallmatrix}}
a(\xi^\pm_{\lambda_0},\xi_\mu)(\chi_{(\mu,\lambda\backslash\{0\})}^+(x)+\chi_{(\mu,\lambda\backslash\{0\})}
^-(x))+\\
&\sum_{\begin{smallmatrix}
\mu\in M_{\overline{q}}(\lambda_0)\\
\sigma(\mu)=-1
\end{smallmatrix}}
\bigg{[}\bigg{(}a(\xi^+_{\lambda_0},\xi^+_\mu)\bigg{(}\frac{-1+i^\frac{p-1}{2}\sqrt{p}}{2}\bigg{)}+
a(\xi^+_{\lambda_0},\xi^-_\mu)\bigg{(}\frac{-1-i^\frac{p-1}{2}\sqrt{p}}{2}\bigg{)}\bigg{)}+\\
&\bigg{(}a(\xi^+_{\lambda_0},\xi^+_\mu)\bigg{(}\frac{-1-i^\frac{p-1}{2}\sqrt{p}}{2}\bigg{)}+
a(\xi^+_{\lambda_0},\xi^-_\mu)\bigg{(}\frac{-1+i^\frac{p-1}{2}\sqrt{p}}{2}\bigg{)}\bigg{)}\bigg{]}\chi_{(\mu,
\lambda\backslash\{0\})}(x)\\
=-\bigg{(}&\sum_{\begin{smallmatrix}
\mu\in M_{\overline{q}}(\lambda_0)\\
\sigma(\mu)=1
\end{smallmatrix}}
a(\xi^\pm_{\lambda_0},\xi_\mu)(\chi_{(\mu,\lambda\backslash\{0\})}^+(x)+\chi_{(\mu,\lambda\backslash\{0\})}^-(x))+\\
&\sum_{\begin{smallmatrix}
\mu\in M_{\overline{q}}(\lambda_0)\\
\sigma(\mu)=-1
\end{smallmatrix}}
(a(\xi^+_{\lambda_0},\xi^+_\mu)+a(\xi^+_{\lambda_0},\xi^-_\mu))\chi_{(\mu,\lambda\backslash\{0\})}(x)\bigg{)}.
\end{align*}
\end{enumerate}
\end{proof}

\begin{thm}\label{thm:MNj>0}
Let $q$ be an odd positive integer, $(1\leq l\leq(p-1)/2)$ and set $t_l'=\sum_{j=0}^{l-1}t_j$. Now let
$g:=(\boldsymbol{a};\tau)\in\tilde{N}_p^t\tilde{S}_t$ with $\boldsymbol{a}=(a,1,\dots,1)$, where
\begin{align*}
\overline{\zeta}^+_0(a)=\frac{-1+i^\frac{p-1}{2}\sqrt{p}}{2}
\end{align*}
and $\tau=[t_l'](o((1,\dots,q)))$. Let
\begin{align*}
x\in(\tilde{N}_p^{t_l'}\tilde{S}_{t_l'}).
(\tilde{N}_p^{t-t_l'-q}\tilde{S}_{t-t_l'-q}[t_l'+q]).
\end{align*}
We identify this subgroup with $(\tilde{N}_p^{t_l'}\tilde{S}_{t_l'}).
(\tilde{N}_p^{t-t_l'-q}\tilde{S}_{t-t_l'-q}[t_l'])$ by identifying
$\tilde{N}_p^{t-t_l'-q}\tilde{S}_{t-t_l'-q}[t_l'+q]$ with $\tilde{N}_p^{t-t_l'-q}\tilde{S}_{t-t_l'-q}[t_l']$ through
$\tilde{N}_p^{t-t_l'-q}\tilde{S}_{t-t_l'-q}$ via Lemma~\ref{lem:shift}. Now suppose
$gx\in\tilde{N}_p^t\tilde{S}_{\boldsymbol{t}}$.
\begin{enumerate}
\item If $\lambda\in\Delta_t^+$ then
\begin{align*}
\chi_\lambda(gx)=(-1)^\frac{q^2-1}{8}&\sum_{\mu\in M_q(\lambda_l)}(-1)^{L(c^{\lambda_l}_\mu)}
(\chi_{(\mu,\lambda\backslash\{l\})}^+(x)+\chi_{(\mu,\lambda\backslash\{l\})}^-(x)).
\end{align*}
\item If $\lambda\in\Delta_t^-$ then
\begin{align*}
\chi_\lambda^+(gx)=(-1)^\frac{q^2-1}{8}&\sum_{\mu\in M_q(\lambda_l)}(-1)^{L(c^{\lambda_l}_\mu)}
\chi_{(\mu,\lambda\backslash\{l\})}(x).
\end{align*}
\end{enumerate}
\end{thm}

\begin{proof}
Let
\begin{align*}
x_0\in\tilde{N}_p^{t_0-q}[q],&&s_0\in\tilde{S}_{t_0-q}[q],&&x'\in\tilde{N}_p^{t-t_0}[t_0],&&s'\in\tilde{S}_{(0,t_1,\dots
,t_{(p-1)/2})}[t_0],
\end{align*}
with $x=x_0s_0x's'$. As always we identify all the above subgroups with their non-shifted counterpart via
Lemma~\ref{lem:shift}. First let
$t-t_0$ be even. Then
\begin{align*}
\chi_{\lambda\backslash\{0\}}(gx's')=\chi_{\overline{\boldsymbol{t}(\lambda\backslash\{0\})}}(\boldsymbol{a}\tau
x's')\operatorname{Sym}_{\lambda\backslash\{0\}}(\theta_{t-t_0}(\tau)\theta_{t-t_0}(s')),
\end{align*}
which by the construction of $\chi_{\overline{\boldsymbol{t}(\lambda\backslash\{0\})}}$, Lemma~\ref{lem:qcycle}
and~\cite[Theorem 4.4]{pfeiff1994} this is equal to
\begin{align*}
=&(-1)^{\frac{q^2-1}{8}}\sum_{\mu\in M_q(\lambda_l)}
\bigg{(}\big{(}\chi^+_{\overline{\boldsymbol{t}((\mu,\lambda\backslash\{l\})\backslash\{0\})}}(x's')+
\chi^-_{\overline{\boldsymbol{t}((\mu,\lambda\backslash\{l\})\backslash\{0\})}}(x's')\big{)}\\
&(-1)^{L(c^{\lambda_l}_\mu)}\operatorname{Sym}_{((\mu,\lambda\backslash\{l\})\backslash\{0\})}(\theta_{t-t_0-q}
(s'))\bigg{)}\\
=&(-1)^{\frac{q^2-1}{8}}\sum_{\mu\in M_q(\lambda_l)}(-1)^{L(c^{\lambda_l}_\mu)}
\big{(}\chi^+_{((\mu,\lambda\backslash\{l\})\backslash\{0\})}(x's')+\chi^-_{((\mu,\lambda\backslash\{l\})\backslash\{0\}
)}(x's')\big{)}.
\end{align*}
Similarly if $t-t_0$ is odd then
\begin{align*}
\chi^\pm_{\lambda\backslash\{0\}}(gx's')=(-1)^{\frac{q^2-1}{8}}\sum_{\mu\in M_q(\lambda_l)}(-1)^{L(c^{\lambda_l}_\mu)}
\chi_{((\mu,\lambda\backslash\{l\})\backslash\{0\})}(x's').
\end{align*}
Also if $x's'\in\tilde{A}_{p(t-t_0-q)}$ we have that
\begin{align*}
\overline{\chi}^\pm_{\lambda\backslash\{0\}}(gx's')=(-1)^{\frac{q^2-1}{8}}\sum_{\mu\in
M_q(\lambda_l)}(-1)^{L(c^{\lambda_l}_\mu)}\overline{\chi}_{((\mu,\lambda\backslash\{l\})\backslash\{0\})}(x's').
\end{align*}
if $t-t_0$ is even and
\begin{align*}
\overline{\chi}_{\lambda\backslash\{0\}}(gx's')=(-1)^{\frac{q^2-1}{8}}\sum_{\mu\in
M_q(\lambda_l)}(-1)^{L(c^{\lambda_l}_\mu)}
(\overline{\chi}^+_{((\mu,\lambda\backslash\{l\})\backslash\{0\})}(x's')+\overline{\chi}^-_{((\mu,\lambda\backslash\{l\}
)\backslash\{0\})}(x's'))
\end{align*}
if $t-t_0$ is odd. The theorem now follows from Lemma~\ref{lem:t=2} and applying the relevant above equation.
\end{proof}

We now have a Murnaghan-Nakayama rule for $\tilde{N}_p^t\tilde{S}_t$.

\begin{thm}\label{thm:MN}
Let $q$ be an odd positive integer. Set $\tau=o((1,\dots,q))$ and $g:=(\boldsymbol{a};\tau)\in\tilde{N}_p^t\tilde{S}_t$,
where $\boldsymbol{a}=(a,1,\dots,1)$ and
\begin{align*}
\overline{\zeta}^+_0(a)=\frac{-1+i^\frac{p-1}{2}\sqrt{p}}{2}.
\end{align*}
Let $x\in\tilde{N}_p^{t-q}\tilde{S}_{t-q}[q]$ which we identify with $\tilde{N}_p^{t-q}\tilde{S}_{t-q}$ via
Lemma~\ref{lem:shift}.
\begin{enumerate}
\item If $\sigma(\lambda_0)=1$ and $t-t_0$ is even then
\begin{align*}
\chi^\lambda(gx)=-\bigg{(}&\sum_{\begin{smallmatrix}
\mu\in M_{\overline{q}}(\lambda_0)\\
\sigma(\mu)=1
\end{smallmatrix}}
a(\xi_{\lambda_0},\xi_\mu)\chi^{(\mu,\lambda\backslash\{0\})}(x)+\\
&\sum_{\begin{smallmatrix}
\mu\in M_{\overline{q}}(\lambda_0)\\
\sigma(\mu)=-1
\end{smallmatrix}}
a(\xi_{\lambda_0},\xi^\pm_\mu)(\chi^{(\mu,\lambda\backslash\{0\})+}(x)+\chi^{(\mu,\lambda\backslash\{0\})-}
(x))\bigg{)}+\\
(-1)^\frac{q^2-1}{8}\bigg{[}&\sum_{j=1}^\frac{p-1}{2}\sum_{\mu\in M_q(\lambda_j)}(-1)^{L(c^{\lambda_j}_\mu)}
(\chi^{(\mu,\lambda\backslash\{j\})+}(x)+\chi^{(\mu,\lambda\backslash\{j\})-}(x))\bigg{]}.
\end{align*}
\item If $\sigma(\lambda_0)=-1$ and $t-t_0$ is even then
\begin{align*}
&\chi^{\lambda+}(gx)
=-\sum_{\begin{smallmatrix}
\mu\in M_{\overline{q}}(\lambda_0)\\
\sigma(\mu)=1
\end{smallmatrix}}
a(\xi^\pm_{\lambda_0},\xi_\mu)\chi^{(\mu,\lambda\backslash\{0\})}(x)+\\
\sum_{\begin{smallmatrix}
\mu\in M_{\overline{q}}(\lambda_0)\\
\sigma(\mu)=-1
\end{smallmatrix}}
\bigg{[}&\bigg{(}a(\xi^+_{\lambda_0},\xi^+_\mu)\bigg{(}\frac{-1+i^\frac{p-1}{2}\sqrt{p}}{2}\bigg{)}+
a(\xi^+_{\lambda_0},\xi^-_\mu)\bigg{(}\frac{-1-i^\frac{p-1}{2}\sqrt{p}}{2}\bigg{)}\bigg{)}\chi^{(\mu,\lambda\backslash\{
0\})+}(x)+\\
&\bigg{(}a(\xi^+_{\lambda_0},\xi^+_\mu)\bigg{(}\frac{-1-i^\frac{p-1}{2}\sqrt{p}}{2}\bigg{)}+
a(\xi^+_{\lambda_0},\xi^-_\mu)\bigg{(}\frac{-1+i^\frac{p-1}{2}\sqrt{p}}{2}\bigg{)}\bigg{)}\chi^{(\mu,\lambda\backslash\{
0\})-}(x)\bigg{]}+\\
&(-1)^\frac{q^2-1}{8}\bigg{[}\sum_{j=1}^\frac{p-1}{2}\sum_{\mu\in M_q(\lambda_j)}(-1)^{L(c^{\lambda_j}_\mu)}
\chi^{(\mu,\lambda\backslash\{j\})}(x)\bigg{]}.
\end{align*}
\item If $\sigma(\lambda_0)=1$ and $t-t_0$ is odd then
\begin{align*}
&\chi^{\lambda+}(gx)=
-\sum_{\begin{smallmatrix}
\mu\in M_{\overline{q}}(\lambda_0)\\
\sigma(\mu)=-1
\end{smallmatrix}}
a(\overline{\xi}^\pm_{\lambda_0},\overline{\xi}_\mu)\chi^{(\mu,\lambda\backslash\{0\})}(x)+\\
\sum_{\begin{smallmatrix}
\mu\in M_{\overline{q}}(\lambda_0)\\
\sigma(\mu)=1
\end{smallmatrix}}
\bigg{[}&\bigg{(}a(\overline{\xi}^+_{\lambda_0},\overline{\xi}^+_\mu)\bigg{(}\frac{-1+i^\frac{p-1}{2}\sqrt{p}}{2}\bigg{)
}+a(\overline{\xi}^+_{\lambda_0},\overline{\xi}^-_\mu)
\bigg{(}\frac{-1-i^\frac{p-1}{2}\sqrt{p}}{2}\bigg{)}\bigg{)}\chi^{(\mu,\lambda\backslash\{0\})+}(x)+\\
&\bigg{(}a(\overline{\xi}^+_{\lambda_0},\overline{\xi}^+_\mu)\bigg{(}\frac{-1-i^\frac{p-1}{2}\sqrt{p}}{2}\bigg{)}
+a(\overline{\xi}^+_{\lambda_0},\overline{\xi}^-_\mu)
\bigg{(}\frac{-1+i^\frac{p-1}{2}\sqrt{p}}{2}\bigg{)}\bigg{)}\chi^{(\mu,\lambda\backslash\{0\})-}(x)\bigg{]}+\\
&(-1)^\frac{q^2-1}{8}\bigg{[}\sum_{j=1}^\frac{p-1}{2}\sum_{\mu\in M_q(\lambda_j)}(-1)^{L(c^{\lambda_j}_\mu)}
\chi^{(\mu,\lambda\backslash\{j\})}(x)\bigg{]}.
\end{align*}
\item If $\sigma(\lambda_0)=-1$ and $t-t_0$ is odd then
\begin{align*}
\chi^\lambda(gx)
=-\bigg{(}&\sum_{\begin{smallmatrix}
\mu\in M_{\overline{q}}(\lambda_0)\\
\sigma(\mu)=1
\end{smallmatrix}}
a(\xi^\pm_{\lambda_0},\xi_\mu)(\chi^{(\mu,\lambda\backslash\{0\})+}(x)+\chi^{(\mu,\lambda\backslash\{0\})-}
(x))+\\
&\sum_{\begin{smallmatrix}
\mu\in M_{\overline{q}}(\lambda_0)\\
\sigma(\mu)=-1
\end{smallmatrix}}
(a(\xi^+_{\lambda_0},\xi^+_\mu)+a(\xi^+_{\lambda_0},\xi^-_\mu))\chi^{(\mu,\lambda\backslash\{0\})}(x)\bigg{)}+\\
(-1)^\frac{q^2-1}{8}\bigg{[}\sum_{j=1}^\frac{p-1}{2}&\sum_{\mu\in M_q(\lambda_j)}(-1)^{L(c^{\lambda_j}_\mu)}
(\chi^{(\mu,\lambda\backslash\{j\})+}(x)+\chi^{(\mu,\lambda\backslash\{j\})-}(x))\bigg{]}.
\end{align*}
\end{enumerate}
\end{thm}

\begin{proof}
Apply the proof~\cite[Theorem 4.4]{pfeiff1994} using Theorems~\ref{thm:MNj=0} and~\ref{thm:MNj>0}.
\end{proof}

\section{The Brou\'{e} perfect isometry}\label{sec:main}

In this section we construct a Brou\'{e} perfect isometry between a block of $\tilde{S}_n$ with abelian defect and
weight $w$ and positive sign and the block of spin characters of $\tilde{N}_p^w\tilde{S}_w$. We also construct a
Brou\'{e} perfect isometry between the corresponding blocks of $\tilde{A}_n$ and
$\tilde{N}_p^w\tilde{S}_w\cap\tilde{A}_{pw}$. We will then use these isometries to construct the Brou\'{e} perfect
isometry of Theorem~\ref{thm:main}.
\newline
\newline
Let $n$ be a positive integer. We denote by $C$ the set of elements of $\tilde{S}_n$ in conjugacy classes labelled by
partitions of $n$ with no part an odd multiple of $p$. We define $S_{S_n}$
to be the set of elements of $S_n$ all of whose cycles have length $1$ or an odd multiple of $p$ and
$S:=\{o(g)|g\in S_{S_n}\}$.
\newline
\newline
Let $C'$ be the set of conjugacy classes of $\tilde{N}^w_p\tilde{S}_w$ labelled by $(\pi_0,\dots,\pi_{p-1})$ with
$\pi_0$ having only even parts. We define $S_{N_p\wr S_w}'$ to be the set of elements of
$N_p\wr S_w$ in conjugacy classes labelled by $(\pi_0,\varnothing,\dots,\varnothing,\pi_{p-1})$, where
$\pi_0$ has only odd parts and $\pi_{p-1}=(1,\dots,1)$. Note that if $g\in S_{N_p\wr S_w}'$ then $g$
has odd order and so $o(g)$ makes sense if we view $\tilde{N}^w_p\tilde{S}_w$ as a subgroup of $\tilde{S}_{pw}$. We
now set $S':=\{o(g)|g\in S_{N_p\wr S_w}'\}$.
\newline
\newline
Now suppose $\beta=(\beta_1,\beta_2,\dots)$ has only odd parts and $p|\beta|\leq n$. Then $\beta$ labels a conjugacy
class of $S$. We pick a representative $s_\beta$ for this conjugacy class to be
\begin{align*}
o((1,\dots,p\beta_1)(p\beta_1+1,\dots,p(\beta_1+\beta_2))\dots).
\end{align*}
If in addition $|\beta|\leq w$ then $\beta$ labels a conjugacy class of $S'$. We pick a representative $s_\beta'$ for
this conjugacy class to be
\begin{align*}
\prod_l((a,1,\dots,1);\tau_l),
\end{align*}
where
\begin{align*}
\overline{\zeta}^+_0(a)=\frac{-1+i^\frac{p-1}{2}\sqrt{p}}{2}&&\text{ and }&&\tau_l=
o\bigg{(}\bigg{(}\bigg{(}\sum_{j=1}^{l-1}\beta_j\bigg{)}+1,\dots,\sum_{j=1}^l\beta_j\bigg{)}\bigg{)}\text{ for all }l.
\end{align*}
Let $\gamma\in\mathcal{D}_{n-p|w|}^+$ for some $w\in\mathbb{N}_0$ be a $p-$bar core. In particular $\gamma$ labels a
$p-$block of $\tilde{S}_n$ of weight $w$. Let $E_\gamma$ denote the set of strict partitions (not necessarily of $n$)
with $p-$bar core $\gamma$. We define the bijection
\begin{align*}
\Psi:E_\gamma&\to\sqcup_t\Delta_t\\
\lambda&\mapsto\lambda^{(\overline{p})}.
\end{align*}
Note that adding any $p-$bar hook to a partition $\lambda$ changes the sign of $\sigma(\lambda)$ unless the hook is a
part equal to $p$. Then as $\sigma(\gamma)=1$ by induction we see that $\sigma(\lambda)=\sigma(\Psi(\lambda))$.

\begin{thm}\label{thm:mainS}
When $w<p$ then $I$ is a Brou\'{e} perfect isometry, where $I$ is defined as follows:
\begin{align*}
I:\mathbb{Z}\operatorname{Irr}(\tilde{S}_{n,\gamma})&\to\mathbb{Z}\operatorname{IrrSp}(\tilde{N}_p^w\tilde{S}_w)&\\
\xi_\lambda&\mapsto\delta_{\overline{p}}(\lambda)\delta_{\overline{p}}(\Psi(\lambda))
(-1)^{w\frac{p^2-1}{8}+|\Psi(\lambda)_0|}\chi^{\Psi(\lambda)}&&\text{if }\sigma(\lambda)=1,\\
\xi_\lambda^\epsilon&\mapsto\delta_{\overline{p}}(\lambda)\delta_{\overline{p}}(\Psi(\lambda))
(-1)^{w\frac{p^2-1}{8}+|\Psi(\lambda)_0|}\chi^{\Psi(\lambda)\eta_{\overline{p}}(\lambda)}&&\text{if
}\sigma(\lambda)=-1,
\end{align*}
where,
\begin{align*}
\eta_{\overline{p}}(\lambda)=
\begin{cases}
\eta\delta_{\overline{p}}(\lambda)\delta_{\overline{p}}(\Psi(\lambda))
(-1)^{w+\frac{p-1}{4}(|\Psi(\lambda)_0|-l(\Psi(\lambda)_0))}&\text{if }\sigma(\Psi(\lambda)_0)=1,\\
\eta\delta_{\overline{p}}(\lambda)\delta_{\overline{p}}(\Psi(\lambda))
(-1)^{w+\frac{p-1}{4}(|\Psi(\lambda)_0|-l(\Psi(\lambda)_0)+1)}&\text{if }\sigma(\Psi(\lambda)_0)=-1.
\end{cases}
\end{align*}
\end{thm}

\begin{proof}
We set $G:=\tilde{S}_n$ and $G':=\tilde{N}^w_p\tilde{S}_w$. First we note that by~\cite[Lemma 4.11]{brugra2014}
$\tilde{S}_{n,\gamma}$ is a union of $C-$blocks of $G$. Next we note that if $\chi$
(respectively $\psi$) is an irreducible non-spin (respectively spin) character of $G'$ then
$\langle\operatorname{res}_{C'}(\chi),\operatorname{res}_{C'}(\psi)\rangle_{G'}=0$ and so
$\operatorname{IrrSp}(\tilde{N}_p^w\tilde{S}_w)$ is a union of $C'-$blocks of $G$.
\newline
\newline
By~\cite[Proposition 4.12]{brugra2014}
$G$ has an MN-structure with respect to $C$ and $\tilde{S}_{n,\gamma}$. Now let $x\in G'$ have disjoint cycle
decomposition $\prod_l(x_l;\tau_l)$. We partition the product $x=\prod_{l\in L_0}x_l.\prod_{l\in L_1}x_l$ with $l\in
L_0$
if and only if $f(\theta_{pw}(x_l))$ is conjugate to $y_0$ in $N_p$.
Then, by multiplying both products by $z$ if necessary, we can assume $\prod_{l\in L_0}x_l\in S'$ and $\prod_{l\in
L_1}x_l\in C'$. We set
\begin{align*}
G_{s_\beta'}':=\tilde{N}^{w-|\beta|}_p\tilde{S}_{w-|\beta|}[|\beta|].
\end{align*}
Then by iteratively applying
Theorem~\ref{thm:MN} we have an MN-structure for $G'$
with respect to $C'$ and $\operatorname{IrrSp}(\tilde{N}^w_p\tilde{S}_w)$. Therefore part (1) of
Theorem~\ref{thm:mainJ} holds for $G$ and $G'$.
\newline
\newline
By an abuse of notation we allow $I$ to denote the map in the statement of the theorem for varying $w$. Let $\beta$ be a
partition having only odd parts and $p|\beta|\leq n$. We first note that if $|\beta|>w$ then by induction and
Theorem~\ref{thm:MNSn} $r^\beta=0$. Let's now assume $|\beta|\leq w$. We wish to show that $I\circ r^\beta=r'^\beta\circ
I$. By induction we can assume $l(\beta)=1$. Let $\beta=(q)$ for some odd positive integer $q$ and $\lambda\vdash n$.
\newline
\newline
First suppose $\sigma(\Psi(\lambda)_0)=1$ and $w-|\Psi(\lambda)_0|$ is even. If $\mu\in M_{\overline{pq}}(\lambda)$ with
$\sigma(\mu)=1$ then $\mu$ must be gotten
from $\lambda$ by removing a part of $\lambda$ equal to $pq$. By Theorem~\ref{thm:MNSn} $\xi_\mu$ appears in
$r^\beta(\xi_\lambda)$ with
coefficient $a(\xi_\lambda,\xi_\mu)$ and by Theorem~\ref{thm:MN} $\chi^{\Psi(\mu)}$ appears in
$r'^\beta(\chi^{\Psi(\lambda)})$ with coefficient $-a(\xi_{\Psi(\lambda)_0},\xi_{\Psi(\mu)_0})$. Then by
Lemma~\ref{lem:leglength} and the fact that
\begin{align*}
(-1)^{\frac{p^2-1}{8}}(-1)^{\frac{q^2-1}{8}}=(-1)^{\frac{(pq)^2-1}{8}},
\end{align*}
the ratio of these two coefficients is
\begin{align*}
-(-1)^{\frac{p^2-1}{8}}\delta_{\overline{p}}(\lambda)\delta_{\overline{p}}(\mu)
\delta_{\overline{p}}(\Psi(\lambda))\delta_{\overline{p}}(\Psi(\mu)).
\end{align*}
If $\sigma(\mu)=-1$ then by Theorem~\ref{thm:MNSn} $\xi^+_\mu$ and $\xi^-_\mu$ appear in $r^\beta(\xi_\lambda)$ with
the same coefficient $a(\xi_\lambda,\xi_\mu^+)=a(\xi_\lambda,\xi_\mu^-)$. If $\sigma(\Psi(\mu)_0)=-1$ and
$\Psi(\mu)_j=\Psi(\lambda)_j$ for all $j>0$ then
by Theorem~\ref{thm:MN} $\chi^{\Psi(\mu)+}$ and $\chi^{\Psi(\mu)-}$ appear in
$r'^\beta(\chi^{\Psi(\lambda)})$ with the same coefficient
$-a(\xi_{\Psi(\lambda)_0},\xi^+_{\Psi(\mu)_0})=-a(\xi_{\Psi(\lambda)_0},\xi^-_{\Psi(\mu)_0})$.
Again the ratio of these two coefficients is
\begin{align*}
-(-1)^{\frac{p^2-1}{8}}\delta_{\overline{p}}(\lambda)\delta_{\overline{p}}(\mu)
\delta_{\overline{p}}(\Psi(\lambda))\delta_{\overline{p}}(\Psi(\mu)).
\end{align*}
If $\Psi(\mu)_l\neq\Psi(\lambda)_l$ and $\Psi(\mu)_j=\Psi(\lambda)_j$ for all $j\neq l>0$ then by Theorem~\ref{thm:MN}
$\chi^{\Psi(\mu)+}$ and
$\chi^{\Psi(\mu)-}$ appear in $r'^\beta(\chi^{\Psi(\lambda)})$ with the same coefficient
$-a(\xi_{\Psi(\lambda)_0},\xi^+_{\Psi(\mu)_0})=-a(\xi_{\Psi(\lambda)_0},\xi^-_{\Psi(\mu)_0})$.
The ratio of these two coefficients is
\begin{align*}
(-1)^{\frac{p^2-1}{8}}\delta_{\overline{p}}(\lambda)\delta_{\overline{p}}(\mu)
\delta_{\overline{p}}(\Psi(\lambda))\delta_{\overline{p}}(\Psi(\mu))
\end{align*}
and hence we have proved $I\circ r^\beta(\xi_\lambda)=r'^\beta\circ I(\xi_\lambda)$.
\newline
\newline
All the other cases are similar. We prove the most complicated case that is that the multiplicity of $\chi^{\Psi(\mu)+}$
in $I\circ r^\beta(\xi_\lambda^+)$ is equal to its multiplicity in $r'^\beta\circ I(\xi_\lambda^+)$ where
$\sigma(\lambda)=\sigma(\mu)=-1$. In this case $\mu$ must be gotten from $\lambda$ by removing a part of $\lambda$ equal
to $pq$. By Thoeroem~\ref{thm:MNSn} $\xi_\mu^\pm$ appears in $r^\beta(\xi_\lambda^+)$ with coefficient
\begin{align*}
a(\xi^+_\lambda,\xi^\pm_\mu)=
\frac{1}{2}(-1)^{\frac{(pq)^2-1}{8}}(\delta_{\overline{p}}(\lambda)\delta_{\overline{p}}(\mu)\pm
i^{\frac{pq-1}{2}}\sqrt{pq}).
\end{align*}
By Thoeroem~\ref{thm:MN} $\chi^{\Psi(\mu)+}$ appears in $r'^\beta(\chi^{\Psi(\lambda)\pm})$ with coefficient
\begin{align*}
a(\xi^+_{\Psi(\lambda)_0},\xi^+_{\Psi(\mu)_0})\bigg{(}\frac{-1\pm i^\frac{p-1}{2}\sqrt{p}}{2}\bigg{)}+
a(\xi^+_{\Psi(\lambda)_0},\xi^-_{\Psi(\mu)_0})\bigg{(}\frac{-1\mp i^\frac{p-1}{2}\sqrt{p}}{2}\bigg{)}
\end{align*}
if $\sigma(\Psi(\lambda)_0)=-1$ and $w-|\Psi(\lambda)_0|$ is even or
\begin{align*}
a(\overline{\xi}^+_{\Psi(\lambda)_0},\overline{\xi}^+_{\Psi(\mu)_0})\bigg{(}\frac{-1\pm
i^\frac{p-1}{2}\sqrt{p}}{2}\bigg{)}+
a(\overline{\xi}^+_{\Psi(\lambda)_0},\overline{\xi}^-_{\Psi(\mu)_0})\bigg{(}\frac{-1\mp
i^\frac{p-1}{2}\sqrt{p}}{2}\bigg{)}
\end{align*}
if $\sigma(\Psi(\lambda)_0)=1$ and $w-|\Psi(\lambda)_0|$ is odd. In either case this coefficient is
\begin{align*}
&\frac{1}{2}(-1)^{\frac{q^2-1}{8}}(\alpha^{\Psi(\lambda)_0}_{\Psi(\mu)_0}+i^{\frac{q-1}{2}}\sqrt{q})
\bigg{(}\frac{-1\pm i^\frac{p-1}{2}\sqrt{p}}{2}\bigg{)}+
\frac{1}{2}(-1)^{\frac{q^2-1}{8}}(\alpha^{\Psi(\lambda)_0}_{\Psi(\mu)_0}-i^{\frac{q-1}{2}}\sqrt{q})
\bigg{(}\frac{-1\mp i^\frac{p-1}{2}\sqrt{p}}{2}\bigg{)}\\
=&\frac{1}{2}(-1)^{\frac{q^2-1}{8}}(-\delta_{\overline{p}}(\Psi(\lambda))\delta_{\overline{p}}(\Psi(\mu))
\pm i^{\frac{p+q-2}{2}}\sqrt{pq}).
\end{align*}
Now since
\begin{align*}
(-1)^{\frac{(p-1)(q-1)}{4}}i^{\frac{p+q-2}{2}}=i^{\frac{pq-1}{2}},
\end{align*}
the coefficient $a(\xi^+_\lambda,\xi^\pm_\mu)$ is
\begin{align*}
&\frac{1}{2}(-1)^{\frac{(pq)^2-1}{8}}(\delta_{\overline{p}}(\lambda)\delta_{\overline{p}}(\mu)\pm
i^{\frac{pq-1}{2}}\sqrt{pq})\\
=-&\frac{1}{2}(-1)^{\frac{p^2-1}{8}}(-1)^{\frac{q^2-1}{8}}\delta_{\overline{p}}(\lambda)\delta_{\overline{p}}(\mu)
\delta_{\overline{p}}(\Psi(\lambda))\delta_{\overline{p}}(\Psi(\mu))\\
&(-\delta_{\overline{p}}(\Psi(\lambda))\delta_{\overline{p}}(\Psi(\mu))\mp
\delta_{\overline{p}}(\lambda)\delta_{\overline{p}}(\mu)
\delta_{\overline{p}}(\Psi(\lambda))\delta_{\overline{p}}(\Psi(\mu))
(-1)^{\frac{(p-1)(q-1)}{4}}i^{\frac{p+q-2}{2}}\sqrt{pq}).
\end{align*}
Then since
\begin{align*}
(|\Psi(\lambda)_0|-l(\Psi(\lambda)_0))-(|\Psi(\mu)_0|-l(\Psi(\mu)_0))=q-1
\end{align*}
we have that the coefficient
of $\chi^{\Psi(\mu)+}$ in $I\circ r^\beta(\xi_\lambda^+)$ is equal to its coefficient in $r'^\beta\circ
I(\xi_\lambda^+)$.
We have now proved that part (2) of Theorem~\ref{thm:mainJ} holds for $G$ and $G'$. Furthermore, by
Remark~\ref{rem:part3}
part (3) also holds.
\newline
\newline
Let $x\in G$ and $x'\in G'$. We want to verify property (1) of a Brou\'{e} perfect isometry. First suppose
$x\in\tilde{A}_n$ and let
$\Phi\in\mathbb{Z}b^\vee$ where $b$ is a basis of $\mathbb{Z}\operatorname{Irr}(B_{x_S})^{C\cap G_{x_S}}$. By
Lemma~\ref{lem:vee} applied to $C$ and Theorem~\ref{thm:values}, $\Phi(x)$ is non-zero only if $x$ is $p-$regular.
So again by Lemma~\ref{lem:vee} applied to the
set of $p-$regular elements we have that $\Phi\downarrow_{G_{x_S}\cap\tilde{A}_n}$ is a projective character. Then by
Lemma~\ref{lem:virprocha} we have that $\hat{I}(x,x')\in|C_G(x)|\mathcal{R}$. Similarly if
$x'\in\tilde{A}_{pw}$ then $\hat{I}(x,x')\in|C_{G'}(x')|\mathcal{R}$. Note that in this final
case Lemma~\ref{lem:oneeven} takes the place of Theorem~\ref{thm:values}.
\newline
\newline
Now assume $x\notin\tilde{A}_n$ has cycle type $\pi$. Then, by Theorem~\ref{thm:values},
$\xi_\lambda^{(\pm)}(x)\neq0$ only if $\pi=\lambda$. Therefore
\begin{align*}
\hat{I}(x,x')=\pm(\xi_\lambda^+(x)\chi^{\Psi(\lambda)+}(x')+\xi_\lambda^-(x)\chi^{\Psi(\lambda)-}(x'))
\end{align*}
and so we may assume $x'\notin\tilde{A}_{pw}$. Then by Theorem~\ref{thm:values} and
Lemma~\ref{lem:sqrtR} $\hat{I}(x,x')\in p^{l(\Psi(\lambda)_0)}\mathcal{R}$. Then by Lemma~\ref{lem:dbcent}
$\hat{I}(x,x')\in|C_G(x)|\mathcal{R}$. Similarly if $x'\notin\tilde{A}_{pw}$ is of type $(\pi_0,\dots,\pi_{(p-1)/2})$
then we may assume $x$ is of type $\lambda$, where $\sigma(\lambda)=-1$ and
\begin{align*}
\hat{I}(x,x')=\pm(\xi_\lambda^+(x)\chi^{\Psi(\lambda)+}(x')+\xi_\lambda^-(x)\chi^{\Psi(\lambda)-}(x')).
\end{align*}
Then by Theorem~\ref{thm:values} and Lemma~\ref{lem:sqrtR} either $\chi^{\Psi(\lambda)\pm}(x')\neq0$ or
$l(\pi_0)=l(\Psi(\lambda)_0)$ and $l(\pi_{p-1})$ and $\hat{I}(x,x')\in p^{l(\Psi(\lambda)_0)}\mathcal{R}$ and so, by
Lemma~\ref{lem:ordcent}, $\hat{I}(x,x')\in|C_{G'}(x')|\mathcal{R}$.
\newline
\newline
Now we look at property (2) of a Brou\'{e} perfect isometry. If $x$ is $p-$singular and $x'$ is
$p-$regular then by Lemma~\ref{lem:gen} $\hat{I}(x,x')=0$ if $x\notin C$ so let's assume $x\in C$. Suppose $x$ has
cycle type $\pi$ for some $\pi\in\mathcal{P}_n$ and then since $x\notin C$, $\pi$ must have some part divisible by
$2p$ and so $\pi\notin\mathcal{O}_n$. If $\pi\notin\mathcal{D}^-_n$ then by Theorem~\ref{thm:values}
$\xi^{(\pm)}_\lambda(x)=0$ for all $\lambda\in\mathcal{P}_n$ and hence $\hat{I}(x,x')=0$ so let's assume
$\pi\in\mathcal{D}^-_n$. Again by Theorem~\ref{thm:values} we have that $\xi_\lambda^{(\pm)}(x)\neq0$ if and only if
$\lambda=\pi$. Therefore $|\Psi(\lambda)_0|>0$ and so by~\ref{lem:>0=}
$\chi^{\Psi(\lambda)+}(x')=\chi^{\Psi(\lambda)-}(x')$ and hence $\hat{I}(x,x')=0$.
\newline
\newline
Next suppose $x$ is $p-$regular and $x'$ is $p-$singular. By Lemma~\ref{lem:gen} $\hat{I}(x,x')=0$ if $x'\notin C'$
so let's assume $x'\in C'$. Let $x'$ be of type $\pi$. Then $\pi_0$ must have some even part and so by
Lemmas~\ref{lem:oneeven} and~\ref{lem:sqrtR} $\chi^{\lambda\pm}(x)\neq0$ only if $x\notin\tilde{A}_{pt}$ and
$|\lambda_0|>0$. Then $\Psi^{-1}(\lambda)$ must have some part divisible by $p$ and so by Theorem~\ref{thm:values}
$\xi_{\Psi^{-1}(\lambda)}^+(x)=\xi_{\Psi^{-1}(\lambda)}^-(x)$ and hence $\hat{I}(x,x')=0$.
\end{proof}

We now show that we have a Brou\'{e} perfect isometry between the corresponding block of $\tilde{A}_n$ and
$\tilde{N}_p^w\tilde{S}_w\cap\tilde{A}_{pw}$.

\begin{thm}\label{thm:mainA}
The isometry
\begin{align*}
I_A:\mathbb{Z}\operatorname{Irr}(\tilde{A}_{n,\gamma})&
\to\mathbb{Z}\operatorname{IrrSp}(\tilde{N}_p^w\tilde{S}_w\cap\tilde{A}_{pw})\\
\overline{\xi}_\lambda&\mapsto\delta_{\overline{p}}(\lambda)\delta_{\overline{p}}(\Psi(\lambda))
(-1)^{w\frac{p^2-1}{8}+|\Psi(\lambda)_0|}\overline{\chi}^{\Psi(\lambda)}&&\text{if }\sigma(\lambda)=-1,\\
\overline{\xi}_\lambda^\eta&\mapsto\delta_{\overline{p}}(\lambda)\delta_{\overline{p}}(\Psi(\lambda))
(-1)^{w\frac{p^2-1}{8}+|\Psi(\lambda)_0|}\overline{\chi}^{\Psi(\lambda)\eta}&&\text{if }\sigma(\lambda)=1,
\end{align*}
is a Brou\'{e} perfect isometry.
\end{thm}

\begin{proof}
Set $G:=\tilde{A}_n$ and $G':=\tilde{N}_p^w\tilde{S}_w\cap\tilde{A}_{pw}$. We prove $I_A$ is a Brou\'{e} perfect
isometry
by studying $\hat{I}_A-\frac{1}{2}\hat{I}$ where $I$ is taken from Theorem~\ref{thm:mainS}.
\newline
\newline
Let $x\in\tilde{A}_n$ and $x'\in\tilde{N}_p^w\tilde{S}_w\cap\tilde{A}_{pw}$. Note that if $\lambda\in\mathcal{D}_n^-$
then
\begin{align*}
\overline{\xi}_\lambda(x)\overline{\chi}^{\Psi(\lambda)}(x')-
\frac{1}{2}(\xi^+_\lambda(x)\chi^{\Psi(\lambda)\pm}(x')+\xi^-_\lambda(x)\chi^{\Psi(\lambda)\mp}(x'))=0.
\end{align*}
Now if $\lambda\in\mathcal{D}_n^+$ and $x$ is not of type $\lambda$ then by Theorem~\ref{thm:values}
\begin{align*}
\overline{\xi}^+_\lambda(x)\overline{\chi}^{\Psi(\lambda)\pm}(x')+
\overline{\xi}^-_\lambda(x)\overline{\chi}^{\Psi(\lambda)\mp}(x')-\frac{1}{2}\xi_\lambda(x)\chi^{\Psi(\lambda)}(x')=0.
\end{align*}
Therefore $(\hat{I}_A-\frac{1}{2}\hat{I})(x,x')=0$ unless $x$ is of type $\lambda\in\mathcal{D}_n^+$ and in this case
\begin{align*}
(\hat{I}_A-\frac{1}{2}\hat{I})(x,x')&=\pm\big{(}\overline{\xi}^+_\lambda(x)\overline{\chi}^{\Psi(\lambda)\pm}(x')+
\overline{\xi}^-_\lambda(x)\overline{\chi}^{\Psi(\lambda)\mp}(x')-\frac{1}{2}\xi_\lambda(x)\chi^{\Psi(\lambda)}(x')\big{
)}\\
&=\pm\frac{1}{2}(\overline{\xi}^+_\lambda(x)-\overline{\xi}^-_\lambda(x))
(\overline{\chi}^{\Psi(\lambda)\pm}(x')-\overline{\chi}^{\Psi(\lambda)\mp}(x')).
\end{align*}
Let $x'$ be of type $\pi$. By Lemma~\ref{lem:sqrtAR}, $(\hat{I}_A-\frac{1}{2}\hat{I})(x,x')=0$ unless
$l(\pi_0)=l(\Psi(\lambda)_0)$ and $l(\pi_{p-1})=0$ and in this case, by Theorem~\ref{thm:values} and
Lemma~\ref{lem:sqrtAR}, $(\hat{I}_A-\frac{1}{2}\hat{I})(x,x')\in p^{l(\Psi(\lambda)_0)}\mathcal{R}$. Then by
Lemmas~\ref{lem:dbcent} and~\ref{lem:ordcent},
\begin{align*}
(\hat{I}_A-\frac{1}{2}\hat{I})(x,x')\in|C_G(x)|\mathcal{R}\cap|C_{G'}(x')|\mathcal{R}
\end{align*}
and so $\hat{I}_A$ satisfies property (1) of a Brou\'{e} perfect isometry.
\newline
\newline
Suppose $x$ is $p-$regular and $x'$ is $p-$singular. As above we can assume $x$ is of type $\lambda\in\mathcal{D}_n^+$.
Then
$|\Psi(\lambda)_0|=0$ and so by Lemma~\ref{lem:lambda0=}, $(\hat{I}_A-\frac{1}{2}\hat{I})(x,x')=0$. Now
suppose $x$ is $p-$singular and $x'$ is $p-$regular. Again we can assume $x$ is of type $\lambda\in\mathcal{D}_n^+$.
Then
$|\Psi(\lambda)_0|>0$ and so by Lemma~\ref{lem:lambda0p} $(\hat{I}_A-\frac{1}{2}\hat{I})(x,x')=0$.
\end{proof}

Before we prove Theorem~\ref{thm:main} we need the following well-known lemma.

\begin{lem}\label{lem:mornor}
Let $G$ be a finite group, $H$ a normal subgroup and $\mathcal{R}He$ a block of $\mathcal{R}H$ such that $N_G(e)=H$.
Then $f=\sum_{g\in G/H}geg^{-1}$ is a block idempotent of $\mathcal{R}G$ and $f\mathcal{R}Ge$ induces a Morita
equivalence between $\mathcal{R}He$ and $\mathcal{R}Gf$.
\end{lem}

We are now in a position to prove Theorem~\ref{thm:main} which we first restate.

\begin{thm*}
Let $p$ be an odd prime, $n$ a positive integer and $B$ a $p-$block of $\tilde{S}_n$ or $\tilde{A}_n$ with abelian
defect
group. Then there exists a Brou\'{e} perfect isometry between $B$ and its Brauer correspondent.
\end{thm*}

\begin{proof}
First suppose $w\geq p$. Then there is a $p-$subgroup of $S_{pw}$ isomorphic to $C_p^p\rtimes C_p$ where $C_p$ acts on
$C_p^p$ by cyclically permuting factors. This subgroup is not abelian therefore the Sylow $p-$subgroups of $S_{pw}$ and
hence $\tilde{S}_{pw}$ are not abelian. From now on we assume $0<w<p$ as the result is clearly true when $w=0$. Then by
comparing orders we see that a Sylow $p-$subgroup $Q$ of $S_{pw}$ is isomorphic to $C_p^w$ and
\begin{align*}
P:=\{o(g)|g\in Q\}\cong C_p^w
\end{align*}
is a Sylow $p-$subgroup of $\tilde{S}_{pw}$. Now $N_{S_{pw}}(Q)\cong N_p\wr S_w$ and so, as every element of $P$ has odd
order, $N_{\tilde{S}_{pw}}(P)\cong\tilde{N}_p^w\tilde{S}_w$. Now suppose $\gamma\vdash(n-pw)$. Then $P$ is a defect group
for $\tilde{S}_{n,\gamma}$ and
\begin{align*}
N_{\tilde{S}_n}(P)\cong\tilde{S}_{n-pw}(\tilde{N}_p^w\tilde{S}_w[n-pw]).
\end{align*}
First suppose $\sigma(\gamma)=1$. Then we have the following Morita equivalences
\begin{align*}
&\mathcal{R}\tilde{S}_{n-pw}(\tilde{N}_p^w\tilde{S}_w[n-pw])e_{n-pw,\gamma}\\
\sim_{\operatorname{Mor}}&\mathcal{R}\tilde{A}_{n-pw}(\tilde{N}_p^w\tilde{S}_w[n-pw])\overline{e}^+_{n-pw,\gamma}\\
\cong&\mathcal{R}\tilde{A}_{n-pw}\overline{e}^+_{n-pw,\gamma}\otimes_{\mathcal{R}}
\mathcal{R}\tilde{N}_p^w\tilde{S}_w\bigg{(}\frac{1-z}{2}\bigg{)}\\
\sim_{\operatorname{Mor}}&\mathcal{R}\tilde{N}_p^w\tilde{S}_w\bigg{(}\frac{1-z}{2}\bigg{)},
\end{align*}
where the first equivalence is given by Lemma~\ref{lem:mornor}. Similarly,
\begin{align*}
&\mathcal{R}((\tilde{S}_{n-pw}(\tilde{N}_p^w\tilde{S}_w[n-pw]))\cap\tilde{A}_n)(\overline{e}^+_{n-pw,\gamma}+\overline{e}^-_{n-pw,\gamma})\\
\sim_{\operatorname{Mor}}&\mathcal{R}\tilde{A}_{n-pw}((\tilde{N}_p^w\tilde{S}_w\cap\tilde{A}_{pw})[n-pw])\overline{e}^+_{n-pw,\gamma}\\
\cong&\mathcal{R}\tilde{A}_{n-pw}\overline{e}^+_{n-pw,\gamma}\otimes_{\mathcal{R}}
\mathcal{R}(\tilde{N}_p^w\tilde{S}_w\cap\tilde{A}_{pw})\bigg{(}\frac{1-z}{2}\bigg{)}\\
\sim_{\operatorname{Mor}}&\mathcal{R}(\tilde{N}_p^w\tilde{S}_w\cap\tilde{A}_{pw})\bigg{(}\frac{1-z}{2}\bigg{)}.
\end{align*}
When $\sigma(\gamma)=-1$ we also have the Morita equivalences
\begin{align*}
&\mathcal{R}\tilde{S}_{n-pw}(\tilde{N}_p^w\tilde{S}_w[n-pw])(e^+_{n-pw,\gamma}+e^-_{n-pw,\gamma})\\
\sim_{\operatorname{Mor}}&\mathcal{R}\tilde{S}_{n-pw}((\tilde{N}_p^w\tilde{S}_w\cap\tilde{A}_{pw})[n-pw])e^+_{n-pw,\gamma}\\
\cong&\mathcal{R}\tilde{S}_{n-pw}e^+_{n-pw,\gamma}\otimes_{\mathcal{R}}
\mathcal{R}(\tilde{N}_p^w\tilde{S}_w\cap\tilde{A}_{pw})\bigg{(}\frac{1-z}{2}\bigg{)}\\
\sim_{\operatorname{Mor}}&\mathcal{R}(\tilde{N}_p^w\tilde{S}_w\cap\tilde{A}_{pw})\bigg{(}\frac{1-z}{2}\bigg{)},
\end{align*}
and
\begin{align*}
&\mathcal{R}((\tilde{S}_{n-pw}(\tilde{N}_p^w\tilde{S}_w[n-pw]))\cap\tilde{A}_n)\overline{e}_{n-pw,\gamma}\\
\cong&\mathcal{R}\tilde{A}_{n-pw}(\tilde{N}_p^w\tilde{S}_w[n-pw])\overline{e}_{n-pw,\gamma}\\
\cong&\mathcal{R}\tilde{A}_{n-pw}\overline{e}_{n-pw,\gamma}\otimes_{\mathcal{R}}
\mathcal{R}\tilde{N}_p^w\tilde{S}_w\bigg{(}\frac{1-z}{2}\bigg{)}\\
\sim_{\operatorname{Mor}}&\mathcal{R}\tilde{N}_p^w\tilde{S}_w\bigg{(}\frac{1-z}{2}\bigg{)},
\end{align*}
where the first isomorphism is given by
\begin{align*}
\mathcal{R}\tilde{A}_{n-pw}(\tilde{N}_p^w\tilde{S}_w[n-pw])\overline{e}&\to
\mathcal{R}((\tilde{S}_{n-pw}(\tilde{N}_p^w\tilde{S}_w[n-pw]))\cap\tilde{A}_n)\overline{e}\\
x\overline{e}&\mapsto
\begin{cases}
x\overline{e}&\text{if }x\in\tilde{A}_{n-pw}((\tilde{N}_p^w\tilde{S}_w\cap\tilde{A}_{pw})[n-pw]),\\
ix(e^+-e^-)&\text{if }x\in(\tilde{N}_p^w\tilde{S}_w\backslash\tilde{A}_{pw})[n-pw],
\end{cases}
\end{align*}
where $\overline{e}:=\overline{e}_{n-pw,\gamma}$ and $e^\pm:=e^+-e^-$.
\newline
\newline
Any Morita equivalence gives rise to a Brou\'{e} perfect isometry with all signs positive (see~\cite[1.2]{broue1990}).
The result now follows from~\cite[Theorem 4.15]{brugra2014} and Theorems~\ref{thm:mainS} and~\ref{thm:mainA}.
\end{proof}

\begin{ack*}
The author gratefully, acknowledges financial support by ERC Advanced Grant $291512$.
\end{ack*}

\nocite{brugra2014, stembr1989, micols1990, olsson1993, schur1911, cabane1988, stembr1989, jamker1981, pfeiff1994,
chukes2002, churou2008, humphr1986, kessar1996, broue1990}

\newpage
\bibliographystyle{plain}
\bibliography{biblio}

\end{document}